\documentclass[a4paper,fleqn,leqno]{article}

\usepackage{latexsym,amssymb,amsthm,amsmath,mathtools,epsfig,color,epsf}
\usepackage[english]{babel}
\usepackage[normalem]{ulem}
\usepackage{hyperref}
\usepackage[utf8]{inputenc}
\usepackage{graphicx,tikz,pgfplots,pgfplotstable}
\usepackage{ifthen,tkz-euclide,txfonts}
\usepackage[ampersand]{easylist}
\pgfplotsset{compat=1.5}
\usepackage{caption}
\usepackage{subcaption}
\usepackage{color}
\usepackage[inline]{enumitem}
\usepackage{comment}
\usetikzlibrary{matrix}
\usetikzlibrary{arrows,positioning}

\makeatletter
\@addtoreset{equation}{section}
\makeatother

\newcounter{extralabel}[section]

\newtheorem{ittheorem}{Theorem}
\newtheorem{itlemma}{Lemma}
\newtheorem{itproposition}{Proposition}
\newtheorem{itdefinition}{Definition}
\newtheorem{itcorollary}{Corollary}
\newtheorem{itconjecture}{Conjecture}
\newtheorem{itremark}{Remark}
\newtheorem{itassumption}{Assumption}

\newenvironment{theorem}{\addtocounter{extralabel}{1}
	\begin{ittheorem}}{\end{ittheorem}}

\newenvironment{proposition}{\addtocounter{extralabel}{1}
	\begin{itproposition}}{\end{itproposition}}

\newenvironment{definition}{\addtocounter{extralabel}{1}
	\begin{itdefinition}}{\end{itdefinition}}

\newenvironment{corollary}{\addtocounter{extralabel}{1}
	\begin{itcorollary}}{\end{itcorollary}}

\newenvironment{remark}{\addtocounter{extralabel}{1}
	\begin{itremark}}{\end{itremark}}

\newenvironment{assumption}{\addtocounter{extralabel}{1}
	\begin{itassumption}}{\end{itassumption}}

\def\1{{\mathchoice {1\mskip-4mu\mathrm l} 
		{1\mskip-4mu\mathrm l}
		{1\mskip-4.5mu\mathrm l} {1\mskip-5mu\mathrm l}}}

\def\1{{\mathchoice {\rm 1\mskip-4mu l} {\rm 1\mskip-4mu l}
		{\rm 1\mskip-4.5mu l} {\rm 1\mskip-5mu l}}}

\def\N{{\mathbb N}}
\def\Z{{\mathbb Z}}

\def\R{{\mathbb R}}

\newcommand{\E}{{\mathbb E}}
\renewcommand{\P}{{\mathbb P}}

\newcommand{\dd}{{\rm d}}
\newcommand{\ee}{{\rm e}}
\newcommand{\rhosi}{{\rho_i^{\rm{stat}}}}
\newcommand{\rhosz}{{\rho_0^{\rm{stat}}}}
\newcommand{\rhoso}{{\rho_1^{\rm{stat}}}}
\newcommand{\rhoho}{{\rho_1^{\rm{hom}}}}
\newcommand{\rhohz}{{\rho_0^{\rm{hom}}}}

\makeatother
\makeatother
\allowdisplaybreaks
\begin{document}
	
	
	\title{Switching interacting particle systems:\\[0.2cm]
		scaling limits, uphill diffusion and boundary layer}
	
	\author{Simone Floreani$^1$, Cristian Giardin\`{a}$^2$, Frank den Hollander$^3$,\\ 
		Shubhamoy Nandan$^4$, Frank Redig$^5$}
	
	\date{\today}
	
	\maketitle
	
	\begin{abstract}
		In this paper we consider three classes of interacting particle systems on $\Z$: independent random walks, the exclusion process, and the inclusion process. We allow particles to switch their jump rate (the rate identifies the \textit{type} of particle) between $1$ (\textit{fast particles}) and $\epsilon\in[0,1]$ (\textit{slow particles}). 
		The switch between the two jump rates happens at rate $\gamma\in(0,\infty)$. In the exclusion process, the interaction is such that each site can be occupied by at most one particle of each type. In the inclusion process, the interaction takes places between particles of the same type at different sites and between particles of different type at the same site.
		
		We derive the macroscopic limit equations for the three systems, obtained after scaling space by $N^{-1}$, time by $N^2$, the switching rate by $N^{-2}$, and letting $N\to\infty$. The limit equations for the macroscopic densities associated to the fast and slow particles is the well-studied double diffusivity model. This system of reaction-diffusion equations was introduced to model polycrystal diffusion and dislocation pipe diffusion, with the goal to overcome the limitations imposed by Fick's law. In order to investigate the microscopic out-of-equilibrium properties, we analyse the system on $[N]=\{1,\ldots,N\}$, adding boundary reservoirs at sites $1$ and $N$ of fast and slow particles, respectively. Inside $[N]$ particles move as before, but now particles are injected and absorbed at sites $1$ and $N$ with prescribed rates that depend on the particle type. We compute the steady-state density profile and the steady-state current. It turns out that uphill diffusion is possible, i.e., the total flow can be in the direction of increasing total density. This phenomenon, which cannot occur in a single-type particle system, is a violation of Fick's law made possible by the switching between types. We rescale the microscopic steady-state density profile and steady-state current and obtain the steady-state solution of a boundary-value problem for the double diffusivity model.

		
		
		\medskip\noindent
		\emph{Acknowledgement:} The research in this paper was supported by the Netherlands Organisation for Scientific Research (NWO) through grant TOP1.17.019. S.F. thanks Antonio Agresti for several enlightening discussions and Mark Veraar for useful suggestions.
	\end{abstract}
	
	\bigskip
	
	\footnoterule
	\noindent
	\hspace*{0.3cm} {\footnotesize $^{1)}$ 
		Delft Institute of Applied Mathematics, TU Delft, Delft, The Netherlands,
		{\sc s.floreani@tudelft.nl}}\\ 
	\hspace*{0.3cm} {\footnotesize $^{2)}$ 
		Modena and Reggio Emilia University, Modena, Italy, 
		{\sc cristian.giardina@unimore.it}}\\
	\hspace*{0.3cm} {\footnotesize $^{3)}$ 
		Mathematical Institute, Leiden University, Leiden, The Netherlands, 
		{\sc denholla@math.leidenuniv.nl}}\\
	\hspace*{0.3cm} {\footnotesize $^{4)}$ 
		Mathematical Institute, Leiden University, Leiden, The Netherlands,
		{\sc s.nandan@math.leidenuniv.nl}}\\
	\hspace*{0.3cm} {\footnotesize $^{5)}$ 
		Delft Institute of Applied Mathematics, TU Delft, Delft, The Netherlands,
		{\sc f.h.j.redig@tudelft.nl}}

	
	
	\section{Introduction}
	\label{s.intro}
	
	Section~\ref{ss.back} provides the background and the motivation for the paper. Section~\ref{ss.model} defines the model. Section~\ref{ss.dual} identifies the dual and the stationary measures. Section~\ref{ss.outline} gives a brief outline of the remainder of the paper. 
	
	
	\subsection{Background and motivation}
	\label{ss.back} 
	
	Interacting particle systems are used to model and analyse properties of \emph{non-equilibrium systems}, such as macroscopic profiles, long-range correlations and macroscopic large deviations. Some models have additional structure, such as duality or integrability properties, which allow for a study of the fine details of non-equilibrium steady states, such as microscopic profiles and correlations. Examples include zero-range processes, exclusion processes, and models that fit into the algebraic approach to duality, such as inclusion processes and related diffusion processes, or models of heat conduction, such as the Kipnis-Marchioro-Presutti model \cite{CGGR13,DLS02,DEHP93,GKRV09,KMP82}. Most of these models have indistinguishable particles of which the total number is conserved, and so the relevant macroscopic quantity is the \emph{density} of particles.
	
	Turning to more complex models of non-equilibrium, various exclusion processes with \emph{multi-type particles} have been studied \cite{FM06,FM07,K2019}, as well as reaction-diffusion processes \cite{BL07,BDP92,DMP91,DFL85,DFL86}, where non-linear reaction-diffusion equations are obtained in the hydrodynamic limit, and large deviations around such equations have been analysed. In the present paper, we focus on a reaction-diffusion model that on the one hand is simple enough so that via duality a complete microscopic analysis of the non-equilibrium profiles can be carried out, but on the other hand exhibits interesting phenomena, such as \emph{uphill diffusion} and \emph{boundary-layer effects}. In our model we have two types of particles, \textit{fast} and \textit{slow}, that jump at rate $1$ and $\epsilon\in[0,1]$, respectively. Particles of identical type are allowed to interact via exclusion or inclusion. There is no interaction between particles of different type that are at different sites. Each particle can change type at a rate that is adapted to the particle interaction (exclusion or inclusion), and is therefore interacting with particles of different type at the same site. An alternative and equivalent view is to consider two layers of particles, where the layer determines the jump rate (rate $1$ for bottom layer, rate $\epsilon$ for top layer) and where on each layer the particles move according to exclusion or inclusion, and to let particles change layer at a rate that is appropriately chosen in accordance with the interaction. In the limit as $\epsilon\downarrow0$, particles are immobile on the top layer. 
	
	We show that the \emph{hydrodynamic limit} of all three dynamics is a linear reaction-diffusion system known under the name of \textit{double diffusivity model}, namely, 
	\begin{equation}
	\begin{cases}
	\partial_t \rho_0 = \Delta \rho_0 + \Upsilon(\rho_1-\rho_0),\\
	\partial_t \rho_1 = \epsilon \Delta \rho_1 + \Upsilon(\rho_0-\rho_1),
	\end{cases}
	\end{equation}
	where $\rho_i$, $i\in\{0,1\}$, are the macroscopic densities of the two types of particles, and $\Upsilon\in(0,\infty)$ is the scaled switching rate. The above system was introduced in \cite{A79} to model polycrystal diffusion (more generally, diffusion in inhomogeneous porous media) and dislocation pipe diffusion, with the goal to overcome the restrictions imposed by Fick's law. Non-Fick behaviour is immediate from the fact that the total density $\rho=\rho_0+\rho_1$ does not satisfy the classical diffusion equation.
	
	The double diffusivity model was studied extensively in the PDE literature \cite{HA80I,HA80II,H81}, while its discrete counterpart was analysed in terms of a single random walk switching between two layers \cite{H80}. The same macroscopic model was studied independently in the mathematical finance literature in the context of switching diffusion processes \cite{YZ10}. Thus, we have a family of interacting particle systems whose macroscopic limit is relevant in several contexts. Another context our three dynamics fit into are models of interacting active random walks with an internal state that changes randomly (e.g.\ activity, internal  energy) and that determines their diffusion rate and or drift \cite{DM18,FM2018,GPB16,KSS86, MJKKSMRS18,PKS16, ABBS21}. 
	
	An additional motivation to study two-layer models comes from population genetics. Individuals live in colonies, carry different genetics types, and can be either active or dormant. While active, individuals resample by adopting the type of a randomly sampled individual in the same colony, and migrate between colonies by hopping around. Active individuals can become dormant, after which they suspend resampling and migration, until they become active again. Dormant individuals reside in what is called a \emph{seed bank}. The overall effect of dormancy is that extinction of types is slowed down, and so genetic diversity is enhanced by the presence of the seed bank. A wealth of phenomena can occur, depending on the parameters that control the rates of resampling, migration, falling asleep and waking up \cite{BKpr,GdHOpr}. Dormancy not only affects the long-term behaviour of the population quantitatively. It may also lead to qualitatively different equilibria and time scales of convergence. For a panoramic view on the role of dormancy in the life sciences, we refer the reader to \cite{LdHWBBpr}.
	
	From the point of view of non-equilibrium systems driven by boundary reservoirs, switching interacting particle systems have not been studied. On the one hand, such systems have both reaction and diffusion and therefore exhibit a richer non-equilibrium behaviour. On the other hand, the macroscopic equations are linear and exactly solvable in one dimension, and so these systems are simple enough to make a detailed microscopic analysis possible. As explained above, the system can be viewed as an interacting particle system on two layers. Therefore duality properties are available, which allows for a detailed analysis of the system coupled to reservoirs, dual to an absorbing system. In one dimension the analysis of the microscopic density profile reduces to a computation of the absorption probabilities of a simple random walk on a two-layer system absorbed at the left and right boundaries. From the analytic solution, we can identify both the density profile and the current in the system. This leads to two interesting phenomena. The first phenomenon is \emph{uphill diffusion} (see e.g.\ \cite{CMP18, CGGV17, CDP, DMP21, R15}), i.e., in a well-defined parameter regime the current can go against the particle density gradient: when the total density of particles at the left end is higher than at the right end, the current can still go from right to left. The second phenomenon is \emph{boundary-layer behaviour}: in the limit as $\epsilon\downarrow0$, in the macroscopic stationary profile the densities in the top and bottom layer are equal, which for unequal boundary conditions in the top and bottom layer results in a \emph{discontinuity} in the stationary profile. Corresponding to this jump in the macroscopic system,  we identify a boundary layer of size $\sqrt{\epsilon}\,\log(1/\epsilon)$ in the microscopic system where the densities are unequal. The quantification of the \emph{size} of this boundary layer is an interesting corollary of the exact macroscopic stationary profile that we obtain from the microscopic system  via duality.
	
	
	\subsection{Three models}
	\label{ss.model}
	
	For $\sigma\in\{-1,0,1\}$ we introduce an interacting particle system on $\Z$ where the particles randomly switch their jump rate between two possible values, $1$ and $\epsilon$, with $\epsilon\in[0,1]$. For $\sigma=-1$ the particles are subject to the exclusion interaction, for $\sigma=0$ the particles are independent, while for $\sigma=1$ the particles are subject to the inclusion interaction. Let
	$$
	\begin{aligned}
	\eta_0(x) &:= \text{number of particles at site } x \text{ jumping at rate } 1,\\
	\eta_1(x) &:= \text{number of particles at site } x \text{ jumping at rate } \epsilon.
	\end{aligned}
	$$
	The configuration of the system is 
	$$
	\eta := \{\eta(x)\}_{x\in \Z} \in \mathcal X 
	= \begin{cases} 
	\{0,1\}^\Z\times \{0,1\}^\Z,
	&\text{if }\sigma=-1,\\ 
	\N_0^{\Z}\times \N_0^{\Z},
	&\text{if }\sigma=0,1,
	\end{cases}
	$$
	where 
	$$
	\eta(x) := (\eta_0(x),\eta_1(x)), \qquad x \in \Z.
	$$ 
	We call $\eta_0=\{\eta_0(x)\}_{x\in \Z}$ and $\eta_1=\{\eta_1(x)\}_{x\in \Z}$ the configurations of \textit{fast particles}, respectively, \textit{slow particles}. When $\epsilon = 0$ we speak of  \textit{dormant particles} (see Fig.~\ref{fig:sleepingnew}). 
	
	\begin{figure}[htbp]
		\begin{subfigure}{\linewidth}	
			\centering
			\includegraphics[width=0.6\linewidth,height=0.2\textheight]{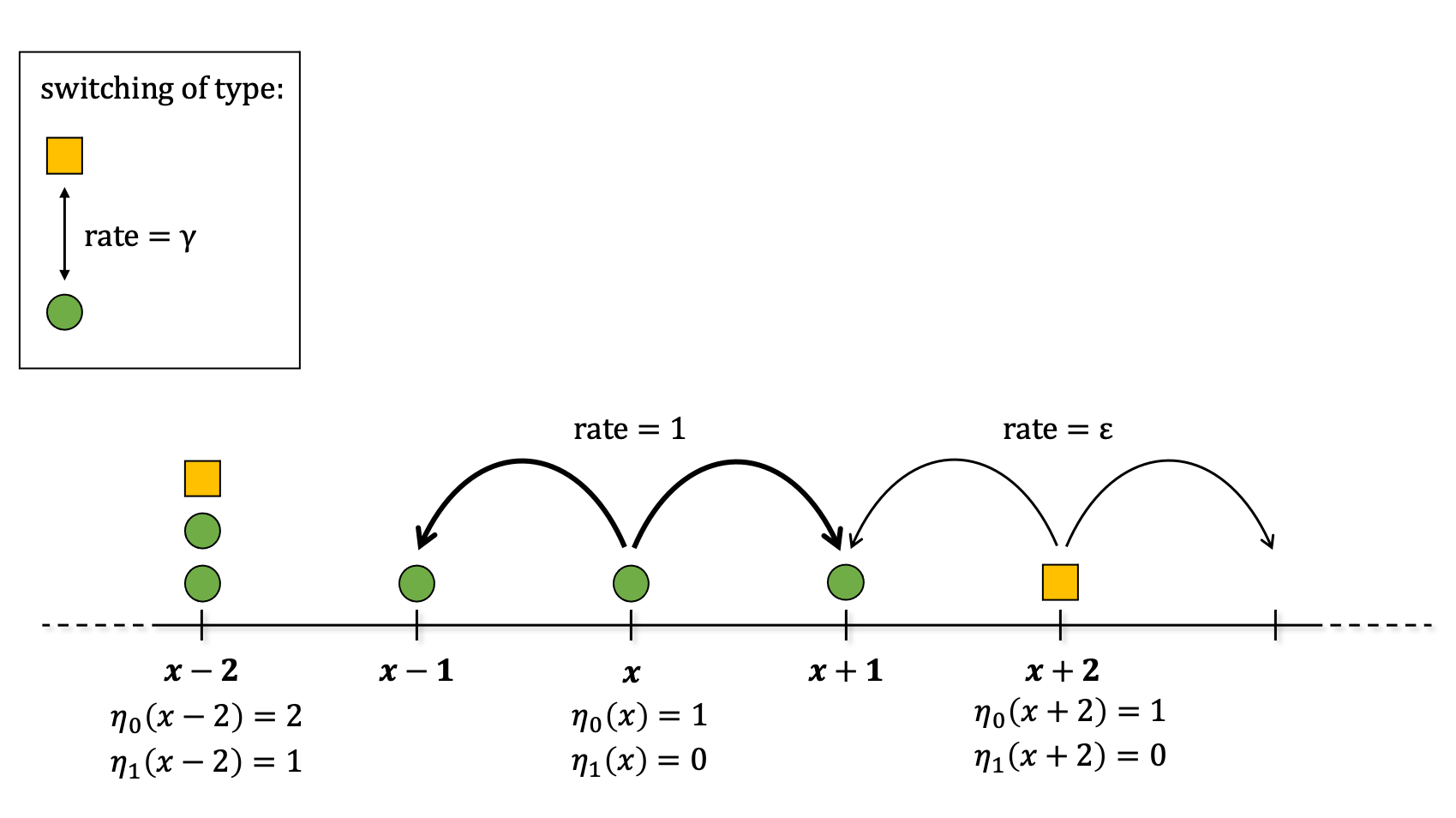}
			\caption{Representation via \emph{slow} and \emph{fast} particles moving on the one-layer graph $\Z$.}
		\end{subfigure}
		\begin{subfigure}{\linewidth}	
			\centering
			\includegraphics[width=0.6\linewidth,height=0.2\textheight]{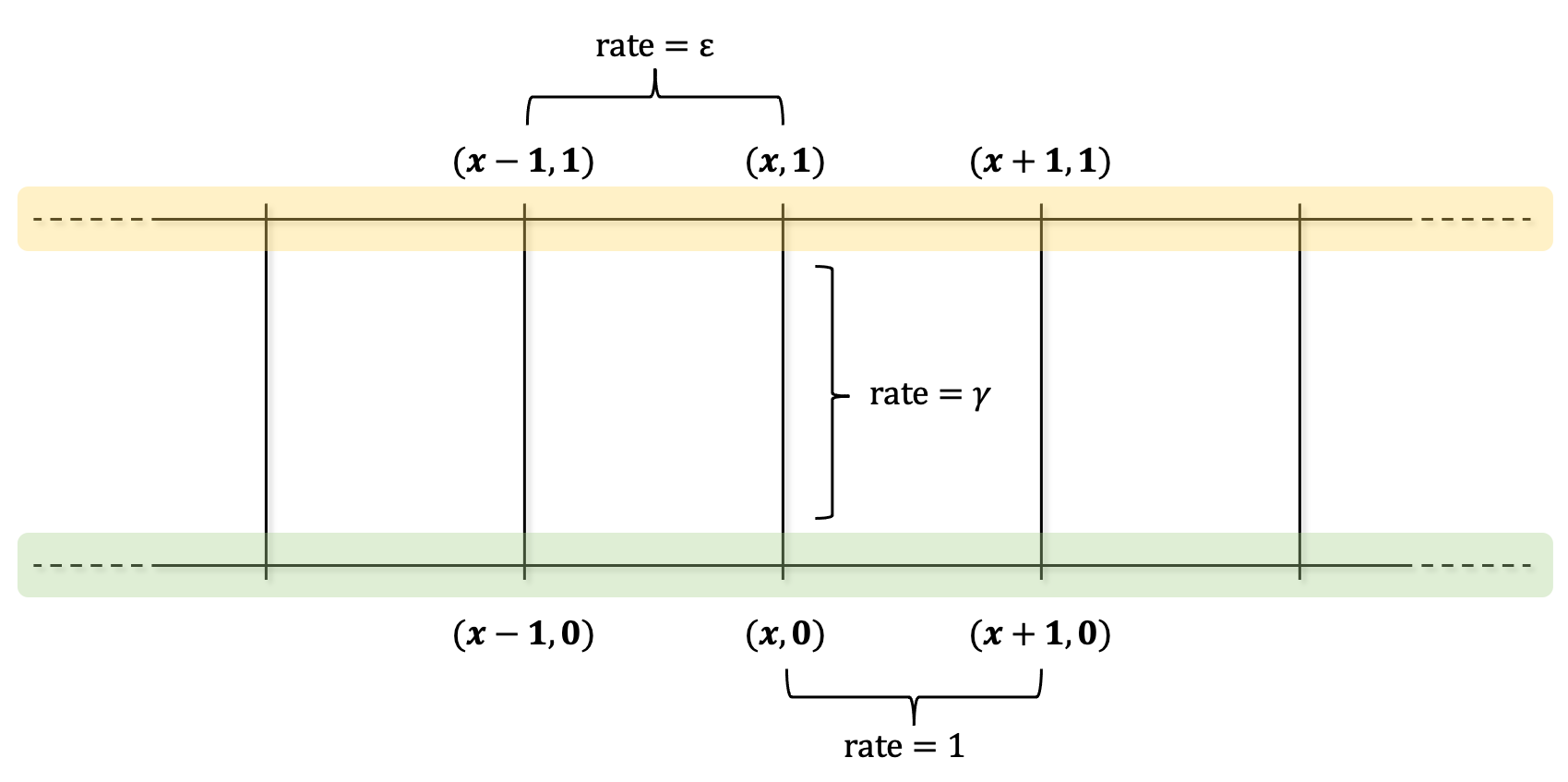}
			\caption{Representation via particles moving on the two-layer graph $\Z\times I$.}
		\end{subfigure}
		\caption{Two equivalent representations of switching independent random walks ($\sigma=0$).}
		\label{fig:sleepingnew}
	\end{figure}
	
	\begin{definition}{\bf [Switching interacting particle systems]}
		\label{def:system}
		{\rm For $\epsilon \in [0,1]$ and $\gamma \in (0,\infty)$, let $L_{\epsilon,\gamma}$ be the generator 
			\begin{equation}
			\label{eq: generator IRW with sleeping particles}
			\begin{aligned}
			L_{\epsilon,\gamma} := L_0 +\epsilon L_1 + \gamma L_{0 \updownarrow 1},
			\end{aligned}
			\end{equation}
			acting on bounded cylindrical functions $f\colon\mathcal X\to \R$ as
			\begin{equation*}
			\begin{aligned}
			&(L_0 f)(\eta) = \sum_{|x-y|=1} \Big\{\eta_0(x)(1+\sigma\eta_0(y))\,\big[f((\eta_0-\delta_x+\delta_y,\eta_1))-f(\eta)\big]\\[-0.2cm]
			&\qquad\qquad\qquad\qquad + \eta_0(y)(1+\sigma\eta_0(x))\,\big[f((\eta_0+\delta_x-\delta_y,\eta_1))-f(\eta)\big]\Big\},\\[0.2cm]
			&(L_1 f)(\eta)=  \sum_{|x-y|=1} \Big\{\eta_1(x)(1+\sigma\eta_1(y))\,\big[f((\eta_0,\eta_1-\delta_x+\delta_y))-f(\eta)\big]\\[-0.2cm]
			&\qquad\qquad\qquad\qquad + \eta_1(y)(1+\sigma\eta_1(x))\, \big[f((\eta_0,\eta_1+\delta_x-\delta_y))-f(\eta)\big]\Big\},\\[0.2cm]
			&(L_{0 \updownarrow 1}f)(\eta) = \gamma \sum_{x\in \Z^d} \Big\{\eta_0(x)(1+\sigma\eta_1(x))\, 
			\big[f((\eta_0-\delta_x,\eta_1+\delta_x))-f(\eta)\big]\\[-0.2cm]
			&\qquad\qquad\qquad\qquad +\eta_1(x)(1+\sigma\eta_0(x))\, \big[f((\eta_0+\delta_x,\eta_1-\delta_x))-f(\eta)\big]\Big\}.
			\end{aligned}
			\end{equation*}
			The Markov process $\{\eta(t)\colon\,t\ge 0\}$ on state space $\mathcal X$ with 
			$$
			\eta(t) := \{\eta(x,t)\}_{x\in\Z} = \big\{(\eta_0(x,t),\eta_1(x,t))\big\}_{x\in\Z},
			$$
			hopping rates $1,\epsilon$ and switching rate $\gamma$ is called \textit{switching exclusion process} for $\sigma=-1$, \textit{switching random walks} for $\sigma=0$ (see Fig.~\ref{fig:sleepingnew}), and \textit{switching inclusion process} for $\sigma=1$.} \hfill $\spadesuit$
	\end{definition}
	
	
	\subsection{Duality and stationary measures}
	\label{ss.dual}

	The systems defined in \eqref{eq: generator IRW with sleeping particles} can be equivalently formulated as jump processes on the graph (see Fig.~\ref{fig:sleepingnew}) with vertex set $\{(x,i) \in \Z^d \times I\}$, with $I = \{0,1\}$ labelling the two layers, and edge set given by the nearest-neighbour relation  
	$$
	(x,i)\sim (y,j) \quad \text{ when } \quad
	\begin{cases}
	|x-y|=1 \text{ and } i=j,\\
	x=y \text{ and } |i-j| =1.
	\end{cases}
	$$ 
	In this formulation the particle configuration is 
	$$
	\eta=(\eta_i(x))_{(x,i)\in\Z\times I}
	$$ 
	and the generator $L$ is given by 
	\begin{equation}
	\begin{aligned}
	\label{eq: equivalent generator}
	&(Lf)(\eta)
	= \sum_{i\in I}\sum_{|x-y|=1}\epsilon^i\eta_i(x) (1+\sigma\eta_i(y))\,[f(\eta-\delta_{(x,i)}+\delta_{(y,i)})-f(\eta)]\\
	&\qquad\qquad\qquad \quad+\epsilon^i\eta_i(y) (1+\sigma\eta_i(x))\,[f (\eta-\delta_{(y,i)}+\delta_{(x,i)})-f(\eta)]\\
	&\qquad\qquad +\sum_{i\in I}\gamma\sum_{x\in\Z} \eta_i(x)(1+\sigma\eta_{1-i})\,[f(\eta-\delta_{(x,i)}+\delta_{(x,1-i)}) -f(\eta)].
	\end{aligned}
	\end{equation}
	Thus, a single particle (when no other particles are present) is subject to two movements:
	\begin{itemize}
		\item[i)]
		\textit{Horizontal movement}: 
		In layer $i=0$ and $i=1$ the particle performs a nearest-neighbour random walk on $\Z$ at rate $1$, respectively, $\epsilon$.
		\item[ii)]
		\textit{Vertical movement}: 
		The particle switches layer at the same site at rate $\gamma$.
	\end{itemize} 
	
	It is well known (see e.g. \cite{RS18}) that for these systems  there exists a one-parameter family of reversible product measures
	$$
	\big\{\mu_\theta=\otimes_{(x,i)\in\Z\times I}\nu_{(x,i),\theta}\colon\, \theta \in \Theta\big\}
	$$
	with $\Theta=[0,1]$ if $\sigma=-1$ and $\Theta=[0,\infty)$ if $\sigma\in\{0,1\}$, and with marginals given by 
	\begin{align}
	\label{eq:marginals}
	\nu_{(x,i),\theta} = \begin{dcases}  
	\text{Bernoulli}\,(\theta),  &\sigma = -1,\\[.15cm]
	\text{Poisson}\,(\theta), &\sigma = 0,\\[.15cm]
	\text{Negative--Binomial}\,(1,\tfrac{\theta}{1+\theta}), 
	&\sigma = 1.
	\end{dcases}
	\end{align}
	Moreover, the \textit{classical self-duality relation} holds, i.e., for all configurations $\eta, \, \xi \in \mathcal X$ and for all times $t\ge0$,
	$$\E_\eta[D(\xi,\eta_t)]=\E_\xi[D(\xi_t,\eta)],$$ with $\{\xi(t): \ t\ge 0\}$ and $\{\eta(t):\ t\ge 0\}$ two copies of the process with generator given in \eqref{eq: generator IRW with sleeping particles} and self-duality function $D\colon\,\mathcal{X} \times \mathcal{X} \to \R$ given by
	\begin{equation}
	\label{eq: duality function}
	D(\xi,\eta) := \prod_{(x,i) \in \Z^d \times I} d(\xi_i(x),\eta_i(x)),
	\end{equation}
	with 
	\begin{equation}
	\label{eq: single site duality}
	d(k,n) := \frac{n!}{(n-k)!}\frac{1}{w(k)}\,\mathbf{1}_{\{k\leq n\}}
	\end{equation}
	and 
	\begin{equation}
	w(k):= \begin{cases}   
	\frac{\Gamma(1+k)}{\Gamma(1)}, &\sigma=1,\\
	1, &\sigma=-1,0. 
	\end{cases}
	\end{equation}
	
	\begin{remark}{\bf [Possible extensions]}
		{\rm Note that we could allow for more than two layers, inhomogeneous rates and non-nearest neighbour jumps as well, and the same duality relation would still hold (see e.g.\ \cite{FRSHDL} for an inhomogeneous version of the exclusion process). More precisely, let $\{\omega_i(\{x,y\})\}_{x,y\in\Z}$ and $\{\alpha_i(x)\}_{x\in\Z}$ be collections of bounded weights for $i\in I_M=\{0,1,\ldots, M\}$ with $M<\infty$. Then the interacting particle systems with generator
\begin{equation}
			\begin{aligned}
			(L_{D,\gamma}f)(\eta)
			&=\sum_{i=0}^M D_i \sum_{|x-y|=1} 
			\omega_i(\{x,y\})\,\Big\{\eta_i(x)\,(\alpha_i(y)+\sigma\eta_i(y))\,\big[f(\eta-\delta_{(x,i)}+\delta_{(y,i)})-f(\eta)\big]\\
			&\qquad\qquad\qquad\qquad \qquad+ \eta_i(y)\,(\alpha_i(x)+\sigma\eta_i(x))\,\big[f (\eta-\delta_{(y,i)}+\delta_{(x,i)})-f(\eta)\big]\Big\}\\
			&\quad + \sum_{i=0}^{M-1} \gamma_{\{i,i+1\}}\,\sum_{x\in \Z} \Big\{\eta_i(x)\,
			\big[f(\eta-\delta_{(x,i)}+\delta_{(x,i+1)})-f(\eta)\big]\\
			&\qquad\qquad\qquad\qquad +\eta_{i+1}(x)\,\big[f(\eta-\delta_{(x,i+1)}+\delta_{(x,i)})-f(\eta)\big]\Big\},
			\end{aligned}
			\end{equation}
			with 	$\eta=(\eta_i(x))_{(x,i)\in\Z\times I_M}$, $\{D_i\}_{i\in I_M}$ a bounded decreasing collection of weights in $[0,1]$	and $\gamma_{\{i,i+1\}}\in (0,\infty)$, are still self-dual with duality function as in \eqref{eq: duality function}, but with $I$ replaced by $I_M$ and single-site duality functions given by $d_{(x,i)}(k,n)=\frac{n!}{(n-k)!}\frac{1}{w_{(x,i)}(k)}\,\mathbf{1}_{\{k\leq n\}}$ with
			\begin{align*}
			w_{(x,i)}(k)\ := \begin{dcases}
			\frac{\alpha_i(x)!}{(\alpha_i(x)-k)!}\1_{\{k \leq \alpha_i(x)\}}, 
			&\sigma = -1,\\
			\alpha_i(x)^k, 
			&\sigma = 0,\\
			\frac{\Gamma(\alpha_i(x)+k)}{\Gamma(\alpha_i(x))}, &\sigma = 1.
			\end{dcases}
			\end{align*}
			In the present paper we prefer to stick to the two-layer homogeneous setting in order not to introduce extra notations. However, it is straightforward to extend many of our results to the inhomogeneous multi-layer model.}\hfill $\spadesuit$
	\end{remark}

	\subsection{Outline}
	\label{ss.outline}
	
	Section~\ref{s.hydro} identifies and analyses the \textit{hydrodynamic limit} of the system in Definition~\ref{def:system} after scaling space, time and switching rate diffusively. We thereby exhibit a class of interacting particle systems whose microscopic dynamics scales to a macroscopic dynamics called the double diffusivity model. Moreover, we provide a discussion on the solutions of this model, connecting mathematical literature applied to material science and to financial mathematics. Section~\ref{s.bdres} looks at what happens, both microscopically and macroscopically, when \textit{boundary reservoirs} are added, resulting in a non-equilibrium flow. Here the possibility of \textit{uphill diffusion} becomes manifest, which is absent in single-layer systems, i.e., the two layers interact in a way that allows for a violation of Fick's law. We characterise the parameter regime for uphill diffusion. Moreover, we show that, in the limit as  $\epsilon\downarrow 0$, the macroscopic stationary profile of the type-1 particles adapts to the microscopic stationary profile of the type-0 particles, resulting in a \emph{discontinuity} at the boundary in the case of unequal boundary conditions on the top layer and the bottom layer. Appendix~\ref{s.app*} provides the inverse of a certain boundary-layer matrix. 

	
	\section{The hydrodynamic limit}
	\label{s.hydro}
	
	In this section we scale space, time and switching diffusively, so as to obtain a \textit{hydrodynamic limit}. In Section~\ref{ss.micmac} we scale space by $1/N$, time by $N^2$, the switching rate by $1/N^2$, introduce scaled microscopic empirical distributions, and let $N\to\infty$ to obtain a system of macroscopic equations. In Section~\ref{ss.repr} we recall some known results for this system, namely, there exists  a unique solution that can be represented in terms of an underlying diffusion equation or, alternatively, via a Feynman-Kac formula involving the switching diffusion process.
	
	
	\subsection{From microscopic to macroscopic}
	\label{ss.micmac}
	
	Let $N \in \N$, and consider the scaled generator $L_{\epsilon,\gamma_N}$ (recall \eqref{eq: generator IRW with sleeping particles}) with $\gamma_N = \Upsilon/N^2$ for some $\Upsilon \in (0,\infty)$, i.e., the reaction term is slowed down by a factor $N^2$ in anticipation of the diffusive scaling we are going to consider.
	
	In order to study the collective behaviour of the particles after scaling of space and time, we introduce the following empirical density fields, which are Radon measure-valued càdlàg processes: 
	$$
	\mathsf X^{N}_0(t) := \frac{1}{N} \sum_{x\in \Z} \eta_{0}(x,tN^2)\,\delta_{x/N},\qquad
	\mathsf X^{N}_1(t) := \frac{1}{N} \sum_{x\in \Z} \eta_{1}(x,tN^2)\,\delta_{x/N}.
	$$
	In order to derive the hydrodynamic limit for the switching interacting particle systems, we need the following set of assumptions. In the following we denote by $C^\infty_c(\R)$ the space of infinitely differentiable functions with values in $\R$ and compact support, by $C_b(\R;\sigma)$ the space of bounded and continuous functions  with values in $\R_+$ for $\sigma\in\{0,1\}$ and with values in $[0,1]$ for $\sigma=-1$, by $C_0(\R)$ the space of continuous functions vanishing at infinity, by $C^2_0(\R)$ the space of twice differentiable functions vanishing at infinity and by $M$ the space of Radon measure on $\R$.
	
	\begin{assumption}{\bf [Compatible initial conditions]}
		\label{assumption: initial condition}
		{\rm Let $
			\bar\rho_i\in C_b(\R;\sigma)$  for $i\in\{0,1\}$
			be two given functions, called initial macroscopic profiles. We say that a sequence $(\mu_N)_{N\in \N}$ of measures on $\mathcal X$ is a sequence of compatible initial conditions when:
			\begin{itemize}
				\item[(i)] 
				For any $i\in\{0,1\}$, $g\in C^\infty_c(\R)$ and $\delta >0,$ 
				$$
				\lim_{N\to \infty}\mu_N\left( \left|  \langle\mathsf X^{N}_0(t),g\rangle - \int_\R \dd x\ \bar \rho_i(x)g(x)   \right| >\delta  \right)=0.
				$$  
				\item[(ii)] 
				There exists a constant $C<\infty$ such that
				\begin{equation}
				\label{eq: bound initial condition}
				\sup_{(x,i) \in \Z \times I} \E_{\mu_N}[\eta_i(x)^2] \le C.
				\end{equation}
			\end{itemize}
		}\hfill $\spadesuit$
	\end{assumption}
	\noindent
	Note that Assumption \ref{assumption: initial condition}(ii) is the same as employed in \cite[Theorem  1, Assumption (b)]{CS2020} and  is trivial for the exclusion process.
	\begin{theorem}{\bf [Hydrodynamic scaling]}
		\label{theorem:HDL IRW with slow particles}
		Let $\bar \rho_0,\bar \rho_1 \in C_b(\R;\sigma)$ be two initial macroscopic profiles, and let $(\mu_N)_{N\in\N}$ be a sequence of compatible initial conditions. Let $\P_{\mu_N}$ be the law of the measure-valued process 
		$$
		\{X^N(t)\colon\,t \geq 0\} , \qquad X^N(t) := (X_0^N(t), X_1^N(t)),
		$$ 
		induced by the initial measure $\mu_N$. Then, for any $T, \delta>0$ and $g\in C^\infty_c(\R)$,
		$$
		\lim_{N\to \infty}\P_{\mu_N}\left(\sup_{t\in[0,T]} \left|\,\langle\mathsf X^{N}_i(t),g\rangle
		- \int_\R \dd x\, \rho_i(x,t)g(x)\,\right| >\delta \right)=0, \qquad i \in I,
		$$
		where $\rho_0$ and $\rho_1$ are the unique continuous and bounded strong solutions of the system 
		\begin{equation}
		\label{eq:HDL IRW with slow particles}
		\begin{cases}
		\partial_t \rho_0 = \Delta \rho_0 + \Upsilon(\rho_1-\rho_0),\\
		\partial_t \rho_1 = \epsilon \Delta \rho_1 + \Upsilon(\rho_0-\rho_1),
		\end{cases}
		\end{equation}
		with initial conditions
		\begin{equation}
		\label{eq:initial conditions}
		\begin{cases}
		\rho_0(x,0) = \bar \rho_0(x),\\
		\rho_1(x,0) = \bar \rho_1(x).
		\end{cases}
		\end{equation}
	\end{theorem}
	\begin{proof} 
		The proof follows the standard route presented in \cite[Section 8]{S05} (see also \cite{DMP91, CS2020}). We still explain the main steps because the two-layer setup is not standard.
		First of all, note that the macroscopic equation \eqref{eq:HDL IRW with slow particles} can be straightforwardly identified by computing the action of the rescaled generator $L^N = L_{\epsilon,\Upsilon/N^2}$ on the cylindrical functions $f_i(\eta):=\eta_i(x), \ i\in \{0,1\}$, namely ,
		$$
		(L^N f_i)(\eta) = \epsilon^i \left[ \eta_i(x+1)-2\eta_i(x)+\eta_i(x-1)\right]
		+ \frac{\Upsilon}{N^2}\left[\eta_{1-i}(x)-\eta_i(x)\right]
		$$
		and hence, for any $g\in C^{\infty}_c(\R)$,
		$$
		\begin{aligned}
		\int_0^{tN^2} \dd s\,L^N(\langle\mathsf X^{N}_i(s),g\rangle)
		&= \int_0^{tN^2}\dd s\, \frac{\epsilon^i}{N} \sum_{x\in\Z} \eta_i(x,s)\,
		\tfrac{1}{2}\,\big[g((x+1)/N)-2g(x/N)+g((x-1)/N)\big]\\
		&\quad +\int_0^{tN^2} \dd s\,\frac{1}{N} \sum_{x\in\Z} g(x/N)\, 
		\frac{\Upsilon}{N^2}\,\big[\eta_{1-i}(x,s)-\eta_i(x,s)\big],\\
		\end{aligned}
		$$
		where we moved the generator of the simple random walk to the test function by using reversibility w.r.t. the counting measure. By the regularity of $g$, we thus have 
		$$
		\begin{aligned}
		&\int_0^{tN^2} \dd s\,L^N(\langle\mathsf X^{N}_i(s),g\rangle)
		= \int_0^{t}\dd s\, \langle\mathsf X^{N}_i(sN^2),\epsilon \Delta g\rangle  +\int_0^{tN^2} \dd s\,
		\frac{\Upsilon}{N^2}\,\big[\langle\mathsf X^{N}_{1-i}(s),g\rangle-\langle\mathsf X^{N}_i(s),g\rangle\big]+ o(\tfrac{1}{N^2}),
		\end{aligned}
		$$
		which is the discrete counterpart of the weak formulation of the right-hand side of \eqref{eq:HDL IRW with slow particles}, i.e., $\int_0^t\dd s\int_\R \dd x\, \rho_{i}(x,s)\Delta g(x) + \Upsilon \int_0^t\dd s\int_\R\dd x\, [\rho_{1-i}(x,s)-\rho_i(x,s)]g(x)$.
		Thus, as a first step, we show that 
		\begin{multline*}\lim_{N\to \infty}\P_{\mu_N}\left( \sup_{t\in[0,T]}\left|  \langle\mathsf X^{N}_i(t),g\rangle-  \langle\mathsf X^{N}_i(0),g\rangle   -  \int_0^{t}\dd s\, \langle\mathsf X^{N}_i(sN^2),\epsilon^i\Delta g\rangle  \right.\right.\\\left. \left.-\int_0^{tN^2} \dd s\,
		\frac{\Upsilon}{N^2}\,\big[\langle\mathsf X^{N}_{1-i}(s)-\mathsf X^{N}_i(s),g\rangle\big]   \right|   >\delta     \right)=0.
		\end{multline*}
		In order to prove the above convergence we employ the Dynkin's formula for Markov processes, which gives that the process defined as 
		\begin{align*}
		M^N_i(g,t):=  \langle\mathsf X^{N}_i(t),g\rangle- \langle\mathsf X^{N}_i(0),g\rangle-\int_0^{tN^2} \dd s\,L^N(\langle\mathsf X^{N}_i(s),g\rangle)
		\end{align*}
		is a martingale w.r.t. the natural filtration generated by the process $\{\eta_t\}_{t\ge 0}$ and with predictable quadratic variation expressed in terms of the carr\'{e} du champ, i.e.,
		$$\langle M^N_i(g,t), M^N_i(g,t)  \rangle = \int_0^t \dd s\,\E_{\mu_N}\left[ \Gamma^N_i(g,s)  \right]$$
		with 
		$$ \Gamma^N_i(g,s) =L^N \left( \langle\mathsf X^{N}_i(s),g\rangle\ \right)^2 -\langle\mathsf X^{N}_i(s),g\rangle L^N \left( \langle\mathsf X^{N}_i(s),g\rangle\ \right).$$
		We then have, by Chebyshev's inequality and Doob's martingale inequality, 
		\begin{align}\label{eq: Chebyshev argumebt }
		\nonumber&\P_{\mu_N}\left(  \sup_{t\in[0,T]}\left|  \langle\mathsf X^{N}_0(s),g\rangle-  \langle\mathsf X^{N}_0(s),g\rangle   -  \int_0^{t}\dd s\, \langle\mathsf X^{N}_0(sN^2),\epsilon\Delta g\rangle  \right.\right.\\&\qquad \qquad \qquad \qquad \qquad \quad \quad\qquad \qquad \qquad\left.\left.-\int_0^{tN^2} \dd s\,
		\frac{\Upsilon}{N^2}\,\big[\langle\mathsf X^{N}_1(s),g\rangle-\langle\mathsf X^{N}_0(s),g\rangle\big]   \right|   >\delta     \right)\\
		\nonumber&\le \frac{1}{\delta^2}\E_{\mu_N}\left[  \sup_{t\in[0,T]}\left |M^N_i(g,s)\right|^2 \right]\le \frac{4}{\epsilon^2}\E_{\mu_N}\left[ \left |M^N_i(g,T)\right|^2 \right]=\frac{4}{\epsilon^2}\E_{\mu_N}\left[    \langle   M^N_i(g,T), M^N_i(g,T)  \rangle ^2 \right]\\
		&=\frac{4}{\delta^2N^2}\E_{\mu_N}\left[ \int_0^{N^2T}\dd s\,\sum_{x\in\Z^d}\eta_i(x,s)(1+\sigma \eta_i(x\pm 1,s))\left( g\left(\frac{x\pm 1}{N}\right)   -g\left(\frac{x}{N}\right)     \right)  \right]\\
		&\nonumber+\frac{4\Upsilon}{\delta^2N^4}\E_{\mu_N}\left[  \int_0^{N^2T} \dd s\, \sum_{x\in\Z^d} (\eta_i(x,s)+\eta_{1-i}(x,s)+2\sigma\eta_i(x,s)\eta_{1-i}(x,s)) g^2\left(\frac{x}{N}\right)   \right],
		\end{align}
		where in the last equality we explicitly computed the carré du champ.
		Let $k\in\N$ be such that the support of $g$ is in $[-k,k]$. Then, by the regularity of $g$, \eqref{eq: Chebyshev argumebt } is bounded by
		\begin{multline}\frac{4}{\delta^2N^2}(N^2T)(2k+1)N \frac{\| \nabla g\|_{\infty}}{N^2}\sup_{x,\in\Z,s\in[0,N^2T]}\E_{\mu_N}\left[ \eta_i(x,s)(1+\sigma \eta_i(x + 1,s)) \right]\\+\frac{4\Upsilon}{\delta^2N^4}(N^2T)(2k+1)N\| g \|_{\infty}\sup_{x,\in\Z,s\in[0,N^2T]}\E_{\mu_N}\left[ \eta_i(x,s)+\eta_{1-i}(x,s)+2\sigma\eta_i(x,s)\eta_{1-i}(x,s) \right] .
		\end{multline}
		We now show that, as a consequence of \eqref{eq: bound initial condition},  for any $(x,i),\ (y,j)\in\Z\times I,$
		\begin{equation}\label{eq: bound at later time}
		\E_{\mu_N}\left[ \eta_i(x,s)\right]\le C, \quad \E_{\mu_N}\left[ \eta_i(x,s) \eta_j(y,s)\right]\le C,
		\end{equation}
		from which we obtain 
		\begin{align}
		\nonumber&\P_{\mu_N}\left(  \sup_{t\in[0,T]}\left|  \langle\mathsf X^{N}_0(s),g\rangle-  \langle\mathsf X^{N}_0(s),g\rangle   -  \int_0^{t}\dd s\, \langle\mathsf X^{N}_0(sN^2),\epsilon\Delta g\rangle \right.\right.\\&\qquad \qquad \qquad \qquad \qquad \quad \quad\qquad \qquad \qquad\left.\left. -\int_0^{tN^2} \dd s\,
		\frac{\Upsilon}{N^2}\,\big[\langle\mathsf X^{N}_1(s),g\rangle-\langle\mathsf X^{N}_0(s),g\rangle\big]   \right|   >\delta     \right)\\
		\nonumber&\le \frac{8T}{\delta^2N}(2k+1)\| \nabla g\|_{\infty}C+\Upsilon\frac{16 T}{\delta^2N}(2k+1)\| g \|_{\infty} C,
		\end{align}
		and the desired convergence follows.
		In order to prove \eqref{eq: bound at later time}, first of all note 	that, by the Cauchy-Schwartz inequality, it follows from \eqref{eq: bound initial condition}  that, for any $(x,i),\ (y,j)\in\Z\times I$,
		\begin{equation}\label{eq: bound x ne y}
		\E_{\mu_N}\left[ \eta_i(x)\eta_j(y) \right]\le C.
		\end{equation}
		Moreover, recalling the duality functions given in \eqref{eq: duality function} and defining the configuration $\xi=\delta_{(x,i)}+\delta_{(y,j)}$ for $(x,i)\ne(y,j)$, we have that
		\begin{align*}
		&\E_{\mu_N}\left[ \eta_i(x,t)\eta_j(y,t) \right]=\E_{\mu_N}[D(\xi,\eta_t)]=\int_{\mathcal X} \E_\eta[D(\xi,\eta_t)]\,\dd\mu_N(\eta)\\&=\int_{\mathcal X} \E_\xi[D(\xi_t,\eta)]\,\dd\mu_N(\eta)=\E_\xi\left[  \E_{\mu_N}[D(\xi_t,\eta)] \right]
		\end{align*}
		and, labeling the particles in the dual configuration as $(X_t,i_t)$ and $(Y_t,j_t)$ with initial conditions $(X_0,i_0)=(x,i)$ and $(Y_0,j_0)=(y,j)$, we obtain
		\begin{align}\label{eq: bound duality any time}
		\nonumber&	\E_{\mu_N}\left[ \eta_i(x,t)\eta_j(y,t) \right]=\E_\xi\left[  \E_{\mu_N}[\eta_{i_t}(X_t)\eta_{j_t}(Y_t) \1 _{(X_t,i_t)\ne (Y_t,j_t)}] +\E_{\mu_N}[\eta_{i_t}(X_t)(\eta_{i_t}(X_t)-1) \1 _{(X_t,i_t)=(Y_t,j_t)}] \right] \\
		&\le \E_\xi\left[  \E_{\mu_N}[\eta_{i_t}(X_t)\eta_{j_t}(Y_t)] \right]\le \E_\xi\left[ \sup_{(x,i),(y,j)\in\Z\times\{0,1\}} \E_{\mu_N}[\eta_{i}(x)\eta_{j}(y)] \right]\le C,
		\end{align}
		where we used \eqref{eq: bound x ne y} in the last inequality. Similarly, for $\xi=\delta_{(x,i)}$ and $(X_t,i_t)$ the  dual particle with initial condition $(X_0,i_0)=(x,i)$,	we have that $\E_{\mu_N}\left[ \eta_i(x,t) \right]\le \E_{\mu_N}\left[D(\xi,\eta_t   )\right]=\E_\xi[ \E_{\mu_N}  [\eta_{i_t}(X_t)  ]  ]$.  Using that $\eta_i(x)\le \eta_i(x)^2$ for any $(x,i)\in \Z\times I$ and using  \eqref{eq: bound initial condition}, we obtain \eqref{eq: bound at later time}. 
		
		The proof is concluded after showing the following:
		\begin{itemize}
			\item[i)] Tightness holds for the sequence of distributions of the processes $\{\mathsf X^N_i\}_{N\in\N}$, denoted by $\{Q_N\}_{N\in \N}$.
			\item[ii)] All limit points coincide and  are supported by the unique path $\mathsf X _i(t, \dd x)=\rho_i(x,t)\,\dd x$, with $\rho_i$ the unique weak (and in particular strong) bounded and continuous solution  of \eqref{eq:HDL IRW with slow particles}.
		\end{itemize}
		While for (i) we provide an explanation, we skip the proof of (ii) because it is standard and based on PDE arguments, namely, the existence and the uniqueness of the solutions in the class of continuous-time functions with values in $C_b(\R,\sigma)$ (we refer to  \cite[Lemma 8.6 and 8.7]{S05} for further details), and the fact that Assumption \ref{assumption: initial condition}(i) ensures that  the initial condition of \eqref{eq:HDL IRW with slow particles} is also matched.

		Tightness of the sequence $\{Q_N\}_{N\in\N}$ follows from the compact containment condition on the one hand, i.e., for any $\delta>0$ and $t>0$ there exists a compact set $K\subset M$ such that $\P_{\mu_N}(X^N_i\in K)>1-\delta$, and the equi-continuity condition on the other hand, i.e., $\limsup_{N\to \infty}\P_{\mu_N}(  \omega(\mathsf X_i^N,\delta,T))\ge \mathfrak e)\le \mathfrak e$ for $\omega(\alpha,\delta,T):=\sup\{  d_M(\alpha(s),\alpha(t)): s,t\in[0,T], |s-t|\le \delta    \}$ with $d_M$ the metric on Radon measures defined as 
		$$d_M(\nu_1,\nu_2):=\sum_{j\in\N}2^{-j}\left(   1 \wedge \left | \int_\R \phi_j \dd \nu_1 -\int_\R \phi_j \dd \nu_2 \right |  \right)$$ for an appropriately chosen sequence of functions $(  \phi_j )_{j\in\N}$ in $C^{\infty}_c(\R)$. We refer to \cite[Section A.10]{S05} for details on the above metric and to the proof of \cite[Lemma 8.5]{S05} for the equi-continuity condition. We conclude by proving the compact containment condition. Define 
		$$K:=\Big\{ \nu \in M\ s.t. \ \exists \, k\in\N \ s.t.  \ \nu[-\ell,\ell]\le A(2\ell+1)\ell^2   \ \forall \ \ell\in[k,\infty]\cap \N \Big\}$$
		with $A>0$ such that $\frac{C\pi}{6A}<\epsilon$.
		By \cite[Proposition A.25]{S05}, we have that $K$ is a pre-compact subset of $M$. Moreover, by the Markov inequality and Assumption \ref{assumption: initial condition}(ii), it follows that 
		\begin{align*}
		&Q_N(\bar K^c)\le \sum_{\ell\in\N}\P_{\mu_N}\left(  \mathsf X_i^N([-\ell,\ell])\ge A(2\ell+1)\ell^2  \right) \le \sum_{\ell\in\N}\frac{1}{ A(2\ell+1)\ell^2 } \E_{\mu_N}\left[   \mathsf X_i^N([-\ell,\ell])  \right] \\
		&=\sum_{\ell\in\N}\frac{1}{ A(2\ell+1)\ell^2 } \sum_{x\in [-\ell,\ell]\cap\frac{\Z}{N}}\E_{\mu_N}\left[ \eta_i(x, tN^2) \right]\le \sum_{\ell\in\N} \frac{1}{ A(2\ell+1)\ell^2 } \frac{2\ell N+1}{N}C\le \frac{C}{A} \sum_{\ell\in\N} \frac{1}{ \ell^2 }<\epsilon,
		\end{align*}
		from which it follows that $Q_N(\bar K)>1-\epsilon$ for any $N$.
	\end{proof}

	\begin{remark}{\bf [Total density]}
		{\rm (i) If $\rho_0,\rho_1$ are smooth enough and satisfy \eqref{eq:HDL IRW with slow particles}, then by taking extra derivatives we see that the total density $\rho:=\rho_0+\rho_1$ satisfies the \textit{thermal telegrapher equation} 
			\begin{equation}
			\label{eq: telegrapher}
			\partial_t \left(\partial_t \rho + 2\gamma\rho\right)
			= - \epsilon\Delta(\Delta\rho) + (1+\epsilon)\Delta\left(\partial_t \rho + \rho\right),
			\end{equation}
			which is second order in $\partial_t$ and fourth order in $\partial_x$ (see \cite{HA80I,HA80II} for a derivation). Note that \eqref{eq: telegrapher} shows that the total density does not satisfy the usual diffusion equation. This fact will be investigated in detail in the next section, where we will analyse the non-Fick property of $\rho$.\\
			(ii) If $\epsilon=1$, then \eqref{eq: telegrapher} simplifies to the \textit{heat equation} $\partial_t\rho =  \Delta\rho$.\\
			(iii) If $\epsilon=0$, then \eqref{eq: telegrapher} reads
			$$
			\partial_t\left(\partial_t \rho + 2\lambda\rho\right)
			= \Delta\left(\partial_t \rho + \rho\right),
			$$
			which is known as the \textit{strongly damped wave equation}. The term $\partial_t(2\lambda\rho)$ is referred to as frictional damping, the term $\Delta(\partial_t \rho)$ as Kelvin-Voigt damping (see \cite{CCT17}).} \hfill $\spadesuit$
	\end{remark}
	
	
	\subsection{Existence, uniqueness and representation of the solution}
	\label{ss.repr}
	
	The existence and uniqueness of a continuous-time solution $(\rho_0(t),\, \rho_1(t))$ with values in $C_b(\R,\sigma)$ of the system in \eqref{eq:HDL IRW with slow particles} can be proved by standard Fourier analysis. Below we recall some known results that have a more probabilistic interpretation.
	
	
	\paragraph{Stochastic representation of the solution.}
	
	The system in \eqref{eq:HDL IRW with slow particles} fits in the realm of switching diffusions (see e.g.\ \cite{YZ10}), which are widely studied in the mathematical finance literature. Indeed, let $\{i_t\colon\,t \geq 0\}$ be the pure jump process on state space $I=\{0,1\}$ that switches at rate $\Upsilon$, whose generator acting on bounded functions $g\colon\,I \to \R$ is 
	$$
	(Ag)(i) := \Upsilon(g(1-i)-g(i)), \qquad i \in I.
	$$
	Let $\{X_t\colon\,t \geq 0\}$ be the stochastic process on $\R$ solving the stochastic differential equation
	$$
	\dd X_t = \psi(i_t)\,\dd W_t,
	$$
	where $W_t=B_{2t}$ with $\{B_t\colon\, t \geq 0\}$ standard Brownian motion, and $\psi\colon\,I\to \{D_0,D_1\}$ is given by 
	$$
	\psi := D_0\,\boldsymbol1_{\{0\}} + D_1\,\boldsymbol1_{\{1\}},
	$$
	with $D_0=1$ and $D_1=\epsilon$ in our setting. Let $\mathcal{L} = \mathcal{L}_{\epsilon,\Upsilon}$ be the generator defined by 
	$$
	(\mathcal L f)(x,i) := \lim_{t\downarrow 0} \frac{1}{t}\,\E_{x,i}[f(X_t,i_t)-f(x,i)]
	$$
	for  $f\colon\,\R \times I \to \R$ such that $f(\cdot,i)\in C^2_0(\R)$. Then, via a standard computation (see e.g.\ \cite[Eq.(4.4)]{F85}),
	it follows that
	$$
	\begin{aligned}
	(\mathcal L f)(x,i) &=\psi(i)(\Delta f)(x,i)+ \Upsilon[f(x,1-i)-f(x,i)]\\[0.2cm]
	&=\begin{cases}
	\Delta f(x,0)+\Upsilon\,[f(x,1)-f(x,0)],
	&i=0,\\
	\epsilon \Delta f(x,1)+\Upsilon\,[f(x,0)-f(x,1)], 
	&i=1.
	\end{cases}
	\end{aligned}
	$$
	We therefore have the following result that corresponds to \cite[Chapter 5, Section 4, Theorem 4.1]{F85}(see also \cite[Theorem 5.2]{YZ10}).
	
	\begin{theorem}{{\bf [Stochastic representation of the solution]}}
		Suppose that $\bar \rho_i\colon\,\R\to\R$ for $i\in I$ are continuous and bounded. Then \eqref{eq:HDL IRW with slow particles} has a unique solution given by 
		$$
		\rho_i(x,t)=\E_{(x,i)}[ \bar\rho_{i_t}(X_t)], \qquad i \in I.
		$$
	\end{theorem}
	
	Note that if there is only one particle in the system \eqref{eq: generator IRW with sleeping particles}, then we are left with a single random walk, say $\{Y_t\colon\,t \geq 0\}$, whose generator, denoted by $\mathsf A$, acts on bounded functions $f\colon\,\Z\times I\to \R$ as
	$$
	(\mathsf Af)(y,i) = \psi(i)\left[\sum_{z\sim y} [f(z,i)-f(y,i)]\right] + \Upsilon\,[f(y,1-i)-f(y,i)].
	$$
	After we apply the generator to the function $f(y,i)=y$, we get 
	$$
	\mathsf (Af)(y,i)=0,
	$$
	i.e., the position of the random walk is a martingale. Computing the quadratic variation via the carr\'{e} du champ, we find
	$$
	\mathsf A(Y_t^2) =\psi(i_t) [(Y_t+1)^2-Y_t^2] + \psi(i_t)[(Y_t-1)^2-Y_t^2] = 2\psi(i_t).
	$$
	Hence the predictable quadratic variation is given by 
	$$
	\int_0^t \dd s\,2\psi(i_s).
	$$
	Note that for $\epsilon=0$ the latter equals the total amount of time the random walk is not \textit{dormant} up to time $t$.
	
	When we diffusively scale the system (scaling the reaction term was done at the beginning of Section \ref{s.hydro}), the quadratic variation becomes
	$$
	\int_0^{tN^2} \dd s\, \psi(i_{N,s}) = \int_0^{t} \dd r\,\psi(i_r).
	$$
	As a consequence, we have the following invariance principle: 
	\begin{itemize}
		\item[]
		Given the path of the process $\{i_t\colon\, t\geq 0\}$,
		$$
		\lim_{N\to \infty} \frac{Y_{N^2t}}{N} = W_{\int_0^{t} \dd r\,\sqrt{\psi(i_r)}},
		$$
		where $W_t=B_{2t}$ with $\{B_t\colon\, t \geq 0\}$ is standard Brownian motion.
	\end{itemize}	
	Thus, if we knew the path of the process $\{i_r\colon\, r\ge 0\}$, then we could express the solution of the system in \eqref{eq:HDL IRW with slow particles} in terms of a time-changed Brownian motion. However, even though $\{i_r\colon\, r\ge 0\}$ is a simple flipping process, we cannot say much explicitly about the random time $\int_0^{t} \dd r\,\sqrt{\psi(i_r)}$. We therefore look for a simpler formula, where the relation to a Brownian motion with different velocities is more explicit. We achieve this by looking at the resolvent of the generator $\mathcal L$. In the following, we denote by $\{S_t, \, t\ge 0\}$ the semigroup on $C_b(\R)$ of $\{W_{t}:\,t\ge 0\}$.
	
	\begin{proposition}{\bf [Resolvent]}
		Let $f\colon\,\R\times I\to \R$ be a bounded and smooth function. Then, for $\lambda>0$, $\epsilon\in(0,1]$ and $i\in I$,
		\begin{equation}
		\begin{aligned}
		\label{eq: resolvent epsilon >0}
		&(\lambda\boldsymbol I-\mathcal L)^{-1}	f(x,i)\\
		&=\int_0^\infty \dd t\, \frac{1}{\epsilon^i}\, \ee^{-\tfrac{1+\epsilon}{\epsilon}\ell(\Upsilon,\lambda)t}
		\left(\cosh(tc_\epsilon(\Upsilon,\lambda))+  \frac{1-\epsilon}{\epsilon}\ell_\epsilon(\Upsilon,\lambda)\frac{\sinh(tc_\epsilon(\Upsilon,\lambda))}{c_\epsilon(\lambda)}\right)\,(S_tf(\cdot,i))(x) \\
		&+  \int_0^\infty  \dd t \,\ee^{-\tfrac{1+\epsilon}{\epsilon}\ell(\Upsilon,\lambda)t}
		\left(\frac{\Upsilon}{\epsilon}\sinh(tc_\epsilon(\Upsilon,\lambda))\right)\,(S_tf(\cdot,1-i))(x)
		,
		\end{aligned}
		\end{equation}
		where $	c_\epsilon(\Upsilon,\lambda)= \sqrt{\left(\tfrac{1-\epsilon}{\epsilon}\right)^2\ell(\Upsilon,\lambda)^2+\frac{\Upsilon^2}{\epsilon}}$ and $\ell(\Upsilon,\lambda)=\frac{\Upsilon+\lambda}{2}$, while for $\epsilon=0$, 
		\begin{align}
		\label{eq: resolvent epsilon =0}
		(\lambda \boldsymbol I-\mathcal L)^{-1}f(x,i)=\int_0^\infty \dd t\,\ee^{-\lambda\tfrac{2\Upsilon+\lambda}{\Upsilon+\lambda}t}\left(  \left(\tfrac{\Upsilon}{\lambda+\Upsilon}\right)^i(S_tf(\cdot,0))(x)+  \left(\tfrac{\Upsilon}{\Upsilon+\lambda}\right)^{i+1}(S_tf(\cdot,1))(x) \right).
		\end{align}
	\end{proposition}
	\begin{proof}
		The proof is split into two parts.
		
		\medskip\noindent
		\underline{Case $\epsilon>0$.} 
		We can split the generator $\mathcal L$ as
		$$
		\mathcal L = \psi(i) \tilde {\mathcal L}
		= \psi(i)\left(\Delta + \frac{1}{\psi(i)} A\right)
		= \psi(i)(\Delta + \tilde A),
		$$
		i.e., we decouple $X_t$ and $i_t$ in the action of the generator. We can now use the Feynman-Kac formula to express the resolvent of the operator $\mathcal L$ in terms of the operator $\tilde{\mathcal{L}}$. Denoting by $\tilde\E$ the expectation of the process with generator $\tilde{\mathcal L}$, we have, for $\lambda \in \R$, 
		$$
		(\lambda \boldsymbol I-\mathcal L)^{-1} f(x,i)
		= \left(\frac{\lambda \boldsymbol I}{\psi}-\tilde{\mathcal L} \right)^{-1}\left( \frac{f(x,i)}{\psi(i)}\right)\\
		= \int_0^\infty \dd t\,\,\tilde \E_{(x,i)} \left[\ee^{-\int_{0}^t  \dd s\,
			\frac{\lambda}{\psi(i_s)} }\,\, \frac{f(X_t,i_t)}{\psi(i_t)} \right],
		$$
		and by the decoupling of $X_t$ and $i_t$ under $\tilde{\mathcal L}$, we get 
		$$
		\begin{aligned}
		&(\lambda \boldsymbol I-\mathcal L)^{-1} f(x,i)\\
		&= \int_0^\infty \dd t\,\,\tilde\E_{i}\left[\ee^{-\lambda\int_{0}^t \dd s\,\frac{1}{\psi(i_s)}}\,\, 
		\frac{\boldsymbol 1_{\{0\}}(i_t)}{\psi(i_t)} \right] (S_{t}f(\cdot,0))(x) 
		+  \int_0^\infty \dd t\,\,\tilde\E_{i} \left[\ee^{-\lambda\int_{0}^t  \frac{1}{\psi(i_s)}}\,\, 
		\frac{\boldsymbol 1_{\{1\}}(i_t)}{\psi(i_t)} \right] (S_{t}f(\cdot,1))(x)\\  
		&= \int_0^\infty \dd t\,\,\tilde \E_{i} \left[\ee^{-\lambda\int_{0}^t \dd s\, \frac{1}{\psi(i_s)}}\,\, 
		\boldsymbol 1_{\{0\}}(i_t) \right](S_{t}f(\cdot,0))(x)
		+ \frac{1}{\epsilon} \int_0^\infty \dd t\,\,\tilde \E_{i} \left[\ee^{-\lambda\int_{0}^t \dd s\,\frac{1}{\psi(i_s)}}\,\, 
		\boldsymbol 1_{\{1\}}(i_t) \right] (S_{t}f(\cdot,1))(x).  
		\end{aligned}
		$$
		Defining  
		$$
		A:=
		\left[\begin{array}{c c c c}
		-\Upsilon & \Upsilon  \\
		\Upsilon& -\Upsilon \\
		\end{array}
		\right],
		\qquad 
		\psi_\epsilon:=
		\left[\begin{array}{c c c c}
		1 & 0  \\
		0 & \epsilon \\
		\end{array}
		\right],
		$$
		and using again the Feynman-Kac formula, we have
		$$
		(\lambda \boldsymbol I-\mathcal L)^{-1} \left[\begin{array}{c c c c}
		f(x,0)  \\
		f(x,1)  \\
		\end{array}
		\right] = \int_0^\infty \dd t\,K_{\epsilon}(t,\lambda) \left[\begin{array}{c c c c}
		(S_tf(\cdot,0))(x)\\
		(S_tf(\cdot,1))(x) \\
		\end{array}
		\right]
		$$
		with 
		$K_{\epsilon}(t,\lambda) = \ee^{t\psi_\epsilon^{-1}(-\lambda \boldsymbol I +A)}\psi_\epsilon^{-1}$.
		
		Using the explicit formula for the exponential of a $2\times 2$ matrix (see e.g.\ \cite[Corollary 2.4]{BS93}), we obtain
		\small
		\begin{equation}
		\ee^{t\psi_\epsilon^{-1}(-\lambda \boldsymbol I +A)}=\ee^{-\tfrac{1+\epsilon}{\epsilon}\ell(\Upsilon,\lambda)t}\left[\begin{array}{c c c c}
		\cosh(tc_\epsilon(\Upsilon,\lambda))+  \tfrac{1-\epsilon}{\epsilon}\ell(\Upsilon,\lambda)\tfrac{\sinh(tc_\epsilon(\Upsilon,\lambda))}{c_\epsilon(\Upsilon,\lambda)}& \Upsilon\tfrac{\sinh(tc_\epsilon(\Upsilon,\lambda))}{c_\epsilon(\Upsilon,\lambda)}  \\
		\tfrac{\Upsilon}{\epsilon}\tfrac{\sinh(tc_\epsilon(\Upsilon,\lambda))}{c_\epsilon(\Upsilon,\lambda)} &\cosh(tc_\epsilon(\Upsilon,\lambda))-  \tfrac{1-\epsilon}{\epsilon}\ell(\Upsilon,\lambda)\tfrac{\sinh(tc_\epsilon(\Upsilon,\lambda))}{c_\epsilon(\Upsilon,\lambda)}\\
		\end{array}
		\right]
		\end{equation}
		\normalsize
		with 
		$
		c_\epsilon(\Upsilon,\lambda)= \sqrt{\left(\tfrac{1-\epsilon}{\epsilon}\right)^2\ell(\Upsilon,\lambda)^2+\frac{\Upsilon^2}{\epsilon}}
		$ and $\ell(\Upsilon,\lambda)=\frac{\Upsilon+\lambda}{2}$, from which we obtain \eqref{eq: resolvent epsilon >0}.
		
		\medskip\noindent
		\underline{Case $\epsilon=0$.} 
		We derive $K_0(t,\lambda)$ by taking the limit  $\epsilon\downarrow0$ in the previous expression, i.e., $K_0(t,\lambda)=\lim_{\epsilon \downarrow 0} K_\epsilon(t,\lambda)$. We thus have that $K_0(t,\lambda)$ is equal to
		\small
		\begin{align*}
		&\lim_{\epsilon \downarrow 0}\ee^{-\tfrac{1+\epsilon}{\epsilon}\ell(\Upsilon,\lambda)t}\left[\begin{array}{c c c c}
		\cosh(tc_\epsilon(\Upsilon,\lambda))+  \tfrac{1-\epsilon}{\epsilon}\ell(\Upsilon,\lambda)\tfrac{\sinh(tc_\epsilon(\Upsilon,\lambda))}{c_\epsilon(\Upsilon,\lambda)}& \tfrac{\Upsilon}{\epsilon}\tfrac{\sinh(tc_\epsilon(\Upsilon,\lambda))}{c_\epsilon(\Upsilon,\lambda)}  \\
		\tfrac{\Upsilon}{\epsilon}\tfrac{\sinh(tc_\epsilon(\Upsilon,\lambda))}{c_\epsilon(\Upsilon,\lambda)} &\tfrac{1}{\epsilon}\cosh(tc_\epsilon(\Upsilon,\lambda))-  \tfrac{1-\epsilon}{\epsilon^2}\ell(\Upsilon,\lambda)\tfrac{\sinh(tc_\epsilon(\Upsilon,\lambda))}{c_\epsilon(\Upsilon,\lambda)}\\
		\end{array}
		\right]\\
		&=\ee^{-\lambda\tfrac{2\Upsilon+\lambda}{\Upsilon+\lambda}t}\left[\begin{array}{c c c c}
		1& \tfrac{\Upsilon}{\Upsilon+\lambda}  \\
	\tfrac{\Upsilon}{\Upsilon+\lambda}  &\left(\tfrac{\Upsilon}{\Upsilon+\lambda}\right)^2 \\
		\end{array}
		\right],
		\end{align*}
		\normalsize
		from which \eqref{eq: resolvent epsilon =0} follows.
	\end{proof}

\begin{remark}{\bf [Symmetric layers]}
		{\rm Note that for $\epsilon=1$ we have
			\begin{align*}
			(\lambda \boldsymbol I-\mathcal L)^{-1} f(x,i) =\int_0^\infty \dd t\, \ee^{-\lambda t}  
			\left( \tfrac{1+\ee^{-2\Upsilon t}}{2}\,(S_tf(\cdot,i))(x)+  \tfrac{1-\ee^{-2\Upsilon t}}{2}\,(S_tf(\cdot,1-i))(x)\right).
			\end{align*}
		}\hfill $\spadesuit$
	\end{remark}
	
	We conclude this section by noting that the system in \eqref{eq:HDL IRW with slow particles} was studied in detail in \cite{HA80I,HA80II}. By taking Fourier and Laplace transforms and inverting them, it is possible to deduce explicitly the solution, which is expressed in terms of solutions to the classical heat equation. More precisely, using formula \cite[Eq.2.2]{HA80II}, we have that 
	\begin{multline}
	\rho_0(x,t)=\ee^{-\Upsilon t}\,(S_t\,\bar\rho_0)(x)
	+\frac{\Upsilon}{1-\epsilon}\ee^{-\Upsilon t}\int_{\epsilon t}^t \dd s\, 
	\left( \left(\frac{s-\epsilon t}{t-s}\right)^{1/2}I_1(\upsilon(s))\,(S_s\,\bar\rho_0)(x)
	+I_0(\upsilon(s))\,(S_s\,\bar\rho_1)(x)\right)
	\end{multline}
	and 
	\begin{multline}\allowdisplaybreaks
	\rho_1(x,t)=\ee^{-\Upsilon t}(S_{\epsilon t}\,\bar\rho_1)(x)\\+\frac{\Upsilon}{1-\epsilon}
	\ee^{-\sigma t}\int_{\epsilon t}^t \dd s\,\left( \left(\frac{s-\epsilon t}{t-s}\right)^{-1/2}I_1(\upsilon(s))\,
	(S_s\,\bar\rho_1)(x)+I_0(\upsilon(s))\,(S_s\,\bar\rho_0)(x) \right),
	\end{multline}
	where 
	$\upsilon(s)=\frac{2\Upsilon}{1-\epsilon}((t-s)(s-\epsilon t))^{1/2}$, and $I_0(\cdot)$ and $I_1(\cdot)$ are the modified Bessel functions.

	
	\section{The system with boundary reservoirs}
	\label{s.bdres}
		In this section we consider a finite version of the switching interacting particle systems introduced in Definition~\ref{def:system} to which boundary reservoirs are added. Section~\ref{ss.modelres} defines the model. Section~\ref{ss.dualres} identifies the dual and the stationary measures. Section~\ref{ss.profileres} derives the non-equilibrium density profile, both for the microscopic system and the macroscopic system, and offers various simulations. In Section~\ref{ss.statcurrent} we compute the stationary horizontal current of slow and fast particles both for the microscopic system and the macroscopic system. Section~\ref{ss.uphillres} shows that in the macroscopic system, for certain choices of the rates, there can be a flow of particles uphill, i.e., against the gradient imposed by the reservoirs.	 Thus, as a consequence of the competing driving mechanisms of slow and fast particles, we can have a flow of particles from the side with lower density to the side with higher density.
	
	\subsection{Model}
	\label{ss.modelres}
	
	We consider the same system as in Definition~\ref{def:system}, but restricted to $V:=\{1,\ldots,N\}\subset\Z$. In addition, we set $\hat{V} := V\cup \{L,R\}$ and attach a \textit{left-reservoir} to $L$ and a \textit{right-reservoir} to $R$, both for fast and slow particles. To be more precise, there are four reservoirs (see Fig.~\ref{fig:sleepingres}): 	
	\begin{figure}[htbp]
		\begin{subfigure}{\linewidth}
			\centering
			\includegraphics[width=0.63\linewidth,height=0.2\textheight]{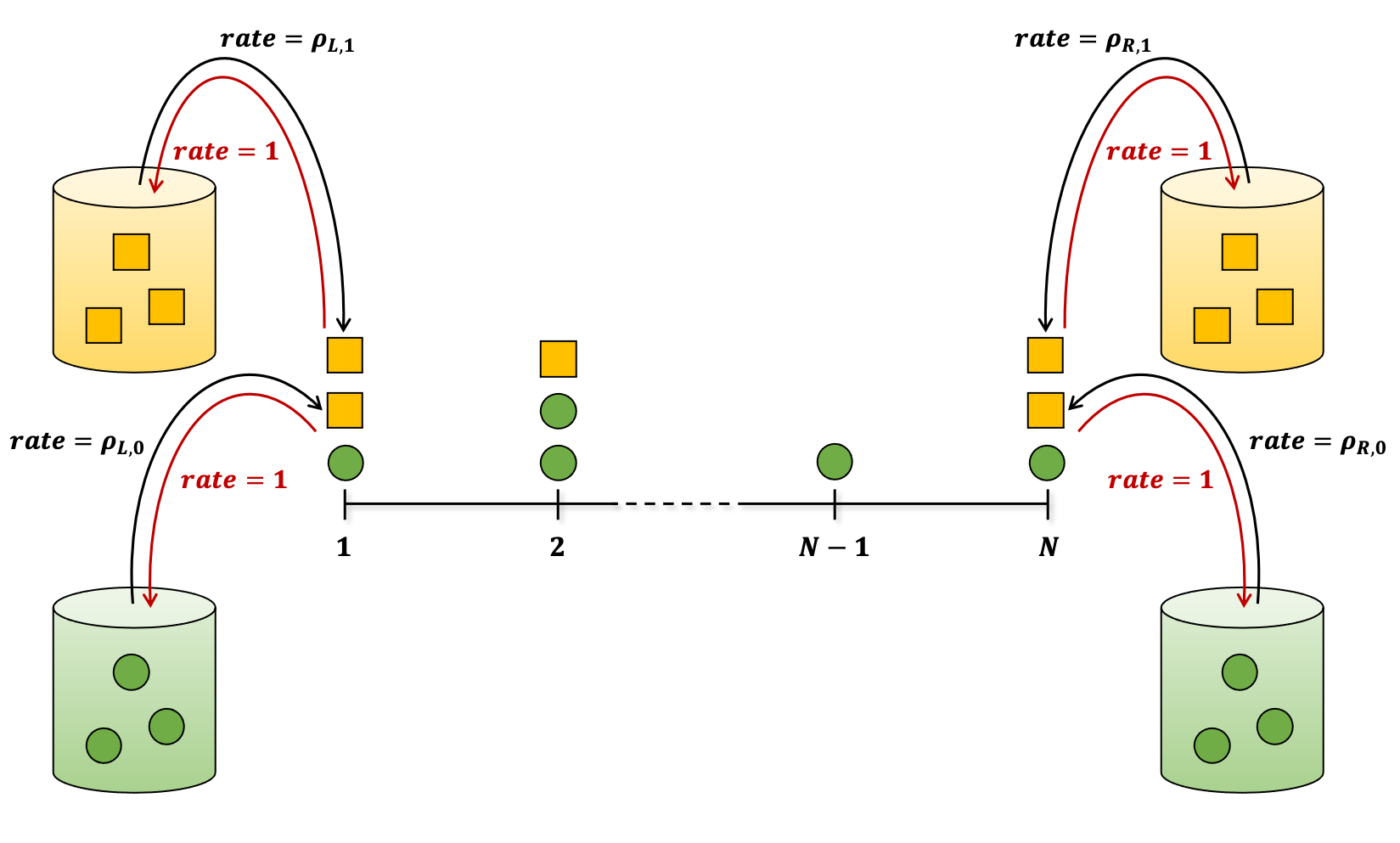}
			\caption{Representation via \emph{slow} and \emph{fast} particles moving on $V$.}
		\end{subfigure}
		\begin{subfigure}{\linewidth}
			\centering
			\includegraphics[width=0.63\linewidth,height=0.212\textheight]{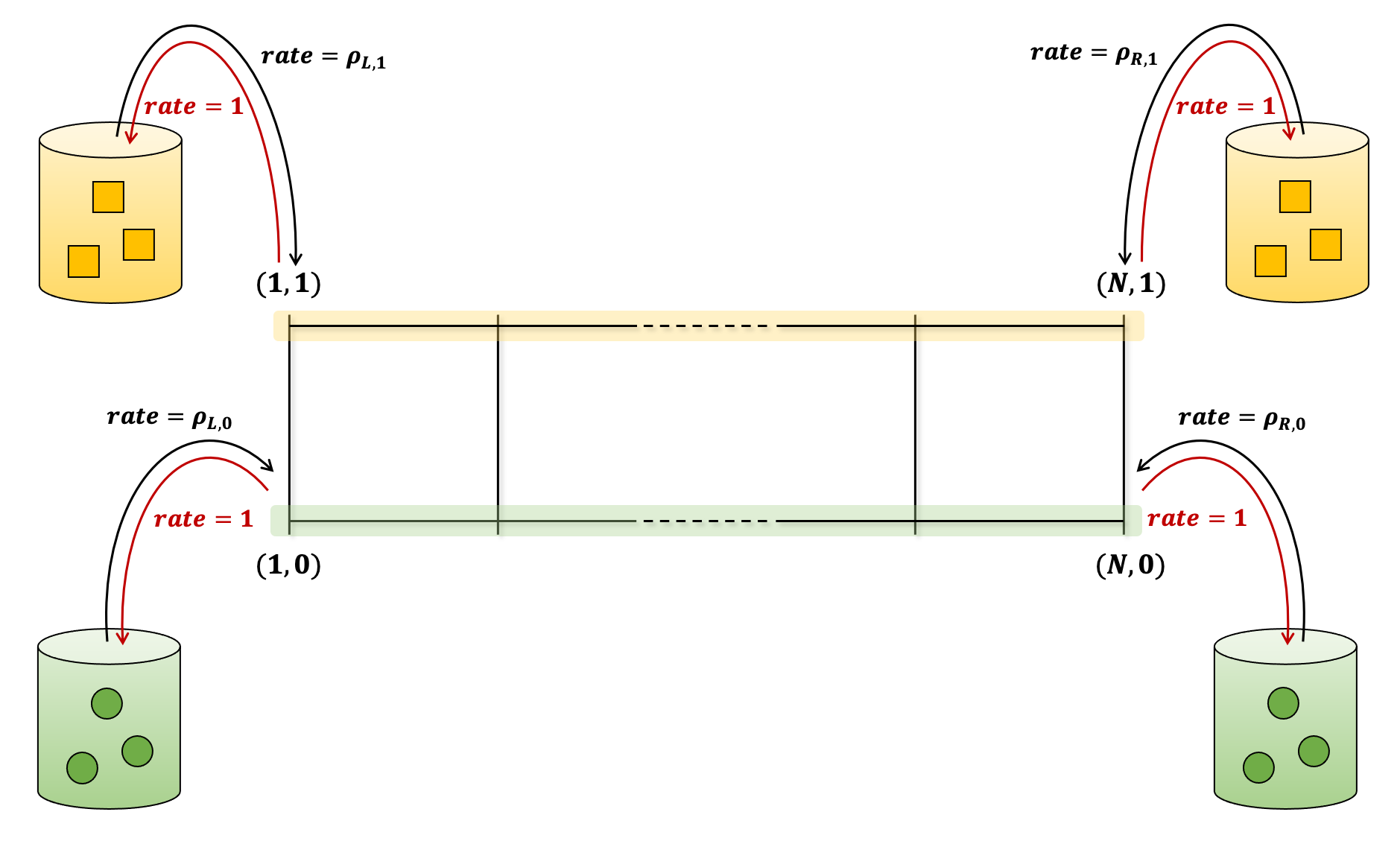}
			\caption{Representation via particles moving on $V\times I$.}
		\end{subfigure}
		\caption{Case $\sigma=0$, $\epsilon>0$ with boundary reservoirs: two equivalent representations.}
		\label{fig:sleepingres}
	\end{figure}
	
	\begin{itemize}
		\item[i)]
		For the fast particles, a left-reservoir at $L$ injects fast particles at $x=1$ at rate $\rho_{L,0}(1+\sigma\eta_0(1,t))$ and a right-reservoir at $R$ injects fast particles at $x=N$ at rate $\rho_{R,0}(1+\sigma\eta_0(N,t))$. The left-reservoir absorbs fast particles at rate $1+\sigma\rho_{L,0}$, while the right-reservoir does so at rate $1+\sigma\rho_{R,0}$.
		\item [ii)]
		For the slow particles, a left-reservoir at $L$ injects slow particles at $x=1$ at rate $\rho_{L,1}(1+\sigma\eta_1(1,t))$ and a right-reservoir at $R$ injects slow particles at $x=N$ at rate $\rho_{R,1}(1+\sigma\eta_1(N,t))$. The left-reservoir absorbs fast particles at rate $1+\sigma\rho_{L,1}$, while the right-reservoir does so at rate $1+\sigma\rho_{R,1}$.
	\end{itemize}
	Inside $V$, the particles move as before.

	For $i\in I$, $x \in V$ and $t\geq 0$, let $\eta_{i}(x,t)$ denote the number of particles in layer $i$ at site $x$ at time $t$. For $\sigma\in\{-1,0,1\}$,  the Markov process $\{\eta(t)\colon\,t \geq 0\}$ with
	$$
	\eta(t) = \{\eta_{0}(x,t),\eta_{1}(x,t)\}_{x \in V}
	$$
	has state space
	$$\mathcal X = \begin{cases} 
	\{0,1\}^V\times \{0,1\}^V, &\sigma=-1,\\ 
	\N_0^{V}\times \N_0^{V}, &\sigma=0,1,\
	\end{cases}
	$$
	and generator
	\begin{equation}
	\label{eq: generator IRW with sleeping particles + reservoirs}
	L  := L_{\epsilon,\gamma,N}= L^{\mathrm{bulk}} + L^{\mathrm{res}}
	\end{equation}
	with
	\begin{align}
	&L^{\mathrm{bulk}}:= L^{\mathrm{bulk}}_0 +\epsilon L^{\mathrm{bulk}}_1
	+ \gamma L^{\mathrm{bulk}}_{0 \updownarrow 1}
	\end{align}
	acting on bounded cylindrical functions $f\colon\,\mathcal X\to \R$ as
	\begin{align*}
	(L^{\mathrm{bulk}}_0 f)(\eta) 
	=&\sum_{x=1}^{N-1} \Big\{\eta_0(x)(1+\sigma\eta_0(x+1))\,\big[f(\eta_0-\delta_{x}
	+\delta_{x+1},\eta_1)-f(\eta_0,\eta_1)\big]\\
	& +\eta_0(x+1)(1+\sigma\eta_0(x))\,\big[f(\eta_0-\delta_{x+1}+\delta_{x},\eta)-f(\eta_0,\eta_1)\big]\Big\},
	\end{align*}
	\begin{align*}
	(L^{\mathrm{bulk}}_1 f)(\eta)=& \sum_{x=1}^{N-1} \Big\{\eta_1(x)(1+\sigma\eta_1(x+1))\,
	\big[f(\eta_0,\eta_1-\delta_{x}+\delta_{x+1})-f(\eta_0,\eta_1)\big]\\
	&+ \eta_1(x+1)(1+\sigma\eta_1(x))\,\big[f(\eta_0,\eta_1-\delta_{x+1}+\delta_{x})-f(\eta_0,\eta_1))\big]\Big\},
	\end{align*}
	\begin{align*}
	(L^{\mathrm{bulk}}_{0 \updownarrow 1}f)(\eta) =& \sum_{x=1}^N \Big\{\eta_0(x)(1+\sigma\eta_1(x))\,
	\big[f(\eta_0-\delta_{x},\eta_{1}+\delta_{x})-f(\eta_0,\eta_1)\big]\\
	&+ \eta_1(x)(1+\sigma\eta_0(x))\,\big[f(\eta_{0}+\delta_{x},\eta_{1}-\delta_{x})-f(\eta_{0},\eta_{1}))\big]\Big\},
	\end{align*}
	and
	\begin{align}
	L^{\mathrm{res}}:=L^{\mathrm{res}}_0+ L^{\mathrm{res}}_1
	\end{align}
	acting as
	\begin{align*}
	&(L^{\mathrm{res}}_0f)(\eta)= \eta_0(1)(1+\sigma\rho_{L,0})\,\big[f(\eta_0-\delta_{1},\eta_{1})-f(\eta_0,\eta_{1})\big]\\
	&\quad + \rho_{L,0}(1+\sigma\eta_0(1))\,\big[f(\eta_0+\delta_{1},\eta_{1})-f(\eta_0,\eta_{1})\big]\\[0.2cm]
	&\nonumber+\eta_0(N)(1+\sigma\rho_{R,0})\,\big[f(\eta_0-\delta_{N},\eta_{1})-f(\eta_0,\eta_{1})\big]
	+ \rho_{R,0}(1+\sigma\eta_0(N))\,\big[f(\eta_0+\delta_{N},\eta)-f(\eta_0,\eta_{1})\big],
	\end{align*}
	\begin{align*}
	& (L^{\mathrm{res}}_1f)(\eta)=\eta_1(1)(1+\sigma\rho_{L,1})\,\big[f(\eta_{0},\eta_1-\delta_{1})-f(\eta_0,\eta_{1})\big]\\
	&\quad +\rho_{L,1}(1+\sigma\eta_1(1))\,\big[f(\eta_{0},\eta_1+\delta_{1})-f(\eta_{0},\eta_{1})\big]\\[0.2cm]
	&+\eta_1(N)(1+\sigma\rho_{R,1})\,\big[f(\eta_{0},\eta_1-\delta_{N})-f(\eta_0,\eta_{1})\big]
	+ \rho_{R,1}(1+\sigma\rho_{R,N})\,\big[f(\eta_{0},\eta_1+\delta_{N})-f(\eta_0,\eta_{1})\big].
	\end{align*}
	
	
	\subsection{Duality and stationary measures}
	\label{ss.dualres}
	In \cite{CGGR13} it was shown that the partial exclusion process, a system of independent random walks and the symmetric inclusion processes on a finite set $V$, coupled with proper left and right reservoirs, are dual to the same particle system but with the reservoirs replaced by absorbing sites. As remarked in  \cite{FRSpr}, the same result holds for more general geometries, consisting of inhomogeneous rates (site and edge dependent), and for many proper reservoirs. Our model is a particular instance of the case treated in \cite[Remark 2.2]{FRSpr}), because we can think of the rate as conductances attached to the edges.
	
	More precisely, we consider the system where particles jump on two copies of 
	$$
	\hat V:= V\cup \{L,R\}
	$$ 
	and follow the same dynamics as before in $V$, but with the reservoirs at $L$ and $R$ absorbing. We denote by $\xi$ the configuration 
	$$
	\xi=(\xi_0,\xi_1):=(\{\xi_0(x)\}_{x\in \hat V},\{\xi_1(x)\}_{x\in \hat V}),
	$$
	where $\xi_i(x)$ denotes the number of particles at site $x$ in layer $i$. The state space is $\hat{\mathcal X} = \N_0^{\hat V}\times \N_0^{\hat V}$, and the generator is
	\begin{equation}
	\label{eq: dual generator IRW with sleeping particles + reservoirs}
	\hat L := \hat{L}_{\epsilon,\gamma,N} = \hat L^{\mathrm{bulk}} + \hat L^{L,R}
	\end{equation}
	with
	\begin{align*}
	&\hat L^{\mathrm{bulk}}:= \hat L_0^{\mathrm{bulk}}+\epsilon \hat L_1^{\mathrm{bulk}} 
	+ \gamma \hat L^{\mathrm{bulk}}_{0 \updownarrow 1}
	\end{align*}
	acting on cylindrical functions $f\colon\,\mathcal X \to \R$ as
	\begin{align*}
	&(\hat L_0^{\mathrm{bulk}}f)(\xi)=\sum_{x=1}^{N-1} \Big\{\xi_0(x)(1+\sigma\xi_0(x+1))\,
	\big[f(\xi_0-\delta_{x}+\delta_{x+1},\xi_1)-f(\xi_0,\xi_1)\big]\\
	&\quad+ \xi_0(x+1)(1+\sigma\xi_0(x))\,\big[f(\xi_0-\delta_{x+1}+\delta_{x},\xi_1)-
	f(\xi_0,\xi_1)\big]\Big\},\\
	&(\hat L_1^{\mathrm{bulk}}f)(\xi)=\sum_{x=1}^{N-1} \Big\{\xi_1(x)(1+\sigma\xi_1(x+1))\,
	\big[f(\xi_0,\xi_1-\delta_{x}+\delta_{x+1})-f(\xi_0,\xi_1)\big]\\
	&\quad+ \xi_1(x+1)(1+\sigma\xi_1(x))\,\big[f(\xi_0,\xi_1-\delta_{x+1}+\delta_{x})-f(\xi_0,\xi_1)\big]\Big\},\\
	&(\hat L^{\mathrm{bulk}}_{0 \updownarrow 1}f)(\eta) =\sum_{x=1}^N \Big\{\xi_0(x)(1+\sigma\xi_1(x))\,
	\big[f(\xi_0-\delta_{x},\xi_{1}+\delta_{x})-f(\xi_0,\xi_1)\big]\\
	&\quad+ \xi_1(x)(1+\sigma\xi_0(x))\,\big[f(\xi_{0}+\delta_{x},\xi_{1}-\delta_{x})-f(\xi_{0},\xi_{1})\big]\Big\},
	\end{align*}
	and
	\begin{align*}
	\hat L^{L,R}=\hat L_0^{L,R}+\hat L_1^{L,R}
	\end{align*}
	acting as
	\begin{align*}
	&(\hat L_0^{L,R}f)(\xi)= \xi_0(1)\,\big[f(\xi_0-\delta_{1},\xi_{1})-f(\xi_0,\xi_{1})\big] 
	+ \xi_0(N)\,\big[f(\xi_0-\delta_{N},\xi_1)-f(\xi_0,\xi_1)\big],\\[0.2cm]
	&(\hat L_1^{L,R}f)(\xi)= \xi_1(1)\,\big[f(\xi_0,\xi_1-\delta_{1})-f(\xi_0,\xi_1)\big]
	+ \xi_1(N)\,\big[f(\xi_0,\xi_1-\delta_{N})-f(\xi_0,\xi_1)\big].
	\end{align*}
	
	\begin{proposition}{{\bf [Duality]} {\rm \cite[Theorem 4.1]{CGGR13} and \cite[Proposition 2.3]{FRSpr}}}
		\label{proposition: duality absorbing sites}
		The Markov processes 
		\begin{align*}
		&\{\eta(t)\colon\,t \geq 0\}, \qquad \eta(t) = \{\eta_0(x,t),\eta_{1}(x,t)\}_{x \in V},\\
		&\{\xi(t)\colon\,t \geq 0\}, \qquad \xi(t) = \{\xi_0(x,t),\xi_{1}(x,t)\}_{x \in \hat V},
		\end{align*}
		with generators $L$ in \eqref{eq: generator IRW with sleeping particles + reservoirs} and $\hat L$ in \eqref{eq: dual generator IRW with sleeping particles + reservoirs} are dual. Namely, for all configurations $\eta\in\mathcal X$, $\xi\in\hat{\mathcal X}$ and times $t\ge 0$, 
		$$ \E_\eta[ D(\xi,\eta_t)]= \E_\xi[ D(\xi_t,\eta)],$$
		where the duality function is given by
		$$
		D(\xi,\eta):= \left(\prod_{i\in I} d_{(L,i)}(\xi_i(L))\right) 
		\times \left(\prod_{x\in V} d(\xi_i(x),\eta_i(x))\right) 
		\times \left(\prod_{i\in I} d_{(R,i)}(\xi_i(R))\right),
		$$
		where,
		for $k,n\in \N$ and $i\in I$, $d(\cdot, \cdot)$ is given in \eqref{eq: single site duality} and
		$$ 
		d_{(L,i)}(k) = \left(\rho_{L,i}\right)^k, \qquad d_{(R,i)}(k) = \left(\rho_{R,i}\right)^k.
		$$
	\end{proposition}
	The proof boils down to checking that the relation 
	$$\hat L D(\cdot,\eta)(\xi)=LD(\xi,\cdot)(\eta)$$
	holds for any $\xi\in \mathcal X$ and $\xi \in\hat{\mathcal X}$, as follows from a rewriting of the proof of \cite[Theorem 4.1]{CGGR13}.
	
	
	\subsection{Non-equilibrium stationary profile}
	\label{ss.profileres}
	
	Also the existence and uniqueness of the non-equilibrium steady state has been established in \cite[Theorem 3.3]{FRSpr} for general geometries, and the argument in that paper can be easily adapted to our setting. 
	\begin{theorem}{{\bf [Stationary measure]} {\rm \cite[Theorem 3.3(a)]{FRSpr}\label{theorem:existence_uniqueness}}}
		For $\sigma\in\{-1,0,1\}$ there exists a unique stationary measure $\mu_{stat} \ $ for $\{\eta(t)\colon\,t \geq 0\}$.
		Moreover, for $\sigma=0$ and for any values of $\{\rho_{L,0}, \ \rho_{L,1}, \ \rho_{R,0}, \ \rho_{R,1}\}$, 
		\begin{equation}
		\label{eq:mustat=poisson}
		\mu_{stat} =\prod_{(x,i) \in V \times I} \nu_{(x,i)},
		\qquad \nu_{(x,i)} = \mathrm{Poisson}(\theta_{(x,i)}),
		\end{equation}
		while, for $\sigma\in\{-1,1\}$, $\mu_{stat}$ is in general not in product form, unless $\rho_{L,0}= \rho_{L,1}=\rho_{R,0}= \rho_{R,1}$, for which 
		\begin{equation}\label{eq: stat meas SEP e SIP }
		\mu_{stat} =\prod_{(x,i) \in V \times I} \nu_{(x,i), \theta},
		\end{equation}
		where $\nu_{(x,i), \theta}$ is given in \eqref{eq:marginals}.
	\end{theorem}
	\begin{proof}
		For $\sigma=-1$, the existence and uniqueness of the stationary measure is trivial by the irreducibility and the finiteness of the state space of the process. For $\sigma\in\{0,1\}$, recall from \cite[Appendix A]{FRSpr} that a probability measure $\mu$ on $\mathcal X$ is said to be tempered if it is characterized by the integrals $\big\{\E_\mu [D(\xi,\eta)]\,:\,\  \xi\in\hat{\mathcal X} \big\}$  and that if there exists a $\theta\in[0,\infty)$ such that $\E_\mu[ D(\xi,\eta)]\le \theta^{|\xi|}$ for any $\xi\in\hat{\mathcal X}$. By means of duality we have that, for any $\eta\in \mathcal X$ and $\xi\in \hat{\mathcal X}$,
		\begin{align}\label{eq: conto duality stat meas}
		\nonumber &\lim_{t\to\infty}\E_\eta[D(\xi,\eta_t)]=\lim_{t\to\infty}\hat{\E}_\xi[D(\xi_t,\eta)]\\
		&= \sum_{i_0=0}^{|\xi|}\sum_{i_{0,L}=0}^{i_0}\sum_{j_{1,L}=0}^{|\xi|-i_0} \rho_{L,0}^{i_{0,L}}\ \rho_{R,0}^{i_0-i_{0,L}}\ \rho_{L,1}^{i_{1,L}}\ \rho_{R,1}^{|\xi|-i_{0}-i_{1,L}} \\
		&\qquad \qquad \qquad \qquad\qquad \quad\times\hat{\P}_\xi\left( \xi_\infty=  i_{0,L}\delta_{(L,0)}+(i_0-i_{0,L})\delta_{(R,0)} + i_{1,L} \delta_{(L,1)}+ (|\xi|-i_{0}-i_{1,L})  \delta_{(R,1)}\right),
		\end{align}
		from which we conclude that $\lim_{t\to\infty}\E_\eta[D(\xi,\eta_t)]\le \max\{\rho_{L,0}, \rho_{R,0}, \rho_{L,1}, \rho_{R,1} \}^{|\xi|}$. Let $\mu_s$ be the unique tempered probability measure such that for any $\xi\in\hat{\mathcal X},$ $\E_{\mu_{stat}}[D(\xi,\eta)]$ coincides with \eqref{eq: conto duality stat meas}. From the convergence of the marginal moments in \eqref{eq: conto duality stat meas}  we conclude that, for any $f:\mathcal X\to\R$ bounded and for any $\eta\in \mathcal X$, 
		$$\lim_{t\to \infty}\E_\eta[f(\eta_t)]=\E_{\mu_{stat}}[f(\eta)].$$
		Thus, a dominated convergence argument yields that for any probability measure $\mu$ on $\mathcal X$,
		$$\lim_{t\to \infty}\E_\mu[f(\eta_t)]=\E_{\mu_{stat}}[f(\eta)],$$ 
		giving that $\mu_{stat}$ is the unique stationary measure.
		The explicit expression in \eqref{eq:mustat=poisson} and \eqref{eq: stat meas SEP e SIP } follows from similar computations as in \cite{CGGR13}, while, arguing by contradiction as in the proof of \cite[Theorem 3.3]{FRSpr}, we can show that the two-point truncated correlations are non-zero for $\sigma\in\{-1,1\}$  whenever at least two  reservoir parameters are different.
	\end{proof}

\subsubsection{Stationary microscopic profile and absorption probability}

In this section we provide an explicit expression for the stationary microscopic density of each type of particle. To this end, let $\mu_{stat}$ be the unique non-equilibrium stationary measure of the process 
$$
\{\eta(t)\colon\,t \geq 0\}, \qquad \eta(t) := \{\eta_0(x,t),\eta_{1}(x,t)\}_{x\in V},
$$
and let $\{\theta_0({x}),\theta_1({x})\}_{x\in V}$ be the stationary microscopic profile, i.e., for $x\in V$ and $i\in I$,
\begin{equation}
\label{eq:definition of microscopic profile}
\theta_i({x})=\E_{\mu_{stat}}[\eta_i(x,t)].
\end{equation}
Write $\P_\xi$ (and $\E_\xi$) to denote the law (and the expectation) of the dual Markov process 
$$
\{\xi(t)\colon\,t \geq 0\}, \qquad \xi(t):=\{\xi_0(x,t),\xi_{1}(x,t)\}_{x\in \hat V},
$$ 
starting from $\xi=\{\xi_{0}(x),\xi_{1}(x)\}_{x\in\hat V}$. For $x\in V$, set
\begin{equation}
\label{eq:vector definition of p,q}
\begin{aligned}
\vec{p}_x
&:= \big[
\begin{array}{c c c c}
\hat p(\delta_{(x,0)}, \delta_{(L,0)}) 
& \hat p(\delta_{(x,0)}, \delta_{(L,1)}) 
& \hat p(\delta_{(x,0)}, \delta_{(R,0)}) 
& \hat p(\delta_{(x,0)}, \delta_{(R,1)})
\end{array}
\big]^T,\\
\vec{q}_x
&:=\big[
\begin{array}{c c c c}
\hat p(\delta_{(x,1)}, \delta_{(L,0)}) 
& \hat p(\delta_{(x,1)}, \delta_{(L,1)}) 
& \hat p(\delta_{(x,1)}, \delta_{(R,0)}) 
& \hat p(\delta_{(x,1)}, \delta_{(R,1)})
\end{array}
\big]^T,
\end{aligned}
\end{equation}
where
\begin{equation}
\label{eq: definition of absorption probability}
\hat p(\xi, \tilde \xi)=\lim\limits_{t\to\infty}\P_\xi(\xi(t) =\tilde \xi),\quad
\xi=\delta_{(x,i)} \text{ for some } (x,i)\in V\times I,\ \tilde{\xi}\in\{\delta_{(L,0)},\delta_{(L,1)},\delta_{(R,0)},\delta_{(R,1)}\},
\end{equation}
and let 
\begin{equation}
\label{eq:vector definition of rho}
\vec{\rho}
:=\big[
\begin{array}{c c c c}
\rho_{(L,0)} & \rho_{(L,1)} & \rho_{(R,0)} & \rho_{(R,1)}
\end{array}
\big]^T.
\end{equation}
Note that $\hat p(\delta_{(x,i)}, \cdot)$ is the probability of the dual process, starting from a single particle at site $x$ at layer $i\in I$, of being absorbed at one of the four reservoirs. Using Proposition~\ref{proposition: duality absorbing sites} and Theorem~\ref{theorem:existence_uniqueness}, we obtain the following.

\begin{corollary}{\bf [Dual representation of stationary profile]}
\label{corollary:dual representation of micro-profile}
For $x\in V$, the microscopic stationary profile is given by
\begin{equation}
\label{eq:microscopic profile}
\begin{aligned}
\theta_0(x) &= \vec{p}_x\, \cdot \,\vec{\rho},\\ 
\theta_1(x) &= \vec{q}_x\, \cdot\, \vec{\rho},
\end{aligned}
\qquad x\in\{1,\ldots,N\},
\end{equation}
where $\vec{p}_x,\vec{q}_x$ and $\vec{\rho}$ are as in \eqref{eq:vector definition of p,q}--\eqref{eq:vector definition of rho}.
\end{corollary}

We next compute the absorption probabilities associated to the dual process in order to obtain a more explicit expression for the stationary microscopic profile $\{\theta_0({x}),\theta_1({x})\}_{x\in V}$. The absorption probabilities $\hat{p}(\cdot\,,\cdot)$ of the dual process satisfy 
$$
(\hat L \hat p)(\cdot, \tilde \xi)(\xi) = 0\qquad\forall\,\xi\in\hat{\mathcal{X}},
$$
where $\hat L$ is the dual generator defined in \eqref{eq: dual generator IRW with sleeping particles + reservoirs}, i.e., they are harmonic functions for the generator $\hat{L}$.

In matrix form, the above translates into the following systems of equations:
\begin{equation}
\label{eq:linear system for absorption probability}
\begin{aligned}
\vec{p}_1 &= \frac{1}{2 + \gamma}\,(\vec{p}_0+\vec{p}_2)+\frac{\gamma}{2+\gamma}\,\vec{q}_1,\\
\vec{q}_1 &= \frac{\epsilon}{(1+\epsilon)+\gamma}\,\vec{q}_2+\frac{1}{(1+\epsilon)+\gamma}\,\vec{q}_0
+\frac{\gamma}{(1+\epsilon)+\gamma}\,\vec{p}_1,\\
\vec{p}_x &= \frac{1}{2+\gamma}\,(\vec{p}_{x-1}+\vec{p}_{x+1})+\frac{\gamma}{2+\gamma}\,\vec{q}_x,
&x \in \{2, \ldots, N-1\},\\
\vec{q}_x &= \frac{\epsilon}{2\epsilon+\gamma}\,(\vec{q}_{x-1}+\vec{q}_{x+1})
+\frac{\gamma}{2\epsilon+\gamma}\,\vec{p}_x,
&x \in \{2, \ldots, N-1\},\\
\vec{p}_N &=\frac{1}{2+\gamma}\,(\vec{p}_{N-1}+\vec{p}_{N+1})+\frac{\gamma}{2+\gamma}\,\vec{q}_N,\\
\vec{q}_N &=  \frac{\epsilon}{(1+\epsilon)+\gamma}\,\vec{q}_{N-1}
+\frac{1}{(1+\epsilon)+\gamma}\,\vec{q}_{N+1}+\frac{\gamma}{(1+\epsilon)+\gamma}\,\vec{p}_N,
\end{aligned}
\end{equation}
where 
$$
\begin{aligned}
&\vec{p}_0 := \big[
\begin{array}{c c c c}
1 & 0 & 0 & 0
\end{array}
\big]^T,& 
&\vec{q}_{0} :=\big[
\begin{array}{c c c c}
0 & 1 & 0 & 0
\end{array}
\big]^T,\\
&\vec{p}_{N+1} := \big[
\begin{array}{c c c c}
0 & 0 & 1 & 0
\end{array}
\big]^T,& 
&\vec{q}_{N+1} := \big[
\begin{array}{c c c c}
0 & 0 & 0 & 1
\end{array}
\big]^T.
\end{aligned}
$$
We divide the analysis of the absorption probabilities into two cases: $\epsilon=0$ and $\epsilon>0$.


\paragraph{Case $\epsilon=0$.}

\begin{proposition}{\bf[Absorption probability for $\epsilon=0$]} 
\label{prop:absobption probability for epsilon=0}
Consider the dual process 
$$
\{\xi(t)\colon\,t \geq 0\}, \qquad \xi(t) = \{\xi_0(x,t),\xi_{1}(x,t)\}_{x \in V},
$$
with generator $\hat{L}_{\epsilon,\gamma,N}$ (see \eqref{eq: dual generator IRW with sleeping particles + reservoirs}) with $\epsilon=0$. Then for the dual process, starting from a single particle, the absorption probabilities $\hat{p}(\cdot,\cdot)$ (see \eqref{eq: definition of absorption probability}) are given by
\begin{equation}
\begin{aligned}
\hat p(\delta_{(x,0)}, \delta_{(L,0)})
&=\frac{1+\gamma}{1+2\gamma}\left(\frac{(1+N)+(1+2N)\,\gamma}{1+N+2N\gamma}
-\frac{1+2\gamma}{1+N+2N\gamma}\,x\right),\\
\hat p(\delta_{(x,0)}, \delta_{(L,1)})
&=\frac{\gamma}{1+2\gamma}\left( \frac{(1+N)+(1+2N)\,\gamma}{1+N+2N\gamma}
-\frac{1+2\gamma}{1+N+2N\gamma}\,x\right),\\
\hat p(\delta_{(x,0)}, \delta_{(R,0)})
&=\frac{1+\gamma}{1+2\gamma}\left( \frac{-\gamma}{1+N+2N\gamma}
+\frac{1+2\gamma}{1+N+2N\gamma}\,x\right),\\
\hat p(\delta_{(x,0)}, \delta_{(R,1)})
&=\frac{\gamma}{1+2\gamma}\left( \frac{-\gamma}{1+N+2N\gamma}+\frac{1+2\gamma}{1+N+2N\gamma}\,x\right),
\end{aligned}
\end{equation} 
\begin{equation}
\begin{aligned}
\hat p(\delta_{(1,1)}, \delta_{(L,0)}) 
&=\frac{\gamma\,(N-\gamma+2N\gamma)}{(1+2\gamma)(1+N+2N\gamma)},\quad
\hat{p}(\delta_{(1,1)}, \delta_{(L,1)})
=\frac{1+N+(1+3N)\gamma-(1-2N)\gamma^2}{(1+2\gamma)(1+N+2N\gamma)},\\
\hat p(\delta_{(1,1)}, \delta_{(R,0)})
&=\frac{\gamma(1+\gamma)}{(1+2\gamma)(1+N+2N\gamma)},\quad
\hat p(\delta_{(1,1)}, \delta_{(R,1)})
=\frac{\gamma^2}{(1+2\gamma)(1+N+2N\gamma)},
\end{aligned}
\end{equation}
and
\begin{equation}
\hat p(\delta_{(x,1)}, \delta_{(\beta,i)})
=\hat p(\delta_{(x,0)}, \delta_{(\beta,i)}), \qquad x\in \{2,\ldots,N-1\},\,(\beta,i)\in \{L,R\}\times I, 
\end{equation}
and 	
\begin{equation}
\begin{aligned}
\hat p(\delta_{(N,1)}, \delta_{(L,0)})
&=\hat{p}(\delta_{(1,1)}, \delta_{(R,0)}),\quad
\hat{p}(\delta_{(N,1)}, \delta_{(L,1)})
=\hat{p}(\delta_{(1,1)}, \delta_{(R,1)}),\\
\hat p(\delta_{(N,1)}, \delta_{(R,0)})
&=\hat{p}(\delta_{(1,1)}, \delta_{(L,0)}),\quad
\hat p(\delta_{(N,1)}, \delta_{(R,1)})
=\hat{p}(\delta_{(1,1)}, \delta_{(L,1)}).
\end{aligned}
\end{equation}
\end{proposition}

\begin{proof}
Note that, for $\epsilon=0$, from the linear system in \eqref{eq:linear system for absorption probability} we get
\begin{equation}
\begin{aligned}
&\vec{p}_{x+1}-\vec{p}_{x}=\vec{p}_{x}-\vec{p}_{x-1},\\
&\vec{q}_{x}=\vec{p}_x,
\end{aligned}
\quad x\in\{2,\ldots,N-1\}.
\end{equation}
Thus, if we set $\vec{c}=\vec{p}_{2}-\vec{p}_{1}$, then it suffices to solve the following 4 linear equations with 4 unknowns $\vec{p}_1,\vec{c},\vec{q}_1$, $\vec{q}_N$:
\begin{equation}
\begin{aligned}
&\vec{p}_1 = \frac{1}{2 + \gamma}\,(\vec{p}_0+\vec{p}_1+\vec{c})+\frac{\gamma}{2+\gamma}\,\vec{q}_1,\\
&\vec{q}_1 = \frac{1}{1+\gamma}\,\vec{q}_0
+\frac{\gamma}{1+\gamma}\,\vec{p}_1,\\
&\vec{p}_1+(N-1)\vec{c} =\frac{1}{2+\gamma}\,(\vec{p}_1+(N-2)\vec{c}+\vec{p}_{N+1})
+\frac{\gamma}{2+\gamma}\,\vec{q}_N,\\
&\vec{q}_N = 
\frac{1}{1+\gamma}\,\vec{q}_{N+1}+\frac{\gamma}{1+\gamma}\,(\vec{p}_1+(N-1)\vec{c}).
\end{aligned}
\end{equation}
Solving the above equations we get the desired result.
\end{proof}

As a consequence, we obtain the stationary microscopic profile for the original process $\{\eta(t)\colon\,t \geq 0\}, \ \eta(t) = \{\eta_0(x,t),\eta_{1}(x,t)\}_{x \in V}$ when $\epsilon=0$.

\begin{theorem}{\bf[Stationary microscopic profile for $\epsilon=0$]}
\label{corollary:microscopic profile for epsilon=0}
$\mbox{}$\\	
The stationary microscopic profile $\{\theta_0(x),\theta_1(x)\}_{x\in V}$ (see \eqref{eq:definition of microscopic profile}) for the process $\{\eta(t)\colon\,t \geq 0\}$ with $\eta(t) = \{\eta_0(x,t),\eta_{1}(x,t)\}_{x \in V}$ with generator $L_{\epsilon,\gamma,N}$ (see \eqref{eq: generator IRW with sleeping particles + reservoirs}) and $\epsilon=0$ is given by
\begin{equation}
\label{eq:microscopic profile bottom layer e=0}
\begin{aligned}
\theta_0(x) 
&= \frac{1+\gamma}{1+2\gamma}\left[\left( \tfrac{(1+N)+(1+2N)\,\gamma}{1+N+2N\gamma}
-\tfrac{1+2\gamma}{1+N+2N\gamma}\,x \right)\rho_{L,0}
+\left( \tfrac{-\gamma}{1+N+2N\gamma}+\tfrac{1+2\gamma}{1+N+2N\gamma}\,x\right)\rho_{R,0}\right]\\
&\qquad +\frac{\gamma}{1+2\gamma}\left[\left( \tfrac{(1+N)+(1+2N)\,\gamma}{1+N+2N\gamma}
-\tfrac{1+2\gamma}{1+N+2N\gamma}\,x\right)\rho_{(L,1)}
+\left( \tfrac{-\gamma}{1+N+2N\gamma}+\tfrac{1+2\gamma}{1+N+2N\gamma}\,x\right)\rho_{(R,1)}\right]
\end{aligned}
\end{equation}
and
\begin{equation}
\label{eq:microscopic profile top layer e=0}
\begin{aligned}
\theta_1(1) &= \frac{\gamma}{1+\gamma}\,\theta_0(1)+\frac{1}{1+\gamma}\,\rho_{(L,1)},\\
\theta_1(x) &=\theta_0(x), &x\in\{2,\ldots,N-1\},\\
\theta_1(N) &= \frac{\gamma}{1+\gamma}\,\theta_0(N)+\frac{1}{1+\gamma}\,\rho_{(R,1)}.
\end{aligned}
\end{equation}
\end{theorem}

\begin{proof}
The proof directly follows from Corollary~\ref{corollary:dual representation of micro-profile} and Proposition~\ref{prop:absobption probability for epsilon=0}.
\end{proof}


\paragraph{Case $\epsilon>0$.}

We next compute the absorption probabilities for the dual process and the stationary microscopic profile for the original process when $\epsilon>0$.

\begin{proposition}
{\bf[Absorption probability for $\epsilon>0$]} 
\label{prop:absobption probability for epsilon>0}
Consider the dual process 
$$
\{\xi(t)\colon\,t \geq 0\}, \qquad \xi(t) = \{\xi_0(x,t),\xi_{1}(x,t)\}_{x \in V},
$$
with generator $\hat{L}_{\epsilon,\gamma}$ (see \eqref{eq: dual generator IRW with sleeping particles + reservoirs}) with $\epsilon>0$. Let $\hat{p}(\cdot,\cdot)$ (see \eqref{eq: definition of absorption probability}) be the absorption probabilities of the dual process starting from a single particle, and let $(\vec{p}_x,\vec{q}_x)_{x\in V}$ be as defined in \eqref{eq:vector definition of p,q}. Then
\begin{equation}
\begin{aligned}
\vec{p}_x &= \vec{c}_1\, x + \vec{c}_2 + \epsilon(\vec{c}_3\,\alpha_1^x+\vec{c}_4\,\alpha_2^x),\\
\vec{q}_x &= \vec{c}_1\, x + \vec{c}_2 - (\vec{c}_3\,\alpha_1^x+\vec{c}_4\,\alpha_2^x),
\end{aligned}
\quad x\in V,
\end{equation}
where $\alpha_1,\alpha_2$ are the two roots of the equation
\begin{equation}
\label{eq:resc_root_def}
\epsilon\alpha^2-(\gamma(1+\epsilon)+2\epsilon)\,\alpha+\epsilon=0,
\end{equation} 
and $\vec{c}_1,\vec{c}_2,\vec{c}_3,\vec{c}_4$ are vectors that depend on the parameters $N,\epsilon,\alpha_1,\alpha_2$ (see \eqref{eq:definition of c1,c2,c3,c4} for explicit expressions). 
\end{proposition}

\begin{proof}
Applying the transformation
\begin{equation}
\label{eq:decouple tranformation}
\vec{\tau}_x := \vec{p}_x+\epsilon \vec{q}_x,\qquad \vec{s}_x := \vec{p}_x-\vec{q}_x,
\end{equation}
we see that the system in \eqref{eq:linear system for absorption probability} decouples in the bulk (i.e., the interior of $V$), and
\begin{equation}
\label{eq:recursive equation}
\vec{\tau}_x = \frac{1}{2}(\vec{\tau}_{x+1}+\vec{\tau}_{x-1}),\qquad
\vec{s}_x = \frac{\epsilon}{\gamma(1+\epsilon)+2\epsilon}(\vec{s}_{x+1}+\vec{s}_{x-1}),
\qquad x\in \{2,\ldots,N-1\}.
\end{equation}
The solution of the above system of recursion equations takes the form
\begin{equation}
\label{eq:decoupled solution}
\vec{\tau}_x = \vec{A}_1 x + \vec{A}_2,\quad\vec{s}_x = \vec{A}_3 \alpha_1^x + \vec{A}_4 \alpha_2^x,
\end{equation}
where $\alpha_1,\alpha_2$ are the two roots of the equation
\begin{equation}
\epsilon\alpha^2-(\gamma(1+\epsilon)+2\epsilon)\,\alpha+\epsilon=0.
\end{equation} 
Rewriting the four boundary conditions in \eqref{eq:linear system for absorption probability} in terms of the new transformations, we get 
\begin{equation}
\big[
\begin{array}{c c c c}
\vec{A}_1 & \vec{A}_2 & \vec{A}_3 & \vec{A}_4
\end{array}
\big] = (1+\epsilon)(M_\epsilon^{-1})^T,
\end{equation}
where $M_\epsilon$ is given by
\begin{equation}
\label{eq:matrix_def}
M_\epsilon:=
\left[\begin{array}{c c c c}
0 & 1 & \epsilon & \epsilon\\
1-\epsilon & 1 & (\epsilon-1)\alpha_1-\epsilon & (\epsilon-1)\alpha_2-\epsilon\\
N+1 & 1 & \epsilon \alpha_1^{N+1} &  \epsilon \alpha_2^{N+1}\\
N+\epsilon & 1 & -\alpha_1^N(\epsilon\alpha_1+(1-\epsilon)) & -\alpha_2^N(\epsilon \alpha_2+(1-\epsilon))
\end{array}
\right].
\end{equation}
Since $\vec{p}_x = \frac{1}{1+\epsilon}(\vec{\tau}_x+\epsilon\vec{s}_x)$ and $\vec{q}_x = \frac{1}{1+\epsilon}(\vec{\tau}_x-\vec{s}_x)$, by setting
$$
\vec{c}_i = \frac{1}{1+\epsilon}\vec{A}_i,\quad i\in\{1,2,3,4\},
$$
we get the desired identities.
\end{proof}

Without loss of generality, from here onwards, we fix the choices of the roots $\alpha_1$ and $\alpha_2$ of the quadratic equation in \eqref{eq:resc_root_def} as
\begin{equation}
\label{eq:definition of two roots}
\alpha_1=1+\frac{\gamma}{2}\left(1+\frac{1}{\epsilon}\right)-\sqrt{\left[1+\frac{\gamma}{2}
\left(1+\frac{1}{\epsilon}\right)\right]^2-1},\quad \alpha_2
=1+\frac{\gamma}{2}\left(1+\frac{1}{\epsilon}\right)+\sqrt{\left[1+\frac{\gamma}{2}
\left(1+\frac{1}{\epsilon}\right)\right]^2-1}.
\end{equation}
Note that, for any $\epsilon, \gamma >0$, we have 
\begin{equation}
\label{eq:product of two roots}
\alpha_1\alpha_2 = 1.
\end{equation}
As a corollary, we get the expression for the stationary microscopic profile of the original process.

\begin{theorem}{\bf[Stationary microscopic profile for $\epsilon>0$]}
\label{corollary:microscopic profile for epsilon>0}
$\mbox{}$\\
The stationary microscopic profile $\{\theta_0(x),\theta_1(x)\}_{x\in V}$ (see \eqref{eq:definition of microscopic profile}) for the process $\{\eta(t)\colon\,t \geq 0\}$ and $\eta(t) = \{\eta_0(x,t),\eta_{1}(x,t)\}_{x \in V}$ with generator $L_{\epsilon,\gamma,N}$ (see \eqref{eq: generator IRW with sleeping particles + reservoirs}) with $\epsilon>0$ is given by
\begin{equation}
	\label{eq:microscopic profile explicit}
\begin{aligned}
\theta_0(x) 
&= (\vec{c}_1\,.\,\vec{\rho}) x +(\vec{c}_2\,.\,\vec{\rho})+\epsilon(\vec{c}_3\,.\,
\vec{\rho})\alpha_1^x+\epsilon(\vec{c}_4\,.\,\vec{\rho})\alpha_2^x,\\
\theta_1(x)
&=(\vec{c}_1\,.\,\vec{\rho}) x +(\vec{c}_2\,.\,\vec{\rho})-(\vec{c}_3\,.\,
\vec{\rho})\alpha_1^x-(\vec{c}_4\,.\,\vec{\rho})\alpha_2^x,
\end{aligned}
\quad x\in V,
\end{equation}
where $(\vec{c}_i)_{1\leq i \leq 4}$ are as in \eqref{eq:definition of c1,c2,c3,c4}, and
$$
\vec{\rho}
:=\big[
\begin{array}{c c c c}
\rho_{(L,0)} & \rho_{(L,1)} & \rho_{(R,0)} & \rho_{(R,1)}
\end{array}
\big]^T.
$$
\end{theorem}

\begin{proof}
The proof follows directly from Corollary~\ref{corollary:dual representation of micro-profile} and Proposition~\ref{prop:absobption probability for epsilon>0}.
\end{proof}

\begin{remark}{\bf [Symmetric layers]} 
{\rm For $\epsilon=1,$ the inverse of the matrix $M_\epsilon$ in the proof of Proposition~\ref{prop:absobption probability for epsilon>0} takes a simpler form. This is because for $\epsilon=1$ the system is fully symmetric. In this case, the explicit expression of the stationary microscopic profile is given by
\begin{equation}
\begin{aligned}
\theta_0(x)=& \frac{1}{2}\left(\frac{N+1-x}{N+1}+ \frac{\alpha_2^{N+1-x}-\alpha_1^{N+1-x}}{\alpha_2^{N+1}
-\alpha_1^{N+1}}\right)\rho_{L,0}+\frac{1}{2}\left(\frac{x}{N+1}+ \frac{\alpha_2^x-\alpha_1^x}{\alpha_2^{N+1}
-\alpha_1^{N+1}} \right)\rho_{R,0}\\
&+\frac{1}{2}\left(\frac{N+1-x}{N+1}- \frac{\alpha_2^{N+1-x}-\alpha_1^{N+1-x}}{\alpha_2^{N+1}
\alpha_1^{N+1}} \right)\rho_{(L,1)}
+\frac{1}{2}\left(\frac{x}{N+1}- \frac{\alpha_2^x-\alpha_1^x}{\alpha_2^{N+1}-\alpha_1^{N+1}}  \right)\rho_{(R,1)}
\end{aligned}
\end{equation}
and 
\begin{equation}
\begin{aligned}
\theta_1(x)=& \frac{1}{2}\left(\frac{N+1-x}{N+1}- \frac{\alpha_2^{N+1-x}-\alpha_1^{N+1-x}}{\alpha_2^{N+1}
-\alpha_1^{N+1}} \right)\rho_{L,0}+\frac{1}{2}\left(\frac{x}{N+1}- \frac{\alpha_2^x-\alpha_1^x}{\alpha_2^{N+1}
-\alpha_1^{N+1}} \right)\rho_{R,0}\\
&+\frac{1}{2}\left(\frac{N+1-x}{N+1}+ \frac{\alpha_2^{N+1-x}-\alpha_1^{N+1-x}}{\alpha_2^{N+1}
-\alpha_1^{N+1}} \right)\rho_{(L,1)} +\frac{1}{2}\left(\frac{x}{N+1}+ \frac{\alpha_2^x-\alpha_1^x}{\alpha_2^{N+1}
-\alpha_1^{N+1}} \right)\rho_{(R,1)}.
\end{aligned}
\end{equation}
However, note that
$$
\theta_0(x) + \theta_1(x) = 2[(\vec{c}_1.\vec{\rho}) x+(\vec{c}_2.\vec{\rho})]-(1-\epsilon)[(\vec{c}_3\,.\,
\vec{\rho})\alpha_1^x-(\vec{c}_4\,.\,\vec{\rho})\alpha_2^x],
$$
which is linear in $x$ only when $\epsilon=1$, and
$$
\theta_0(x) - \theta_1(x) = (1+\epsilon)[(\vec{c}_3\,.\,\vec{\rho})\alpha_1^x+(\vec{c}_4\,.\,\vec{\rho})\alpha_2^x],
$$
which is purely exponential in $x$.}\hfill $\spadesuit$
\end{remark}


\subsubsection{Stationary macroscopic profile and boundary-value problem}
In this section we rescale the finite-volume system with boundary reservoirs, in the same way as was done for the infinite-volume system in Section~\ref{s.hydro} when we derived the hydrodynamic limit (i.e., space is scaled by $1/N$ and the switching rate $\gamma_N$  is scaled such that $\gamma_N N^2\to\Upsilon>0$), and study the validity of Fick's law at stationarity on macroscopic scale. Before we do that, we justify below that the current scaling of the parameters is indeed the proper choice, in the sense that we obtain non-trivial pointwise limits (macroscopic stationary profiles) of the microscopic stationary profiles found in previous sections, and that the resulting limits (when $\epsilon > 0$) satisfy the stationary boundary-value problem given in \eqref{eq:HDL IRW with slow particles} with boundary conditions $ \rho^{\text{stat}}_0(0) = \rho_{L,0}, \ 	\rho^{\text{stat}}_0(1) = \rho_{R,0},\ \rho^{\text{stat}}_1(0) = \rho_{L,1}$ and  	$\rho^{\text{stat}}_1(1) = \rho_{R,1}$.

We say that \textit{the macroscopic stationary profiles} are given by functions $\rho^{\text{stat}}_i: (0,1)\to\R$ for $i\in I$ if, for any $y\in(0,1)$,
\begin{equation}\label{eq: definition macro profile}
\lim_{N\to \infty}\theta_0^{(N)}(\lceil yN\rceil)=\rho^{\text{stat}}_0(y),\quad \lim_{N\to \infty} \theta_1^{(N)}(\lceil yN\rceil)=\rho^{\text{stat}}_1(y).
\end{equation}

\begin{theorem}{\bf[Stationary macroscopic profile]}
Let $(\theta_0^{(N)}(x),\theta_1^{(N)}(x))_{x\in V}$ be the stationary microscopic profile (see \eqref{eq:definition of microscopic profile}) for the process $\{\eta(t)\colon\,t \geq 0\}, \ \eta(t) = \{\eta_0(x,t),\eta_{1}(x,t)\}_{x \in V}$ with generator $L_{\epsilon,\gamma_N,N}$ (see \eqref{eq: generator IRW with sleeping particles + reservoirs}), where $\gamma_N$ is such that $\gamma_NN^2\to\Upsilon$ as $N\to\infty$ for some $\Upsilon>0$. Then, for each $y\in(0,1)$, the pointwise limits (see Fig.~\ref{fig:density})
\begin{equation}
\label{eq:definition of macroscopic profile}
\rhosz(y):=\lim_{N\to \infty} \theta_0^{(N)}(\lceil yN\rceil),\quad \rhoso(y):=\lim_{N\to \infty} \theta_1^{(N)}(\lceil yN\rceil),
\end{equation}
exist and are given by
\begin{equation}
\label{eq:macroscopic profile e=0}
\begin{aligned}
\rhosz(y) &= \rho_{L,0}+(\rho_{R,0}-\rho_{L,0})y, && y\in(0,1),\\
\rhoso(y) &= \rhosz(y), && y\in(0,1),
\end{aligned}
\end{equation}
when $\epsilon=0$, while
\begin{equation}
\label{eq:macroscopic profile bottom layer e>0}
\begin{aligned}
\rhosz(y)&= \frac{\epsilon}{1+\epsilon}\left[\frac{\sinh \left[B_{\epsilon,\Upsilon}\,(1-y)\right]}
{\sinh \left[B_{\epsilon,\Upsilon}\right]}\,(\rho_{(L,0)}-\rho_{(L,1)})
+\frac{\sinh \left[B_{\epsilon,\Upsilon}\,y\right]}
{\sinh \left[B_{\epsilon,\Upsilon}\right]}\,(\rho_{(R,0)}-\rho_{(R,1)})\right]\\
&+ \frac{1}{1+\epsilon}\left[\rho_{(R,0)}\,y+\rho_{(L,0)} \,(1-y)\right]+\frac{\epsilon}{1+\epsilon}
\left[\rho_{(R,1)}\,y+\rho_{(L,1)} \,(1-y)\right],
\end{aligned}
\end{equation}
\begin{equation}
\label{eq:macroscopic profile top layer e>0}
\begin{aligned}
\rhoso(y)
&= \frac{1}{1+\epsilon}\left[\frac{\sinh \left[B_{\epsilon,\Upsilon}\,(1-y)\right]}
{\sinh \left[B_{\epsilon,\Upsilon}\right]}\,(\rho_{(L,1)}-\rho_{(L,0)})
+\frac{\sinh \left[B_{\epsilon,\Upsilon}\,y\right]}
{\sinh \left[B_{\epsilon,\Upsilon}\right]}\,(\rho_{(R,1)}-\rho_{(R,0)})\right]\\
&+ \frac{1}{1+\epsilon}\left[\rho_{(R,0)}\,y+\rho_{(L,0)} \,(1-y)\right]
+\frac{\epsilon}{1+\epsilon}\left[\rho_{(R,1)}\,y+\rho_{(L,1)} \,(1-y)\right],
\end{aligned}
\end{equation}
when $\epsilon>0$, where $B_{\epsilon,\Upsilon}:=\sqrt{\Upsilon(1+\tfrac{1}{\epsilon})}$. Moreover, when $\epsilon>0$, the two limits in \eqref{eq:definition of macroscopic profile} are uniform in $(0,1)$.
\end{theorem}

\begin{proof}
For $\epsilon=0$, it easily follows from \eqref{eq:microscopic profile bottom layer e=0} plus the fact that $\gamma_N N^2\to \Upsilon>0$ and $\tfrac{\lceil yN\rceil}{N}\to y$ uniformly in $(0,1)$ as $N\to\infty$, that
$$
\lim_{N\to\infty}\sup_{y\in(0,1)} \left| \theta_0^{(N)}(\lceil yN\rceil)-[\rho_{(L,0)}+(\rho_{(R,0)}-\rho_{(L,0)})\,y]\right|=0,
$$
and since $\theta_1(x) = \theta_0(x)$ for all $x\in\{2,\ldots,N-1\}$,
for fixed $y\in(0,1)$, we have 
$$
\lim_{N\to \infty} \theta_1^{(N)}(\lceil yN\rceil) = \rhosz(y).
$$
When $\epsilon>0,$ since $\gamma_N N^2\to\Upsilon>0$ as $N\to\infty,$ we note the following:
\begin{equation}
\begin{aligned}
&\gamma_N\overset{N\to\infty}{\longrightarrow} 0,\\
&\lim\limits_{N\to\infty}\alpha_1=\lim\limits_{N\to\infty}\alpha_2=1,\\
&\lim\limits_{N\to\infty}\alpha_1^N = \ee^{-B_{\epsilon,\Upsilon}},\quad \lim\limits_{N\to\infty}\alpha_2^N = \ee^{B_{\epsilon,\Upsilon}}.
\end{aligned}
\end{equation}
Consequently, from the expressions of $(\vec{c}_i)_{1\leq i \leq 4}$ defined in \eqref{eq:definition of c1,c2,c3,c4}, we also have
\begin{equation}
	\begin{aligned}
		\lim\limits_{N\to\infty}N\vec{c}_1 &= \frac{1}{1+\epsilon}\big[
		\begin{array}{c c c c}
			-1 & -\epsilon & 1 & \epsilon
		\end{array}
		\big]^T,\ \lim\limits_{N\to\infty}\vec{c}_2 = \frac{1}{1+\epsilon}\big[
		\begin{array}{c c c c}
			1 & \epsilon & 0 & 0
		\end{array}
		\big]^T,\\
		\lim\limits_{N\to\infty}\vec{c}_3 &= \frac{1}{1+\epsilon}\big[
		\begin{array}{c c c c}
			\frac{\ee^{B_{\epsilon,\Upsilon}}}{\ee^{B_{\epsilon,\Upsilon}}-\ee^{-B_{\epsilon,\Upsilon}}} & -\frac{\ee^{B_{\epsilon,\Upsilon}}}{\ee^{B_{\epsilon,\Upsilon}}-\ee^{-B_{\epsilon,\Upsilon}}} & -\frac{1}{\ee^{B_{\epsilon,\Upsilon}}-\ee^{-B_{\epsilon,\Upsilon}}} & \frac{1}{\ee^{B_{\epsilon,\Upsilon}}-\ee^{-B_{\epsilon,\Upsilon}}}
		\end{array}
		\big]^T,\\
		\lim\limits_{N\to\infty}\vec{c}_4 &= \frac{1}{1+\epsilon}\big[
		\begin{array}{c c c c}
			-\frac{\ee^{-B_{\epsilon,\Upsilon}}}{\ee^{B_{\epsilon,\Upsilon}}-\ee^{-B_{\epsilon,\Upsilon}}} & \frac{\ee^{-B_{\epsilon,\Upsilon}}}{\ee^{B_{\epsilon,\Upsilon}}-\ee^{-B_{\epsilon,\Upsilon}}} & \frac{1}{\ee^{B_{\epsilon,\Upsilon}}-\ee^{-B_{\epsilon,\Upsilon}}} & -\frac{1}{\ee^{B_{\epsilon,\Upsilon}}-\ee^{-B_{\epsilon,\Upsilon}}}
		\end{array}
		\big]^T.
	\end{aligned}
\end{equation}
Combining the above equations with \eqref{eq:microscopic profile explicit}, and the fact that $\tfrac{\lceil yN\rceil}{N}\to y$ uniformly in $(0,1)$ as $N\to\infty$, we get the desired result.
\end{proof}
\begin{remark}{\bf [Non-uniform convergence]}
	\label{rem:uniform convergence of macroscopic profile}
	\rm{Note that for $\epsilon>0$ both stationary macroscopic profiles, when extended continuously to the closed interval $[0,1]$, match the prescribed boundary conditions. This is different from what happens  for $\epsilon=0$, where the continuous extension of $\rhoso$ to the closed interval $[0,1]$ equals  $\rhosz(y)= \rho_{L,0}+(\rho_{R,0}-\rho_{L,0})y$, which does not necessarily match the prescribed boundary conditions unless $\rho_{(L,1)}=\rho_{(L,0)}$ and $ \rho_{(R,1)}=\rho_{(R,0)}$. Moreover, as can be seen from the proof above, for $\epsilon>0,$ the convergence of $\theta_i$ to $\rho_i$ is uniform in $[0,1]$, i.e., 
$$\lim_{N\to\infty}\sup_{y\in[0,1]} \left|\, \rhosz(y)-\theta_0^{(N)}(\lceil yN\rceil)\,\right|=0,\quad \lim_{N\to\infty}\sup_{y\in[0,1]} \left|\, \rhoso(y)-\theta_1^{(N)}(\lceil yN\rceil)\,\right|=0,$$ while for $\epsilon=0$, the convergence of $\theta_1$ to $\rho_1$ is not uniform in $[0,1]$  when either $\rho_{(L,0)}\neq\rho_{(L,1)}$ or $\rho_{(R,0)}\neq\rho_{(R,1)}$. 
	
	Also, if $\rho_i^{\text{stat}, \epsilon}(\cdot)$ denotes the macroscopic profile defined in \eqref{eq:macroscopic profile bottom layer e>0}$-$\eqref{eq:macroscopic profile top layer e>0}, then for $\epsilon>0$ and $i\in\{0,1\}$, we have
\begin{equation}
	\label{eq:macroscopic continuity in epsilon}
	\lim\limits_{\epsilon\to 0}\rho_i^{\text{stat}, \epsilon}(y)\to\rho_i^{\text{stat}, 0}(y)
\end{equation}
for fixed $y\in(0,1)$ and $i\in\{0,1\}$, where $\rho_i^{\text{stat}, 0}(\cdot)$ is the corresponding macroscopic profile in \eqref{eq:macroscopic profile e=0} for $\epsilon=0$. However, this convergence is also not uniform for $i=1$ when $\rho_{(L,0)}\neq\rho_{(L,1)}$ or $\rho_{(R,0)}\neq\rho_{(R,1)}$.}
\hfill$\spadesuit$
\end{remark}

In view of the considerations in Remark~\ref{rem:uniform convergence of macroscopic profile}, we next concentrate on the case $\epsilon >0.$ The following result tells us that for $\epsilon>0$ the stationary macroscopic profiles satisfy a stationary PDE with fixed boundary conditions and also admit a stochastic representation in terms of an absorbing switching diffusion process.

\begin{theorem}{\bf[Stationary boundary value problem]}
\label{theorem:stationary boundary value problem}
Consider the boundary value problem 
\begin{equation}
\label{eq:stationary reaction diffusion with boundaries}
\begin{cases}
0 = \Delta u_0 + \Upsilon(u_1-u_0),\\
0 = \epsilon \Delta u_1 + \Upsilon(u_0-u_1),
\end{cases}
\end{equation}
with boundary conditions
\begin{equation}
\label{eq:stationary boundary condition}
\begin{cases}
u_0(0) = \rho_{L,0}, \ 	u_0(1) = \rho_{R,0},\\
u_1(0) = \rho_{L,1}, \ 	u_1(1) = \rho_{R,1},
\end{cases}
\end{equation}
where $\epsilon,\Upsilon>0$, and the four boundary parameters $\rho_{(L,0)},\, \rho_{(L,1)},\,\rho_{(R,0)},\,\rho_{(R,1)}$ are also positive. Then the PDE admits a unique strong solution given by
\begin{equation}
u_i(y) = \rhosi(y),\qquad y\in[0,1],
\end{equation}
where $(\rhosz(\cdot),\rhoso(\cdot))$ are as defined in \eqref{eq:definition of macroscopic profile}. Furthermore, $(\rhosz(\cdot),\rhoso(\cdot))$ has the stochastic representation
\begin{equation}
\label{eq:stochastic representation of stationary macro profile}
\rhosi(y) = \E_{(y,i)}[\Phi_{i_\tau}(X_\tau)],
\end{equation}
where $\{i_t\colon\,t \geq 0\}$ is the pure jump process on state space $I=\{0,1\}$ that switches at rate $\Upsilon$, the functions $\Phi_0,\Phi_1\colon\,I\to\R_+$ are defined as 
$$
\Phi_0 = \rho_{(L,0)}\,\boldsymbol1_{\{0\}} + \rho_{(R,0)}\,\boldsymbol1_{\{1\}},\qquad \Phi_1 
= \rho_{(L,1)}\,\boldsymbol1_{\{0\}} + \rho_{(R,1)}\,\boldsymbol1_{\{1\}},
$$
$\{X_t\colon\,t \geq 0\}$ is the stochastic process $[0,1]$ that satisfies the SDE
$$
\dd X_t = \psi(i_t)\,\dd W_t
$$
with $W_t=B_{2t}$ and $\{B_t\colon\, t \geq 0\}$ standard Brownian motion, the switching diffusion process $\{(X_t,i_t)\colon\,t\geq 0\}$ is killed at the stopping time 
$$
\tau := \inf\{t\geq 0\,\colon\,X_t\in I\},
$$
and $\psi\colon\,I\to \{1,\epsilon\}$ is given by $\psi := \boldsymbol1_{\{0\}} + \epsilon\,\boldsymbol1_{\{1\}}$.
\end{theorem}

\begin{proof}
It is straightforward to verify that for $\epsilon>0$ the macroscopic profiles $\rho_0,\rho_1$ defined in \eqref{eq:macroscopic profile bottom layer e>0}$-$\eqref{eq:macroscopic profile top layer e>0} are indeed uniformly continuous in $(0,1)$ and thus can be uniquely extended continuously to $[0,1]$, namely, by defining $\rhosi(0)=\rho_{(L,i)},\, \rhosi(1)=\rho_{(R,i)}$ for $i\in I$. Also $\rhosi \in C^\infty([0,1])$ for $i\in I$ and satisfy the stationary PDE \eqref{eq:stationary reaction diffusion with boundaries}, with the boundary conditions specified in \eqref{eq:stationary boundary condition}.

The stochastic representation of a solution of the system in \eqref{eq:stationary reaction diffusion with boundaries} follows from \cite[p385, Eq.(4.7)]{F85}. For the sake of completeness, we give the proof of uniqueness of the solution of \eqref{eq:stationary reaction diffusion with boundaries}. Let $u=(u_0,u_1)$ and $v=(v_0,v_1)$ be two solutions of the stationary reaction diffusion equation with the specified boundary conditions in \eqref{eq:stationary boundary condition}. Then $(w_0,w_1) :=(u_0-v_0,u_1-v_1)$ satisfies
\begin{equation}
	\label{eq:stationary reaction diffusion with zero boundary}
	\begin{cases}
		0 = \Delta w_0 + \Upsilon(w_1-w_0),\\
		0 = \epsilon \Delta w_1 + \Upsilon(w_0-w_1),
	\end{cases}
\end{equation}
with boundary conditions
\begin{equation}
	\label{eq:stationary zero boundary condtions}
		w_0(0) = w_0(1) = w_1(0) = w_1(1) = 0.
\end{equation}
Multiplying the two equations in \eqref{eq:stationary reaction diffusion with zero boundary} with $w_0$ and $w_1$, respectively, and using the identity
$$
w_i\Delta w_i = \nabla\cdot(w_i\nabla w_i)-|\nabla w_i|^2,\quad i\in I,
$$
we get
\begin{equation}
	\begin{cases}
		0 = \nabla\cdot(w_0\nabla w_0)-|\nabla w_0|^2+\Upsilon(w_1-w_0)w_0,\\
		0 = \epsilon\nabla\cdot(w_1\nabla w_1)-\epsilon|\nabla w_1|^2+\Upsilon(w_0-w_1)w_1.
	\end{cases}
\end{equation}
Integrating both equations by parts over $[0,1],$ we get
\begin{equation}
	\begin{aligned}
	0 &= -[w_0(1)\nabla w_0(1)-w_0(0)\nabla w_0(0)] - \int_{0}^{1} dy\,|\nabla w_0(y)|^2+\Upsilon\int_{0}^{1}dy\,(w_1(y)-w_0(y))w_0(y),\\
	0&=-\epsilon[w_1(1)\nabla w_1(1)-w_1(0)\nabla w_1(0)] - \epsilon\int_{0}^{1} dy\,|\nabla w_1(y)|^2+\Upsilon\int_{0}^{1}dy\,(w_0(y)-w_1(y))w_1(y).
	\end{aligned}
\end{equation}
Adding the above two equations and using the zero boundary conditions in \eqref{eq:stationary zero boundary condtions}, we have
\begin{equation}
\int_{0}^{1} dy\,|\nabla w_0(y)|^2+\epsilon\int_{0}^{1} dy\,|\nabla w_1(y)|^2+\Upsilon\int_{0}^{1}dy\,[w_1(y)-w_0(y)]^2 = 0.
\end{equation}
Since both $w_0$ and $w_1$ are continuous and $\epsilon>0,\Upsilon>0$,
it follows that
\begin{equation}
		w_0 = w_1,\quad
		\nabla w_0 = \nabla w_1 = 0,
\end{equation}
and so $w_0=w_1\equiv 0$.
\end{proof}

Note that, as a result of Theorem~\ref{theorem:stationary boundary value problem}, the four absorption probabilities of the switching diffusion process $\{(X_t,i_t)\,:\,t\geq 0\}$ starting from $(y,i)\in[0,1]\times I$ are indeed the respective coefficients of $\rho_{(L,0)},\rho_{(L,1)},\rho_{(R,0)}$, $\rho_{(R,1)}$ appearing in the expression of $\rhosi(y)$. Furthermore note that, as a consequence of Theorem~\ref{theorem:stationary boundary value problem} and the results in \cite[Section 3]{HA80II}, the time-dependent boundary-value problem
\begin{equation}
\label{eq:reaction diffusion with boundaries}
\begin{cases}
\partial_t \rho_0 = \Delta \rho_0 + \Upsilon(\rho_1-\rho_0),\\
\partial_t \rho_1 = \epsilon \Delta \rho_1 + \Upsilon(\rho_0-\rho_1),
\end{cases}
\end{equation}
with initial conditions
\begin{equation}
\label{eq:initial conditions for boundary}
\begin{cases}
\rho_0(x,0) = \bar \rho_0(x),\\
\rho_1(x,0) = \bar \rho_1(x),
\end{cases}
\end{equation}
and boundary conditions
\begin{equation}
\label{eq:boundary condition}
\begin{cases}
\rho_0(0,t) = \rho_{L,0}, \ 	\rho_0(1,t) = \rho_{R,0},\\
\rho_1(0,t) = \rho_{L,1}, \ 	\rho_1(1,t) = \rho_{R,1},
\end{cases}
\end{equation}
admits a unique solution given by 
\begin{equation}
\begin{cases}
\rho_0(x,t) =  \rho^{\text{hom}}_0(x,t)+ \rhosz(x),\\
\rho_1(x,t) = \rho^{\text{hom}}_1(x,t)+ \rhoso(x),
\end{cases}
\end{equation}
where
\begin{equation}
\label{eq: sol hom 0}
\rho_0^{\text{hom}}(x,t)=\ee^{-\Upsilon t}h_0(x,t)+\frac{\Upsilon}{1-\epsilon}
\ee^{-\Upsilon t}\int_{\epsilon t}^t \dd s\,\left( \left(\frac{s-\epsilon t}{t-s}\right)^{1/2}I_1(\upsilon(s))h_0(x,s)
+I_0(\upsilon(s))h_1(x,s) \right),
\end{equation}
\begin{equation}
\label{eq: sol hom 1}
\rho_1^{\text{hom}}(x,t)=\ee^{-\Upsilon t}h_1(x,\epsilon t)+\frac{\Upsilon}{1-\epsilon}\ee^{-\Upsilon t}\int_{\epsilon t}^t \dd s\,
\left( \left(\frac{s-\epsilon t}{t-s}\right)^{-1/2}I_1(\upsilon(s))h_1(x,s)+I_0(\upsilon(s))h_0(x,s) \right),
\end{equation} $\upsilon(s)=\frac{2\Upsilon}{1-\epsilon}((t-s)(s-\epsilon t))^{1/2}$, $I_0(\cdot)$ and $I_1(\cdot)$ are the modified Bessel functions, $h_0(x,t)$, $h_1(x,t)$ are the solutions of 
\begin{equation}
\begin{cases}
\partial_t h_0 = \Delta h_0,  \\
\partial_t h_1 = \Delta h_1, \\
h_0(x,0)=  \bar \rho_0(x)-\rhosz(x), \\
h_1(x,1)= \bar \rho_1(x) -\rhoso(x),\\
h_0(0,t)= h_0(1,t)= h_1(0,t)=h_1(1,t)=0,
\end{cases}
\end{equation}
and $\rhosz(x)$, $\rhoso(x)$ are given in \eqref{eq:macroscopic profile top layer e>0}.

We conclude this section by proving that the solution of the time-dependent boundary-value problem in \eqref{eq:reaction diffusion with boundaries} converges to the stationary profile in \eqref{eq:macroscopic profile top layer e>0}.

\begin{proposition}{\bf [Convergence to stationary profile]} 
Let $	\rhohz(x,t)$ and $\rhoho(x,t)$ be as in \eqref{eq: sol hom 0} and \eqref{eq: sol hom 1}, respectively, i.e., the solutions of the boundary-value problem \eqref{eq:reaction diffusion with boundaries} with zero boundary conditions and initial conditions given by $\rhohz(x,0)=  \bar \rho_0(x)-\rhosz(x)$ and $\rhoho(x,0)= \bar \rho_1(x) -\rhoso(x)$. Then, for any $k\in\N$, 
$$
\lim_{t\to \infty} \Big[\|\rhohz(x,t)\|_{C^k(0,1)}+\|\rhoho(x,t)\|_{C^k(0,1)}\Big]=0.
$$
\end{proposition}

\begin{proof}
We start by showing that
\begin{equation}\label{eq: l2 conv}
\lim\limits_{t\to \infty}\,\Big[ \|\rhohz(x,t)\|_{L^2(0,1)}+\|\rhoho(x,t)\|_{L^2(0,1)}\Big]=0.
\end{equation}	
Multiply the first equation of \eqref{eq:reaction diffusion with boundaries} by $\rho_0$ and the second equation by $\rho_1$. Integration by parts yields 
\begin{equation}
\begin{cases}
\partial_t\left( \int_0^1 \dd x\,\rho_0^2 \right)=-\int_0^1 \dd x\,|\partial_x \rho_0|^2
+\Upsilon \int_0^1 \dd x\,(\rho_1\rho_0-\rho_0^2),\\
\partial_t\left( \int_0^1 \dd x\,\rho_1^2(x,t) \right)  
=-\epsilon\int_0^1 \dd x\,|\partial_x \rho_1|^2 +\Upsilon \int_0^1 \dd x\,(\rho_0\rho_1-\rho_1^2).  
\end{cases}
\end{equation}
Summing the two equations and defining $E(t):= \int_0^1 \dd x\,(\rho_0^2 +\rho_1^2)$, we obtain
\begin{equation}
\partial_tE(t)=-\left(\int_0^1 \dd x\,|\partial_x \rho_0|^2  + \epsilon\int_0^1 \dd x\,|\partial_x \rho_1|^2 \right)
 -\Upsilon \int_0^1 \dd x\,(\rho_0-\rho_1)^2.
\end{equation}
The Poincar\'e inequality implies that $\int_0^1 \dd x\,|\partial_x \rho_0|^2 + \epsilon\int_0^1 \dd x\,|\partial_x \rho_1|^2 \ge C_p E(t)$, in particular, $\partial_tE(t)\le -\epsilon C_p E(t)$, from which we obtain 
$$
E(t)\le \ee^{-C_pt}E(0),
$$
and hence \eqref{eq: l2 conv}.

From \cite[Theorem 2.1]{P83} it follows that 
$$
A:=\left[\begin{array}{c c c c}
\Delta-\Upsilon & \Upsilon \\
\Upsilon & \epsilon \Delta-\Upsilon\\
\end{array}
\right],
$$ 
with domain $D(A)=H^2(0,1)\cap H_0^1(0,1)$, generates a semigroup $\{\mathcal S_t\colon\, t\ge 0\}$. If we set $\vec{\rho}(t)=\mathcal S_t (\vec{\bar{\rho}}^{\text{hom}})$, with $\vec{\bar{\rho}}^{\text{hom}}=\vec{\bar{\rho}}-\vec{\rho}^{\text{stat}}$, then by the semigroup property we have 
$$
\vec{\rho}(t)=\mathcal S_{t-1}(\mathcal S_{1/k})^k(\vec{\bar{\rho}}^{\text{\text{hom}}}), \qquad t \geq 1,
$$ 
and hence $A^k\vec{\rho}(t)=\mathcal S_{t-1}(A\mathcal S_{1/k})^k(\vec{\bar{\rho}}^{\text{hom}})$. If we set $\vec{p}:=(A\mathcal S_{1/k})^k(\vec{\bar{\rho}}^{\text{hom}})$, then we obtain, by \cite[Theorem 5.2(d)]{P83},
$$
\| A^k \vec{\rho}(t) \|_{L^2(0,1)}\le \| \mathcal S_{t-1}\vec{p}\|_{L^2(0,1)},
$$
where $\lim_{t\to \infty}\| \mathcal S_{t-1}\vec{p}\|_{L^2(0,1)}=0$ by the first part of the proof. The compact embedding 
$$
D(A^k)\hookrightarrow H^{2k}(0,1)\hookrightarrow C^k(0,1), \qquad k\in\N,
$$
concludes the proof.
\end{proof}


\subsection{The stationary current}
\label{ss.statcurrent}

In this section we compute the expected current in the non-equilibrium steady state that is induced by different densities at the boundaries. We consider the microscopic and macroscopic systems, respectively.


\paragraph{Microscopic system.}

We start by defining the notion of current. The microscopic currents are associated with the edges of the underlying two-layer graph. 
 In our setting, we denote by ${\cal J}^0_{x,x+1}(t)$ and ${\cal J}^1_{x,x+1}(t)$ the instantaneous current through the horizontal edge $(x,x+1)$, $x\in V$, of the bottom layer, respectively, top layer at time $t$. Obviously,
\begin{equation*}
{\cal J}^0_{x,x+1}(t) = \eta_{0}(x,t) - \eta_{0}(x+1,t), 
\qquad\qquad
{\cal J}^1_{x,x+1}(t) =  \epsilon [\eta_{1}(x,t) -  \eta_{1}(x+1,t)].
\end{equation*}
We are interested in the stationary currents ${J}^0_{x,x+1}$, respectively, ${J}^1_{x,x+1}$, which are obtained as
\begin{equation}
\label{curr01}
J^0_{x,x+1} = \E_{stat}[\eta_{0}(x) - \eta_{0}(x+1)], 
\qquad\qquad
J^1_{x,x+1} =  \epsilon \E_{stat}[\eta_{1}(x) -  \eta_{1}(x+1)], 
\end{equation}
where $\E_{stat}$ denotes expectation w.r.t.\ the unique invariant probability measure of the microscopic system $\{\eta(t)\colon\, t \geq 0\}$ with $\eta(t) = \{\eta_{0}(x,t),\eta_{1}(x,t)\}_{x\in V}$. In other words, $J^0_{x,x+1}$ and $J^1_{x,x+1}$ give the average flux of particles of type $0$ and type $1$ across the bond $(x,x+1)$ due to diffusion. 

Of course, the average number of particle at each site varies in time also as a consequence of the reaction term:
\begin{eqnarray*}
\frac{d}{dt} \E[\eta_0(x,t)] 
& = & \E[{\cal J}^0_{x-1,x}(t) - {\cal J}^0_{x,x+1}(t)] +\gamma( \E[\eta_1(x,t)] - \E[\eta_0(x,t)]), \\
\frac{d}{dt} \E[\eta_1(x,t)] 
& = & \E[{\cal J}^1_{x-1,x}(t) - {\cal J}^1_{x,x+1}(t)] + \gamma(\E[\eta_0(x,t)] -  \E[\eta_1(x,t)]).
\end{eqnarray*}
Summing these equations, we see that there is no contribution of the reaction part to the variation of the average number of particles at site $x$:
\begin{equation*}
\frac{d}{dt} \E[\eta_0(x,t) + \eta_1(x,t)] = \E[{\cal J}_{x-1,x}(t) - {\cal J}_{x,x+1}(t)].
\end{equation*}
The sum 
\begin{equation}
\label{currtot}
J_{x,x+1} = J^0_{x,x+1} + J^1_{x,x+1},
\end{equation}
with $J^0_{x,x+1}$ and $J^1_{x,x+1}$ defined in \eqref{curr01}, will be called the {\em stationary current} between sites at $x,x+1$, $x\in V$, which is responsible for the variation of the total average number of particles at each site, regardless of their type.

\begin{proposition}{\bf [Stationary microscopic current]}
For  $x \in \{2,\ldots,N-1\}$ the stationary currents defined in \eqref{curr01} are given by
\begin{equation}
\label{eq: average microscopic current j0 for epsilon = 0}
J^0_{x,x+1} = - \tfrac{1+\gamma}{1+N+2N\gamma}[\rho_{(R,0)}-\rho_{(L,0)}]-\tfrac{\gamma}{1+N+2N\gamma}[\rho_{(R,1)}-\rho_{(L,1)}],\quad J^1_{x,x+1} = 0,
\end{equation}
when $\epsilon=0$ and by
\begin{equation}
\label{eq: average microscopic current j for epsilon > 0}
\begin{aligned}
	J^0_{x,x+1} &= - \vec{c}_1\cdot \vec{\rho} - \epsilon [(\vec{c}_3\cdot \vec{\rho}) \alpha_1^x (\alpha_1 - 1) 
+  (\vec{c}_4\cdot \vec{\rho}) \alpha_2^x (\alpha_2 - 1)],\\
J^1_{x,x+1}  &= - \epsilon \vec{c}_1\cdot\vec{\rho} + \epsilon [(\vec{c}_3\cdot\vec{\rho}) \alpha_1^x (\alpha_1 - 1) 
+ (\vec{c}_4\cdot\vec{\rho}) \alpha_2^x (\alpha_2 - 1)],
\end{aligned}
\end{equation}
when $\epsilon>0,$ where $\vec{c}_1, \vec{c}_3, \vec{c}_4$ are the vectors defined in \eqref{eq:definition of c1,c2,c3,c4} of Appendix~\ref{s.app*}, and $\alpha_1,\alpha_2$ are defined in \eqref{eq:definition of two roots}. As a consequence, the current $J_{x,x+1}= J^0_{x,x+1} + J^1_{x,x+1}$ is independent of $x$ and is given by
\begin{equation}
	J_{x,x+1} =  - \tfrac{1+\gamma}{1+N+2N\gamma}[\rho_{(R,0)}-\rho_{(L,0)}]-\tfrac{\gamma}{1+N+2N\gamma}[\rho_{(R,1)}-\rho_{(L,1)}]
\end{equation}
when $\epsilon = 0$ and 
\begin{equation}
\label{eq: average microscopic current for epsilon > 0}
J_{x,x+1} = - (1+\epsilon)\left[C_1\,(\rho_{R,0}-\rho_{L,0})+\epsilon\,C_2\,(\rho_{R,1}-\rho_{L,1})\right]
\end{equation}
when $\epsilon>0$, where
\begin{equation}
	\label{eq:definition of micro-current coefficients}
\begin{aligned}
C_1 &= \frac{[\alpha_1(1-\epsilon)(\alpha_1^{N-1}-1) 
+ \epsilon\,(\alpha_1^{N+1}-1)]}{\alpha_1(1-\epsilon)(\alpha_1^{N-1}-1)(N+1)+2\epsilon\,(\alpha_1^{N+1}-1)(N+\epsilon)},\\
C_2 &= \frac{(\alpha_1^{N+1}-1)}{\alpha_1(1-\epsilon)(\alpha_1^{N-1}-1)(N+1)
+2\epsilon\,(\alpha_1^{N+1}-1)(N+\epsilon)}.
\end{aligned}
\end{equation}
\end{proposition}

\begin{proof}
From \eqref{curr01} we have
\begin{equation}
J^0_{x,x+1} =  \theta_0(x)  - \theta_0(x+1),
\qquad\qquad
J^1_{x,x+1} =  \epsilon [\theta_1(x)  - \theta_1(x+1)],
\end{equation}
where $\theta_0(\cdot),\theta_1(\cdot)$ are the average microscopic profiles. Thus, when $\epsilon=0$, the expressions of $J^0_{x,x+1},J^1_{x,x+1}$ and consequently $J_{x,x+1}$ follow directly from \eqref{eq:microscopic profile bottom layer e=0}. 

For $\epsilon>0$, using the expressions of $\theta_0(\cdot),\theta_1(\cdot)$ in \eqref{eq:microscopic profile explicit}, we see that
\begin{equation}
\begin{aligned}		
J^0_{x,x+1} &= \theta_0(x)-\theta_0(x+1)= - \vec{c}_1\cdot \vec{\rho} - \epsilon [(\vec{c}_3\cdot \vec{\rho}) \alpha_1^x (\alpha_1 - 1) 
+  (\vec{c}_4\cdot \vec{\rho}) \alpha_2^x (\alpha_2 - 1)],\\
J^1_{x,x+1}  &=\epsilon[\theta_1(x)-\theta_1(x+1)]= - \epsilon \vec{c}_1\cdot\vec{\rho} + \epsilon [(\vec{c}_3\cdot\vec{\rho}) \alpha_1^x (\alpha_1 - 1) 
+ (\vec{c}_4\cdot\vec{\rho}) \alpha_2^x (\alpha_2 - 1)],
\end{aligned}
\end{equation}
where $\vec{c}_1, \vec{c}_3, \vec{c}_4$ are the vectors defined in \eqref{eq:definition of c1,c2,c3,c4} of Appendix~\ref{s.app*}, and $\alpha_1,\alpha_2$ are defined in \eqref{eq:definition of two roots}. Adding the two equations, we also have
\begin{equation}
	J_{x,x+1} = J^0_{x,x+1} + J^1_{x,x+1} = -(1+\epsilon)\,\vec{c}_1\cdot\vec{\rho}= (1+\epsilon)\left[C_1\,(\rho_{R,0}-\rho_{L,0})+\epsilon\,C_2\,(\rho_{R,1}-\rho_{L,1})\right],
\end{equation}
where $C_1,C_2$ are as in \eqref{eq:definition of micro-current coefficients}.
\end{proof}


\paragraph{Macroscopic system.}

The microscopic current scales like $1/N$. Indeed, the currents associated to the two layers in the macroscopic system can be obtained from the microscopic currents, respectively, by defining
\begin{equation}
\label{currmacro}
J^0(y) = \lim_{N\to\infty}  N J^0_{\lfloor yN\rfloor,\lfloor yN\rfloor +1}, 
\qquad
J^1(y) = \lim_{N\to\infty}  N J^1_{\lfloor yN\rfloor,\lfloor yN\rfloor +1}. 
\end{equation}
Below we justify the existence of the two limits and thereby provide explicit expressions for the macroscopic currents.
\begin{proposition}{\bf [Stationary macroscopic current]}
For  $y \in (0,1)$ the stationary currents defined in \eqref{currmacro} are given by
\begin{equation}
\label{eq:average macroscopic current j0 for epsilon = 0}
J^0(y) = -\left[(\rho_{R,0}-\rho_{L,0})\right],\quad J^1(y) = 0,
\end{equation}
when $\epsilon=0$ and by
\begin{equation}
\label{eq:average macroscopic current j0 for epsilon > 0}
\begin{aligned}
J^0(y) &=\frac{\epsilon B_{\epsilon,\Upsilon}}{1+\epsilon}\left[\frac{\cosh \left[B_{\epsilon,\Upsilon}\,(1-y)\right]}
{\sinh \left[B_{\epsilon,\Upsilon}\right]}\,(\rho_{(L,0)}-\rho_{(L,1)})
-\frac{\cosh \left[B_{\epsilon,\Upsilon}\,y\right]}
{\sinh \left[B_{\epsilon,\Upsilon}\right]}\,(\rho_{(R,0)}-\rho_{(R,1)})\right]\\ &\quad-\frac{1}{1+\epsilon}\left[(\rho_{(R,0)}-\rho_{(L,0)})+\epsilon(\rho_{(R,1)}-\rho_{(L,1)})\right]
\end{aligned}
\end{equation}
and
\begin{equation}
\label{eq:average macroscopic current j1 for epsilon > 0}
\begin{aligned}
	J^1(y) &=-\frac{\epsilon B_{\epsilon,\Upsilon}}{1+\epsilon}\left[\frac{\cosh \left[B_{\epsilon,\Upsilon}\,(1-y)\right]}
	{\sinh \left[B_{\epsilon,\Upsilon}\right]}\,(\rho_{(L,0)}-\rho_{(L,1)})
	-\frac{\cosh \left[B_{\epsilon,\Upsilon}\,y\right]}
	{\sinh \left[B_{\epsilon,\Upsilon}\right]}\,(\rho_{(R,0)}-\rho_{(R,1)})\right]\\ &\quad-\frac{\epsilon}{1+\epsilon}\left[(\rho_{(R,0)}-\rho_{(L,0)})+\epsilon(\rho_{(R,1)}-\rho_{(L,1)})\right]
\end{aligned}	
\end{equation}	 	
when $\epsilon > 0$. As a consequence, the total current $J(y)= J^0(y) + J^1(y)$  is constant and is given by
\begin{equation}
\label{eq: average macroscopic current}
J(y) = -\left[(\rho_{R,0}-\rho_{L,0})+\epsilon\,(\rho_{R,1}-\rho_{L,1})\right].
\end{equation}
\end{proposition}

\begin{proof}
For $\epsilon=0$ the claim easily follows from the expressions of $J^0_{x,x+1}, J^1_{x,x+1}$ given in \eqref{eq: average microscopic current j0 for epsilon = 0} and the fact that $\gamma_N\to 0$ as $N\to\infty$. 

When $\epsilon>0$, we first note the following:
\begin{equation}
	\begin{aligned}
		&\gamma_N N^2\overset{N\to\infty}{\longrightarrow} \Upsilon>0,\\
		&\lim\limits_{N\to\infty}\alpha_1= \lim\limits_{N\to\infty}\alpha_2=1,\\
		&\lim\limits_{N\to\infty}N(\alpha_1-1)=-B_{\epsilon,\Upsilon},\ \lim\limits_{N\to\infty}N(\alpha_2-1)=B_{\epsilon,\Upsilon},\\
		&\lim\limits_{N\to\infty}\alpha_1^N = \ee^{-B_{\epsilon,\Upsilon}},\quad \lim\limits_{N\to\infty}\alpha_2^N = \ee^{B_{\epsilon,\Upsilon}}.
	\end{aligned}
\end{equation}
Consequently, from the expressions for $(\vec{c}_i)_{1\leq i \leq 4}$ defined in \eqref{eq:definition of c1,c2,c3,c4}, we also have
\begin{equation}
	\begin{aligned}
		\lim\limits_{N\to\infty}N\vec{c}_1 &= \frac{1}{1+\epsilon}\,\big[
		\begin{array}{c c c c}
			-1 & -\epsilon & 1 & \epsilon
		\end{array}
		\big]^T,\\
		\lim\limits_{N\to\infty}\vec{c}_3 &= \frac{1}{1+\epsilon}\,\big[
		\begin{array}{c c c c}
			\frac{\ee^{B_{\epsilon,\Upsilon}}}{\ee^{B_{\epsilon,\Upsilon}}-\ee^{-B_{\epsilon,\Upsilon}}} & -\frac{\ee^{B_{\epsilon,\Upsilon}}}{\ee^{B_{\epsilon,\Upsilon}}-\ee^{-B_{\epsilon,\Upsilon}}} & -\frac{1}{\ee^{B_{\epsilon,\Upsilon}}-\ee^{-B_{\epsilon,\Upsilon}}} & \frac{1}{\ee^{B_{\epsilon,\Upsilon}}-\ee^{-B_{\epsilon,\Upsilon}}}
		\end{array}
		\big]^T,\\
		\lim\limits_{N\to\infty}\vec{c}_4 &= \frac{1}{1+\epsilon}\,\big[
		\begin{array}{c c c c}
			-\frac{\ee^{-B_{\epsilon,\Upsilon}}}{\ee^{B_{\epsilon,\Upsilon}}-\ee^{-B_{\epsilon,\Upsilon}}} & \frac{\ee^{-B_{\epsilon,\Upsilon}}}{\ee^{B_{\epsilon,\Upsilon}}-\ee^{-B_{\epsilon,\Upsilon}}} & \frac{1}{\ee^{B_{\epsilon,\Upsilon}}-\ee^{-B_{\epsilon,\Upsilon}}} & -\frac{1}{\ee^{B_{\epsilon,\Upsilon}}-\ee^{-B_{\epsilon,\Upsilon}}}
		\end{array}
		\big]^T.
	\end{aligned}
\end{equation}
Combining the above equations with \eqref{eq: average microscopic current j for epsilon > 0}, we have
\begin{equation}
	\begin{aligned}
		J^0(y) &= \lim_{N\to\infty}  N J^0_{\lfloor yN\rfloor,\lfloor yN\rfloor +1}\\
		&= -\epsilon B_{\epsilon,\Upsilon}\Big[\Big(\lim\limits_{N\to\infty}\vec{c}_4\cdot\vec{\rho}\Big)\,\ee^{B_{\epsilon,\Upsilon}y}-\Big(\lim\limits_{N\to\infty}\vec{c}_3\cdot\vec{\rho}\Big)\,\ee^{-B_{\epsilon,\Upsilon}y}\Big]-\Big(\lim\limits_{N\to\infty}N\vec{c}_1\cdot\vec{\rho}\Big)\\
		&=\frac{\epsilon B_{\epsilon,\Upsilon}}{1+\epsilon}\left[\frac{\cosh \left[B_{\epsilon,\Upsilon}\,(1-y)\right]}
		{\sinh \left[B_{\epsilon,\Upsilon}\right]}\,(\rho_{(L,0)}-\rho_{(L,1)})
		-\frac{\cosh \left[B_{\epsilon,\Upsilon}\,y\right]}
		{\sinh \left[B_{\epsilon,\Upsilon}\right]}\,(\rho_{(R,0)}-\rho_{(R,1)})\right]\\ &\quad-\frac{1}{1+\epsilon}\left[(\rho_{(R,0)}-\rho_{(L,0)})+\epsilon(\rho_{(R,1)}-\rho_{(L,1)})\right]
	\end{aligned} 
\end{equation}
and, similarly,
\begin{equation}
	\begin{aligned}
		J^1(y) &= \lim_{N\to\infty}  N J^1_{\lfloor yN\rfloor,\lfloor yN\rfloor +1}\\
		&= \epsilon B_{\epsilon,\Upsilon}\Big[\Big(\lim\limits_{N\to\infty}\vec{c}_4\cdot\vec{\rho}\Big)\,\ee^{B_{\epsilon,\Upsilon}y}-\Big(\lim\limits_{N\to\infty}\vec{c}_3\cdot\vec{\rho}\Big)\,\ee^{-B_{\epsilon,\Upsilon}y}\Big]-\epsilon\,\Big(\lim\limits_{N\to\infty}N\vec{c}_1\cdot\vec{\rho}\Big)\\
		&=-\frac{\epsilon B_{\epsilon,\Upsilon}}{1+\epsilon}\left[\frac{\cosh \left[B_{\epsilon,\Upsilon}\,(1-y)\right]}
		{\sinh \left[B_{\epsilon,\Upsilon}\right]}\,(\rho_{(L,0)}-\rho_{(L,1)})
		-\frac{\cosh \left[B_{\epsilon,\Upsilon}\,y\right]}
		{\sinh \left[B_{\epsilon,\Upsilon}\right]}\,(\rho_{(R,0)}-\rho_{(R,1)})\right]\\ &\quad-\frac{\epsilon}{1+\epsilon}\left[(\rho_{(R,0)}-\rho_{(L,0)})+\epsilon(\rho_{(R,1)}-\rho_{(L,1)})\right].
	\end{aligned} 
\end{equation}
Adding $J^0(y)$ and $J^1(y)$, we obtain the total current
\begin{equation}
	J(y)=J^0(y)+J^1(y)=-\left[(\rho_{R,0}-\rho_{L,0})+\epsilon\,(\rho_{R,1}-\rho_{L,1})\right],
\end{equation}
which is indeed independent of $y$.
\end{proof}

\begin{remark}{\bf [Currents]} \label{rem:relation of macro-current and macro-gradient}
{\rm Combining the expressions for the density profiles and the current, we see that
$$
J^0(y) = - \frac{d\rho_0}{dy} (y),	
\qquad
J^1(y)  = - \epsilon \frac{d\rho_1}{dy} (y).	
$$
}\hfill $\spadesuit$
\end{remark}


\subsection{Discussion: Fick's law  and uphill diffusion}
\label{ss.uphillres}
In this section we discuss the behaviour of the boundary-driven system as the parameter $\epsilon$ is varied. For simplicity we restrict our discussion to the macroscopic setting, although similar comments hold for the microscopic system as well.

In view of the previous results, we can rewrite the equations for the densities $\rho_0(y,t),\rho_1(y,t)$ as
\begin{equation*}
\begin{cases}
\partial_t \rho_0 = -\nabla  J^0 + \Upsilon(\rho_1-\rho_0),\\
\partial_t \rho_1 = -\nabla  J^1 + \Upsilon(\rho_0-\rho_1),\\
J_0 = - \nabla \rho_0, \\
J_1 =  -\epsilon \nabla \rho_1,
\end{cases}
\end{equation*}
which are complemented with the boundary values (for $\epsilon>0$)
\begin{equation*}
\label{eq:boundary condition2}
\begin{cases}
\rho_0(0,t) = \rho_{L,0}, \ 	\rho_0(1,t) = \rho_{R,0},\\
\rho_1(0,t) = \rho_{L,1}, \ 	\rho_1(1,t) = \rho_{R,1}.
\end{cases}
\end{equation*} 
We will be concerned with the total density $\rho = \rho_0 +\rho_1$, whose evolution equation does not contain the reaction part, and is given by
\begin{equation}
\label{hello}
\begin{cases}
\partial_t \rho=  - \nabla J,\\
J = - \nabla ( \rho_0 + \epsilon \rho_1),
\end{cases}
\end{equation}
with boundary values 
\begin{equation}
\label{eq:boundary condition3}
\begin{cases}
\rho(0,t) = \rho_L = \rho_{L,0}+ \rho_{R,0},\\
\rho(1,t) = \rho_R =  \rho_{R,0}+\rho_{R,1}.
\end{cases}
\end{equation}


\paragraph{Non-validity of Fick's law.} 

From \eqref{hello} we immediately see that Fick's law of mass transport is satisfied if and only if $\epsilon =1$. When we allow diffusion and reaction of slow and fast particles, i.e., $0 \le \epsilon <1$, Fick's law breaks down, since the current associated to the total mass is not proportional to the gradient of the total mass. Rather, the current $J$ is the sum of a contribution $J^0$ due to the diffusion of fast particles of type $0$ (at rate 1) and a contribution $J^1$ due to the diffusion of slow particles of type $1$ (at rate $\epsilon$). Interestingly, the violation of Fick's law opens up the possibility of several interesting phenomena that we discuss in what follows.


\paragraph{Equal boundary densities with non-zero current.} 

In a system with diffusion and reaction of slow and fast particles we may observe a {\em non-zero current when the total density has the same value at the two boundaries}. This is different from what is observed in standard diffusive systems driven by boundary reservoirs, where in order to have a stationary current it is necessary that the reservoirs have different chemical potentials, and therefore different densities, at the boundaries.

Let us, for instance, consider the specific case when $\rho_{L,0} =\rho_{R,1} = 2$ and $\rho_{L,1} = \rho_{R,0} = 4$, which indeed implies equal densities at the boundaries given by $\rho_{L} =\rho_{R} = 6$. The density profiles and currents are displayed in Fig.~\ref{fig:density} for two values of $\epsilon$, which shows the comparison between the Fick-regime $\epsilon = 1$ (left panels) and the non-Fick-regime with very slow particles $\epsilon = 0.001$ (right panels).

On the one hand, in the Fick-regime the profile of both types of particles interpolates between the boundary values, with a slightly non-linear shape that has been quantified precisely in \eqref{eq:macroscopic profile bottom layer e>0}--\eqref{eq:macroscopic profile top layer e>0}. Furthermore, in the same regime $\epsilon=1$, the total density profile is flat and the total current $J$ vanishes because $J^0(y) = -J^1(y)$ for all $y\in [0,1]$.

On the other hand, in the non-Fick-regime with $\epsilon = 0.001$, the stationary macroscopic profile for the fast particles interpolates between the boundary values almost linearly (see \eqref{eq:macroscopic continuity in epsilon}), whereas the profile for the slow particles is non-monotone: it has two bumps at the boundaries and in the bulk closely follows the other profile. This non-monotonicity in the profile of the slow particles is due to the non-uniform convergence in the limit $\epsilon\downarrow 0$, as pointed out in the last part of Remark~\ref{rem:uniform convergence of macroscopic profile}. As a consequence, the total density profile is not flat and has two bumps at the boundaries. Most strikingly, the total current is $J=-2$, since now the current of the bottom layer $J^0$ is dominating, while the current of the bottom layer $J^1$ is small (order $\epsilon$).

\begin{figure}[!htbp] 
\begin{subfigure}[b]{0.5\linewidth}
\centering
\includegraphics[width=0.75\linewidth]{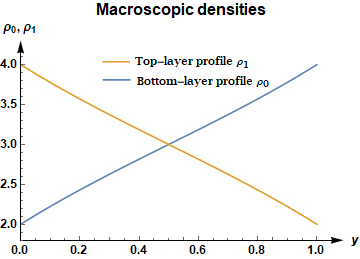} 
\label{fig:density_2} 
\vspace{4ex}
\end{subfigure}
\begin{subfigure}[b]{0.5\linewidth}
\centering
\includegraphics[width=0.75\linewidth]{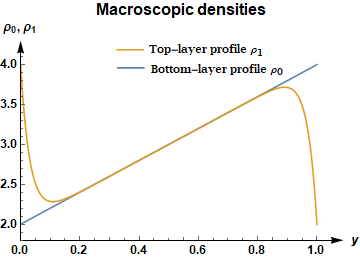} 
\label{fig:density_2_e.001} 
\vspace{4ex}
\end{subfigure} 
\begin{subfigure}[b]{0.5\linewidth}
\centering
\includegraphics[width=0.75\linewidth]{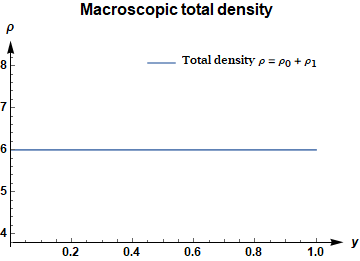} 
\vspace{4ex} 
\end{subfigure}
\begin{subfigure}[b]{0.5\linewidth}
\centering
\includegraphics[width=0.75\linewidth]{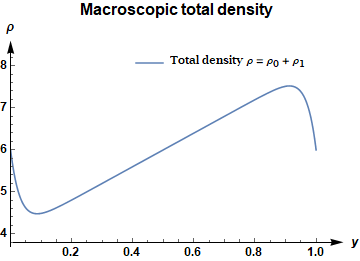} 
\vspace{4ex}
\end{subfigure}
\begin{subfigure}[b]{0.5\linewidth}
\centering
\includegraphics[width=0.75\linewidth]{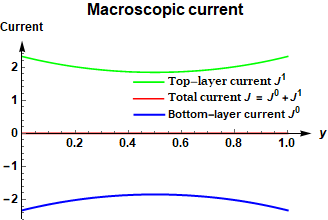} 
\end{subfigure}
\begin{subfigure}[b]{0.5\linewidth}
\centering
\includegraphics[width=0.75\linewidth]{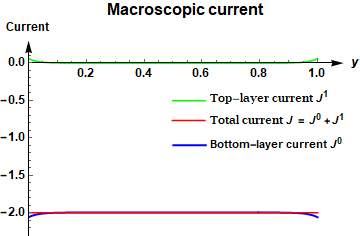} 
\end{subfigure}
\caption{{\small Macroscopic profiles of the densities for slow and fast particles (top panels), macroscopic profile of the total density (central panels), and the currents (bottom panels). Here,  $\rho_{(L,0)}=2,\,\rho_{(L,1)}=4,\,\rho_{(R,0)}=4$ and $\rho_{(R,1)}=2, \Upsilon=1$. For the panels in the left column, $\epsilon=1$ and for the panels in the right column, $\epsilon=0.001$.}}
\label{fig:density}
\end{figure}
\paragraph{Unequal boundary densities with uphill diffusion.}
As argued earlier, since the system does not always obey Fick's law, by tuning the parameters $\rho_{(L,0)},\rho_{(L,1)},\rho_{(R,0)},\rho_{(R,1)}$ and $\epsilon$, we can push the system into a regime where the total current is such that $J<0$ and the total densities are such that $\rho_R < \rho_L $, where $\rho_R = \rho_{(R,0)}+\rho_{(R,1)}$ and $\rho_L = \rho_{(L,0)}+\rho_{(L,1)}$. In this regime, {\em the current goes uphill}, since the total density of particles at the right is lower than at the left, yet the average current is negative.

For an illustration, consider the case when $\rho_{L,1} =6, \rho_{R,0}=4$ and $\rho_{L,0} =\rho_{R,1} = 2$, which implies $\rho_{L} = 8$ and  $\rho_{R} = 6$ and thus $\rho_{R} < \rho_{L}$. The density profiles and currents are shown in Fig.~\ref{fig:density2} for two values of $\epsilon$, in particular, a comparison between the Fick-regime $\epsilon = 1$ (left panels) and the non-Fick-regime with very slow particles $\epsilon = 0.001$ (right panels). As can be seen in the figure, when $\epsilon=1,$ the system obeys Fick's law: the total density linearly interpolates between the two total boundary densities 8 and 6, respectively. The average total stationary current is positive as predicted by Fick's law. However, in the \emph{uphill} regime, the total density is non-linear and the gradient of the total density is not proportional to the total current, violating Fick's law. The total current is negative and is effectively dominated by the current of the fast particles. It will be shown later that the transition into the uphill regime happens at the critical value $\epsilon = \tfrac{|\rho_{(R,0)}-\rho_{(L,0)}|}{|\rho_{(R,1)}-\rho_{(L,1)}|} = \tfrac{1}{2}$. In the limit $\epsilon\downarrow 0$ the total density profile and the current always get dominated in the bulk by the profile and the current of the fast particles, respectively. When $\epsilon<\tfrac{1}{2}$, even though the density of the slow particles makes the total density near the boundaries such that $\rho_R < \rho_L$, it is not strong enough to help the system overcome the domination of the fast particles in the bulk, and so the effective total current goes in the same direction as the current of the fast particles, producing an uphill current. 
\begin{figure}[htbp] 
\begin{subfigure}[b]{0.5\linewidth}
\centering
\includegraphics[width=0.75\linewidth]{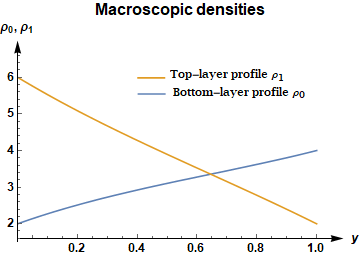} 
\vspace{4ex}
\end{subfigure}
\begin{subfigure}[b]{0.5\linewidth}
\centering
\includegraphics[width=0.75\linewidth]{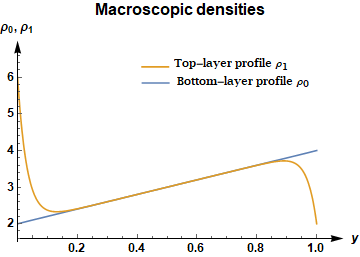} 
\vspace{4ex}
\end{subfigure} 
\begin{subfigure}[b]{0.5\linewidth}
\centering
\includegraphics[width=0.75\linewidth]{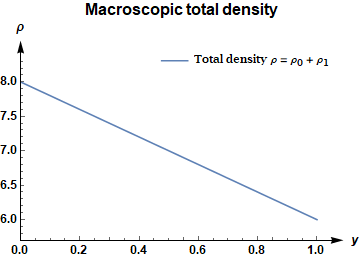} 
\vspace{4ex} 
\end{subfigure}
\begin{subfigure}[b]{0.5\linewidth}
\centering
\includegraphics[width=0.75\linewidth]{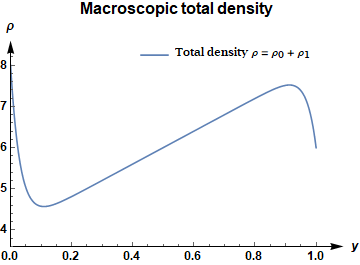} 
\label{fig:total_density_1_e0012642}
\vspace{4ex}
\end{subfigure}
\begin{subfigure}[b]{0.5\linewidth}
\centering
\includegraphics[width=0.75\linewidth]{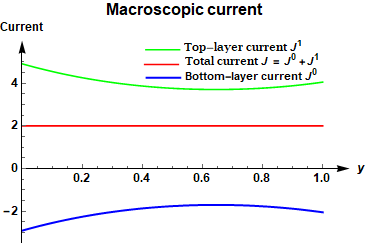} 
\end{subfigure}
\begin{subfigure}[b]{0.5\linewidth}
\centering
\includegraphics[width=0.75\linewidth]{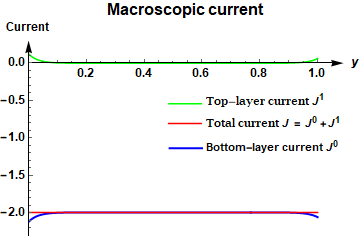} 
\label{fig:current_e0012642}
\end{subfigure}
\caption{\small Macroscopic profiles of the densities for slow and fast particles (top panels), macroscopic profile of the total density (central panels), and the currents (bottom panels). Here,  $\rho_{(L,0)}=2,\,\rho_{(L,1)}=6,\,\rho_{(R,0)}=4$ and $\rho_{(R,1)}=2, \Upsilon=1$. For the panels in the left column, $\epsilon=1$ and for the panels in the right column, $\epsilon=0.001$.}
\label{fig:density2}
\end{figure}

\paragraph{The transition between downhill and uphill.}

We observe that for the choice of reservoir parameters $\rho_{L,1} =6, \rho_{R,0}=4$ and $\rho_{L,0} =\rho_{R,1} = 2$, the change from downhill to uphill diffusion occurs at $\epsilon = \tfrac{|\rho_{(R,0)}-\rho_{(L,0)}|}{|\rho_{(R,1)}-\rho_{(L,1)}|} = \tfrac{1}{2}$. The density profiles and currents are shown in Fig.~\ref{fig:mild_density_2642} for two additional values of $\epsilon$, one in the ``mild'' downhill regime $J>0$ for $\epsilon = 0.75$ (left panels), the other in the ``mild'' uphill regime $J<0$ for $\epsilon = 0.25$ (right panels). In the uphill regime (right panel), i.e., when
$\epsilon=0.75$, the ``mild" non-linearity of the total density profile is already visible, indicating the violation of Fick's law.
\begin{figure}[htbp] 
\begin{subfigure}[b]{0.5\linewidth}
\centering
\includegraphics[width=0.75\linewidth]{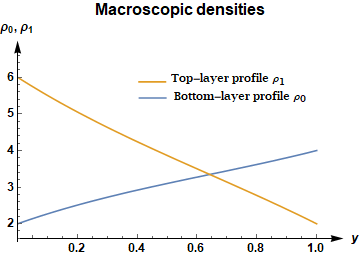}  
\vspace{4ex}
\end{subfigure}
\begin{subfigure}[b]{0.5\linewidth}
\centering
\includegraphics[width=0.75\linewidth]{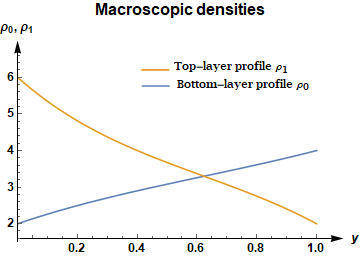}  
\vspace{4ex}
\end{subfigure} 
\begin{subfigure}[b]{0.5\linewidth}
\centering
\includegraphics[width=0.75\linewidth]{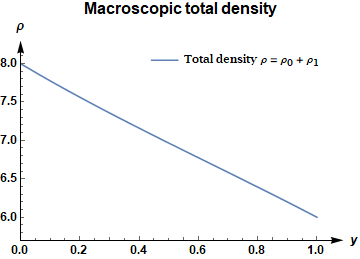} 
\vspace{4ex} 
\end{subfigure}
\begin{subfigure}[b]{0.5\linewidth}
\centering
\includegraphics[width=0.75\linewidth]{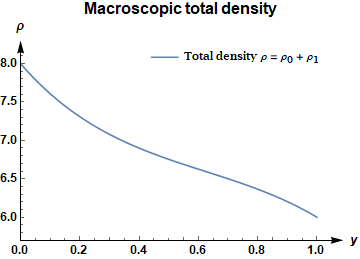} 
\vspace{4ex}
\end{subfigure}
\begin{subfigure}[b]{0.5\linewidth}
\centering
\includegraphics[width=0.75\linewidth]{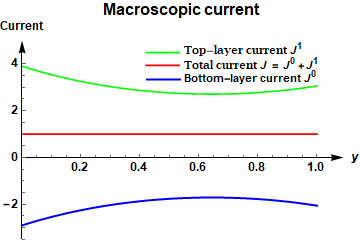} 
\end{subfigure}
\begin{subfigure}[b]{0.5\linewidth}
\centering
\includegraphics[width=0.75\linewidth]{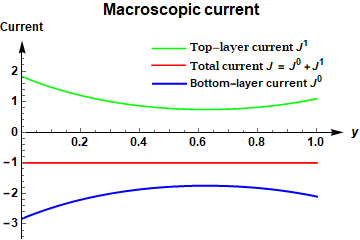} 
\end{subfigure}
\caption{\small Macroscopic profiles of the densities for slow and fast particles (top panels), macroscopic profile of the total density (central panels), and the currents (bottom panels) in the ``mild" downhill and the ``mild" uphill regime. Here,  $\rho_{(L,0)}=2,\,\rho_{(L,1)}=6,\,\rho_{(R,0)}=4$ and $\rho_{(R,1)}=2, \Upsilon=1$. For the panels in the left column, $\epsilon=0.75$ and for the panels in the right column, $\epsilon=0.25$.}
\label{fig:mild_density_2642}
\end{figure}

\paragraph{Identification of the uphill regime.}
We define the notion of uphill current below and identify the parameter ranges for which uphill diffusion occurs.
\begin{definition}{\bf [Uphill diffusion]}
\rm{For parameters $\rho_{(L,0)},\rho_{(L,1)},\rho_{(R,0)},\rho_{(R,1)}$ and $\epsilon>0,$ we say the system has an uphill current in stationarity if the total current $J$ and the difference between the total density of particles in the right and the left side of the system given by $\rho_R - \rho_L$ have the same sign, where it is understood that $\rho_R = \rho_{(R,0)}+\rho_{(R,1)}$ and $\rho_L = \rho_{(L,0)}+\rho_{(L,1)}$.}
\hfill$\spadesuit$
\end{definition}

\begin{proposition}{\bf [Uphill regime]}
Let $a_0 : = \rho_{(R,0)}-\rho_{(L,0)}$ and $a_1:= \rho_{(R,1)}-\rho_{(L,1)}$. Then the macroscopic system admits an uphill current in stationarity if and only if
\begin{equation}
\label{eq:manifold of uphill current}
a_0^2+(1+\epsilon)\,a_0a_1+\epsilon a_1^2 < 0.
\end{equation}
If, furthermore, $\epsilon\in [0,1]$, then
\begin{itemize}
\item[(i)] either 
$$a_0+a_1 > 0 \text{ with } a_0<0,\,a_1> 0$$ 
or 
$$a_0+a_1 < 0 \text{ with } 
a_0 > 0,\,a_1 < 0,$$
\item[(ii)] $\epsilon\in \big[0,-\tfrac{a_0}{a_1}\big]$.
\end{itemize} 
\end{proposition}

\begin{proof}
Note that, by \eqref{eq: average macroscopic current}, there is an uphill current if and only if $a_0 + a_1$ and $a_0 + \epsilon a_1$ have opposite signs. In other words, this happens if and only if
$$
(a_0+a_1)(a_0+\epsilon\,a_1) = a_0^2+(1+\epsilon)\,a_0a_1+\epsilon a_1^2<0.
$$
The above constraint forces $a_0a_1<0$. Further simplification reduces the parameter regime to the following four cases:
\begin{itemize}
\item $a_0+a_1 > 0$ with $a_0<0,\,a_1> 0$ and $\epsilon<-\tfrac{a_0}{a_1}$,
\item $a_0+a_1 < 0$ with $a_0>0,\,a_1< 0$ and $\epsilon<-\tfrac{a_0}{a_1}$,
\item $a_0+a_1 > 0$ with $a_0>0,\,a_1< 0$ and $\epsilon>-\tfrac{a_0}{a_1}$,
\item $a_0+a_1 < 0$ with $a_0<0,\,a_1> 0$ and $\epsilon>-\tfrac{a_0}{a_1}$.
\end{itemize}
Under the assumption $\epsilon\in[0,1]$, only the first two of the above four cases survive.
\end{proof}


\subsection{The width of the boundary layer}

We have seen that for $\epsilon=0$ the microscopic density profile of the fast particles $\theta_0(x)$ linearly interpolates between $\rho_{L,0}$ and $\rho_{R,0}$, whereas the density profile of the slow particles satisfies $\theta_1(x) = \theta_0(x)$ for all $x\in\{2,\ldots,N-1\}$. In the macroscopic setting this produces a continuous macroscopic profile $\rhosz(y) = \rho_{L,0} + (\rho_{R,0}-\rho_{L,0})y$ for the bottom-layer, while the top-layer profile develops two discontinuities at the boundaries when either $\rho_{(L,0)}\neq\rho_{(L,1)}$ or $\rho_{(R,0)}\neq \rho_{(R,1)}$. In particular,
$$
\rhoso(y)\to \big[\rho_{L,0} + (\rho_{R,0}-\rho_{L,0})y\big]\, \mathbf{1}_{(0,1)}(y) 
+ \rho_{L,1} \mathbf{1}_{\{1\}}(y) +  \rho_{R,1} \mathbf{1}_{\{0\}}(y),\quad\epsilon\downarrow 0,
$$
for $y\in [0,1]$. For small but positive $\epsilon,$ the curve is smooth and the discontinuity is turned into a boundary layer. In this section we investigate the width of the left and the right boundary layers as $\epsilon\downarrow 0$. To this end, let us define
\begin{equation}
	\label{eq:strength of discontinuity}
W_L:=|\rho_{(L,0)}-\rho_{(L,1)}|,\quad W_R:=|\rho_{(R,0)}-\rho_{(R,1)}|.
\end{equation}
Note that, the profile $\rho_1$ develops a left boundary layer if and only if $W_L > 0$ and, similarly, a right boundary layer if and only if  $W_R > 0$. 
\begin{definition}
\label{def:width of boundary layers}
\rm{We say that the \emph{left boundary layer} is of size $f_L(\epsilon)$ if there exists $C>0$ such that, for any $c>0,$ 
$$\lim_{\epsilon\downarrow 0} \frac{R_L(\epsilon, c)}{f_L(\epsilon)}=C,$$
where $R_L(\epsilon, c)=\sup\left \{  y\in\big(0,\tfrac{1}{2}\big): \left|\frac{d^2}{d y^2}\rhoso(y)\right| \ge c\right \}$. Analogously, we say that the \emph{right boundary layer} is of size $f_R(\epsilon)$ if there exists $C>0$ such that, for any $c>0,$ 
$$\lim_{\epsilon\downarrow 0} \frac{1-R_R(\epsilon, c)}{f_R(\epsilon)}=C,$$
where $R_R(\epsilon, c)=\inf\left \{  y\in\big(\tfrac{1}{2},1\big): \left|\frac{d^2}{d y^2}\rhoso(y)\right|\geq c\right \}$.}
\end{definition}
\noindent
The widths of the two boundary layers essentially measure the deviation of the top-layer density profile (and therefore also the total density profile) from the bulk linear profile corresponding to the case $\epsilon = 0$. In the following proposition we estimate the sizes of the two boundary layers.
\begin{proposition}{\bf[Width of boundary layers]}
The widths of the two boundary layers are given by
\begin{equation}
	f_L(\epsilon) = f_R(\epsilon) = \sqrt{\epsilon}\,\log(1/\epsilon),
\end{equation}
where $f_L(\epsilon),f_R(\epsilon)$ are defined as in Definition~\ref{def:width of boundary layers}.
\end{proposition}
\begin{proof}
	Note that, to compute $f_L(\epsilon)$, it suffices to keep $W_L>0$ fixed and put $W_R = 0$, where $W_L,W_R$ are as in \eqref{eq:strength of discontinuity}.
	Let $\overline{y}(\epsilon, c)\in(0,\tfrac{1}{2})$ be such that, for some constant $c>0$,
\begin{equation}
\left|\frac{d^2}{d y^2}\rhoso(y)\right| \geq c,
\end{equation}
or equivalently, since $\epsilon\Delta\rho_1 =\Upsilon(\rho_1-\rho_0)$,
\begin{equation}
|\rhoso(y)-\rhosz(y)|\geq\frac{c\epsilon}{\Upsilon}.
\end{equation}
Recalling the expressions of $\rhosz(\cdot)$ and $\rhoso(\cdot)$ for positive $\epsilon$ given in \eqref{eq:macroscopic profile bottom layer e>0}$-$\eqref{eq:macroscopic profile top layer e>0}, we get
\begin{equation}\label{eq:general equation for width}
\left|\frac{\sinh\big[\sqrt{\Upsilon(1+\tfrac{1}{\epsilon})}(1-y)\big]}{\sinh\big[\sqrt{\Upsilon(1+\tfrac{1}{\epsilon})}\big]}\,(\rho_{(L,0)}-\rho_{(L,1)})+\frac{\sinh\big[\sqrt{\Upsilon(1+\tfrac{1}{\epsilon})}\,y\big]}{\sinh\big[\sqrt{\Upsilon(1+\tfrac{1}{\epsilon})}\big]}\,(\rho_{(R,0)}-\rho_{(R,1)})\right|\geq\frac{c\epsilon}{\Upsilon}.
\end{equation}
Using \eqref{eq:strength of discontinuity} plus the fact that $W_R = 0,$ and setting $B_{\epsilon,\Upsilon}:=\sqrt{\Upsilon\left(1+\tfrac{1}{\epsilon}\right)}$, we see that
\begin{equation}
	\label{eq:left width of top layer current}
	\sinh\left[B_{\epsilon,\Upsilon}(1-y)\right]\geq\frac{c\epsilon }{\Upsilon W_L}\sinh\left[B_{\epsilon,\Upsilon}\right].
\end{equation}
Because $\sinh(\cdot)$ is strictly increasing, \eqref{eq:left width of top layer current} holds if and only if
\begin{equation}
\overline{y}(\epsilon, c) \leq 1 - \frac{1}{B_{\epsilon,\Upsilon}}\sinh^{-1}\left[\frac{c\epsilon}{\Upsilon W_L}\sinh
\left(\tfrac{B_{\epsilon,\Upsilon}}{2}\right)\right].
\end{equation}
Thus, for small $\epsilon >0$ we have
\begin{equation}
	R_L(\epsilon,c) = 1 - \frac{1}{B_{\epsilon,\Upsilon}}\sinh^{-1}\left[\frac{c\epsilon}{\Upsilon W_L}\sinh
	\left(\tfrac{B_{\epsilon,\Upsilon}}{2}\right)\right],
\end{equation}
where $R_L(\epsilon, c)$ is defined as in Definition~\ref{def:width of boundary layers}. Since $\sinh^{-1} x = \log(x+\sqrt{x^2+1})$ for $x\in\R$, we obtain
\begin{equation}
	\label{eq:solution of width of top layer current}
	\begin{aligned}
		R_L(\epsilon, c) &= \frac{\sqrt{\epsilon}}{\sqrt{\Upsilon(1+\epsilon)}}\log\left[\frac{N_{\epsilon,\Upsilon}+\sqrt{N_{\epsilon,\Upsilon}^2+1}}{\epsilon CN_{\epsilon,\Upsilon}+\sqrt{(\epsilon CN_{\epsilon,\Upsilon})^2+1}}\right]\\
		& = \frac{\sqrt{\epsilon}}{\sqrt{\Upsilon(1+\epsilon)}}\log(1/\epsilon)+\frac{\sqrt{\epsilon}}{\sqrt{\Upsilon(1+\epsilon)}}\log
	\left[\frac{1+\sqrt{1+(1/N_{\epsilon,\Upsilon})^2}}{C+\sqrt{C^2+(1/(\epsilon N_{\epsilon,\Upsilon}))^2}}\right]\\
	&=\frac{\sqrt{\epsilon}}{\sqrt{\Upsilon(1+\epsilon)}}\log(1/\epsilon) + R_{\epsilon,\Upsilon,W_L},
	\end{aligned}
\end{equation}
where $N_{\epsilon,\Upsilon}:=\sinh\Big(\tfrac{B_{\epsilon,\Upsilon}}{2}\Big), C:=\tfrac{c}{\Upsilon W_L}$, and the error term is
$$
R_{\epsilon,\Upsilon,W_L}:=\frac{\sqrt{\epsilon}}{\sqrt{\Upsilon(1+\epsilon)}}\log
\left[\frac{1+\sqrt{1+(1/N_{\epsilon,\Upsilon})^2}}{C+\sqrt{C^2+(1/(\epsilon N_{\epsilon,\Upsilon}))^2}}\right].
$$
Note that, since $\epsilon N_{\epsilon,\Upsilon}\to \infty$ as $\epsilon\downarrow 0$, we have
\begin{equation}
	\label{eq:error term in width}
	\lim\limits_{\epsilon\downarrow 0}\frac{R_{\epsilon,\Upsilon,W_L}}{\sqrt{\epsilon}} = \frac{1}{\sqrt{\Upsilon}}\log(1/C) < \infty.
\end{equation}
Hence, combining \eqref{eq:solution of width of top layer current}$-$\eqref{eq:error term in width}, we get 
\begin{equation}
	\label{eq:solution of left width}
	\lim\limits_{\epsilon\downarrow 0}\frac{R_L(\epsilon, c)}{\sqrt{\epsilon}\log(1/\epsilon)} = \lim\limits_{\epsilon\downarrow 0}\frac{1}{\sqrt{\Upsilon}(1+\epsilon)}+\lim\limits_{\epsilon\downarrow 0}\frac{R_{\epsilon,\Upsilon,W_L}}{\sqrt{\epsilon}\log(1/\epsilon)} = \frac{1}{\sqrt{\Upsilon}}
\end{equation}
and so, by Definition~\ref{def:width of boundary layers}, $f_L(\epsilon) = \sqrt{\epsilon}\log(1/\epsilon)$.

Similarly, to compute $f_R(\epsilon)$, we first fix $W_L = 0, W_R>0$ and note that, for some $c>0$, we have, by using \eqref{eq:general equation for width},
\begin{equation}
	\label{eq:right width of top layer current}
	|\partial^2\rhoso(y)|\geq c\quad\text{ if and only if }\quad\sinh\left[B_{\epsilon,\Upsilon}\,y\right]\geq\tfrac{c\epsilon }{\Upsilon W_R}\sinh\left[B_{\epsilon,\Upsilon}\right].
\end{equation}
Hence, by appealing to the strict monotonicity of $\sinh(\cdot),$ we obtain
\begin{equation}
	R_R(\epsilon,c)=\inf\left \{  y\in\big(\tfrac{1}{2},1\big): \left|\frac{d^2}{d y^2}\rhoso(y)\right|\geq c\right \} = \frac{1}{B_{\epsilon,\Upsilon}}\sinh^{-1}\left[\frac{c\epsilon}{\Upsilon W_R}\sinh
	\left(\tfrac{B_{\epsilon,\Upsilon}}{2}\right)\right]. 
\end{equation}
Finally, by similar computations as in \eqref{eq:solution of width of top layer current}--\eqref{eq:solution of left width}, we see that
\begin{equation}
	\lim\limits_{\epsilon\downarrow 0}\frac{1-R_R(\epsilon,c)}{\sqrt{\epsilon\log(1/\epsilon)}} = \frac{1}{\sqrt{\Upsilon}}
\end{equation} and hence $f_R(\epsilon)=\sqrt{\epsilon}\log(1/\epsilon)$. 
\end{proof}

\appendix


\section{Inverse of the boundary-layer matrix}
\label{s.app*}

The inverse of the matrix $M_\epsilon$ defined in \eqref{eq:matrix_def} is given by ($\alpha_1$ and $\alpha_2$ are as in \eqref{eq:definition of two roots})
\begin{equation}
\label{eq:inverse of M_epsilon}
M_\epsilon^{-1}:=
\frac{1}{Z}
\begin{bmatrix}
-m_{13} & -m_{14} & m_{13} & m_{14}\\
m_{21} & m_{22} & m_{23} & m_{24}\\
m_{31}(\alpha_2) & m_{32}(\alpha_2) & m_{33}(\alpha_2) & m_{34}(\alpha_2)\\
-m_{31}(\alpha_1) & -m_{32}(\alpha_1) & -m_{33}(\alpha_1) & -m_{34}(\alpha_1)
\end{bmatrix}
,
\end{equation}
where
\begin{equation}
\begin{aligned}
Z &:= \alpha_1^{N + 1} [\alpha_2 (1 - \epsilon) (\alpha_2^{N - 1} + 1) + 
2 \epsilon (\alpha_2^{N + 1} + 1)]\,[\alpha_2 (1 + N) (1 - \epsilon) (\alpha_2^{N - 1} - 
1) + 2 \epsilon (N + \epsilon) (\alpha_2^{1 + N} - 1)],\\
m_{13} &:=\alpha_1^{N + 1}[\alpha_2 (1 - \epsilon) (\alpha_2^{N - 1} + 1) + 
2 \epsilon (\alpha_2^{N + 1} + 1)]\,[\alpha_2 (1 - \epsilon) (\alpha_2^{N - 1} - 1) + \epsilon(\alpha_2^{N + 1} - 1)],\\
m_{14} &:=\epsilon\, \alpha_1^{
	N + 1}[\alpha_2 (1 - \epsilon) (\alpha_2^{N - 1} + 1) + 
2 \epsilon (\alpha_2^{N + 1} + 1)]\,(\alpha_2^{N+1} - 1),\\
m_{21} &:= (1 + N) (1 - \epsilon)^2 (\alpha_2^{N-1}-\alpha_1^{N-1})-\epsilon(1 - \epsilon)^2(\alpha_2 - 
\alpha_1)\\
&\qquad + \epsilon^2 (1 + 2 N + \epsilon)(\alpha_2^{N+1}-\alpha_1^{N+1}) 
+ \epsilon(1 - \epsilon)(2 + 3 N + \epsilon)(\alpha_2^N-\alpha_1^N),\\
m_{22} &:= \epsilon\,[(1 - \epsilon)(1 + N)(\alpha_2^N-\alpha_1^N) 
+ \epsilon (1 + 2 N + \epsilon)(\alpha_2^{N+1}-\alpha_1^{N+1})],\\
m_{23} &:= \epsilon\,(1 - \epsilon) [(N + \epsilon)(\alpha_2 - \alpha_1) 
- (1 - \epsilon)(\alpha_2^N - \alpha_1^N) - \epsilon (\alpha_2^{N+1} - \alpha_1^{N+1})],\\
m_{24}&:= -\epsilon (1 - \epsilon) [(1 + N) (\alpha_2 - \alpha_1) + \epsilon\,(\alpha_2^{N+1} - \alpha_1^{N+1})],
\end{aligned}
\end{equation}
and the polynomials $m_{31}(z),m_{32}(z),m_{33}(z),m_{34}(z)$ are defined as
\begin{equation}
\begin{aligned}
m_{31}(z) &:= -(1 -\epsilon)^2\,z - \epsilon\,(1 - \epsilon) + 
(1 - \epsilon) (N + \epsilon)\,z^N -
\epsilon (1 - 2 N - 3 \epsilon)\,z^{N+1},\\
m_{32}(z) &:= -(1 - \epsilon)(1 + N) z^N - \epsilon\,(1 - \epsilon) - 
\epsilon (1 + 2 N + \epsilon)\,z^{N+1},\\
m_{33}(z) &:= (1 -\epsilon)^2\,z^N + \epsilon\,(1 - \epsilon)\,z^{N+1} - 
(1 - \epsilon) (N + \epsilon)\,z +
\epsilon (1 - 2 N - 3 \epsilon),\\
m_{34}(z) &:= (1 + N)(1 - \epsilon)\,z + 
\epsilon\,(1 - \epsilon)\,z^{N+1} + \epsilon (1 + 2 N + \epsilon).
\end{aligned}
\end{equation}
We remark that most of the terms appearing in the inverse simplify because of \eqref{eq:product of two roots}. We define the four vectors $\vec{c}_1,\vec{c}_2,\vec{c}_3,\vec{c}_4$ as the respective rows of $M_\epsilon^{-1}$, i.e.,
\begin{equation}
\label{eq:definition of c1,c2,c3,c4}
\begin{aligned}
\vec{c}_1 &:= (M_\epsilon^{-1})^T \vec{e}_1,\quad
\vec{c}_2 := (M_\epsilon^{-1})^T \vec{e}_2,\\
\vec{c}_3 &:= (M_\epsilon^{-1})^T \vec{e}_3,\quad
\vec{c}_4 := (M_\epsilon^{-1})^T \vec{e}_4,
\end{aligned}
\end{equation}
where 
$$
\begin{aligned}
&\vec{e}_1 := \big[
\begin{array}{c c c c}
1 & 0 & 0 & 0
\end{array}
\big]^T,& 
&\vec{e}_{2} :=\big[
\begin{array}{c c c c}
0 & 1 & 0 & 0
\end{array}
\big]^T,\\
&\vec{e}_{3} := \big[
\begin{array}{c c c c}
0 & 0 & 1 & 0
\end{array}
\big]^T,& 
&\vec{e}_{4} := \big[
\begin{array}{c c c c}
0 & 0 & 0 & 1
\end{array}
\big]^T.
\end{aligned}
$$

\end{document}